\documentclass[a4paper,reqno, oneside]{amsart}
\usepackage[english]{babel}
\usepackage{amsmath, amssymb, amsthm, amscd}
\usepackage{enumerate}
\usepackage{palatino}
\usepackage{mathpazo}
\usepackage{paralist}
\usepackage[a4paper]{geometry}

\usepackage{tensor}

\usepackage[all]{xy}

\usepackage{subcaption} 
\usepackage{caption} 
\usepackage{floatrow}

\usepackage{appendix}

\usepackage{tikz} 
\usepackage{tikz-cd}

\usepackage{dashbox}

\usepackage{graphicx}

\makeatletter
\newcommand{\ostar}{\mathbin{\mathpalette\make@circled\ast}}
\newcommand{\make@circled}[2]{%
  \ooalign{$\m@th#1\smallbigcirc{#1}$\cr\hidewidth$\m@th#1#2$\hidewidth\cr}%
}
\newcommand{\smallbigcirc}[1]{%
  \vcenter{\hbox{\scalebox{0.77778}{$\m@th#1\bigcirc$}}}%
}
\makeatother

\renewcommand{\setminus}{\smallsetminus}

\let\oldtocsection=\tocsection

\let\oldtocsubsection=\tocsubsection

\let\oldtocsubsubsection=\tocsubsubsection

\renewcommand{\tocsection}[2]{\hspace{0em}\oldtocsection{#1}{#2}}
\renewcommand{\tocsubsection}[2]{\hspace{1em}\oldtocsubsection{#1}{#2}}
\renewcommand{\tocsubsubsection}[2]{\hspace{2em}\oldtocsubsubsection{#1}{#2}}

\theoremstyle{plain}
\newtheorem{theorem}{Theorem}[section]
\newtheorem{prop}[theorem]{Proposition}
\newtheorem{lemma}[theorem]{Lemma}
\newtheorem{cor}[theorem]{Corollary}  
\newtheorem*{claim}{Claim}

\newtheorem*{TheoremA}{Theorem A}
\newtheorem*{TheoremB}{Theorem B}
\newtheorem*{TheoremC}{Theorem C}
\newtheorem*{TheoremD}{Theorem D}

\theoremstyle{definition}
\newtheorem{definition}[theorem]{Definition}

\theoremstyle{remark}
\newtheorem{remark}[theorem]{Remark}
\newtheorem{example}[theorem]{Example}

\numberwithin{equation}{section}

\begin{document}
\setlength{\parindent}{0.cm}

\title[Transport functions and DG Morse homology]{Transport functions for principal bundles and Morse homology with differential graded coefficients}

\author{Maximilian Stegemeyer}
\address{Mathematisches Institut, Universit\"at Freiburg, Ernst-Zermelo-Straße 1, 79104 Freiburg, Germany}
\email{maximilian.stegemeyer@math.uni-freiburg.de}
\date{\today}


\begin{abstract}
   We study transport functions as a Morse-theoretical way of describing principal bundles. Transport functions are maps from the spaces of broken gradient flow lines to a topological group and they encode the transition functions of the principal bundle.
   We describe and extend a construction by Voigt that yields such transport functions and show that one can recover the principal bundle from the transport function.
   Using transport functions with values in a topological group $G$ and a differential graded module over the chains of $G$ we define a chain complex in the style of Barraud-Damian-Humilière-Oancea's Morse homology with differential graded coefficients.
   We prove that in many cases the homology of this complex is the homology of an associated bundle.
   In the case of smooth bundles transport functions arise also from parallel transport with respect to a connection and the corresponding DG Morse complex turns out to be isomorphic to a complex defined in the style of Barraud-Damian-Humilière-Oancea. We eventually consider certain aspects of the functoriality of our constructions.
\end{abstract}
\maketitle

\setcounter{tocdepth}{2}
\tableofcontents

\section{Introduction}

Morse theory plays a central role in geometry and topology when relating geometric and topological properties of smooth manifolds. 
Historically speaking, among the key features are the Morse inequalities relating the Betti numbers of the manifold to the numbers of critical points and the handle decomposition induced by a Morse function.
Furthermore, a Morse function $f\colon M\to \mathbb{R}$ induces a chain complex whose homology - the \emph{Morse homology} - is isomorphic to the singular homology of $M$.
The differential of the Morse complex is defined by a count of the number of gradient flow lines between critical points of index difference $1$.

In their seminal paper \cite{cohen1995morse} Cohen, Jones and Segal argue that in many situations, the collection of the moduli spaces of gradient flow lines of arbitrary index difference captures captures not only the homology but the homotopy type of the manifold $M$ and in certain situations even the homeomorphism type.
The main tool in their description is the \emph{flow category} induced by a Morse function.
The flow category is a topological category $\mathcal{M}_f$ with objects being the critical points of the Morse function $f\colon M\to \mathbb{R}$ and with morphism spaces the compactifications of the moduli spaces $\overline{\mathcal{L}}(x,y)$ for critical points $x,y\in \mathrm{Crit}(f)$.
Cohen, Jones and Segal argue in \cite{cohen1995morse} that the \emph{classifying space} of this category is homotopy equivalent or even homeomorphic to $M$. 
A crucial aspect in the proof of this theorem is the structure of the compactified moduli spaces and the associativity of gluing, see e.g. \cite{wehrheim2012smooth} and \cite{qin2018associativity}. 
We also refer to \cite{calle2026classifying} and references therein.
It is therefore an important question how the topological features of the manifold are mirrored in Morse-theoretical constructions using the flow category and how information about invariants of $M$ can be extracted from the higher-dimensional moduli spaces of gradient flow lines. 

Furthermore, constructions in Morse homology often serve as a blueprint for developments in Floer theory.
In Floer homotopy theory flow categories in the spirit of Cohen, Jones and Segal and related constructions are the central objects of study and it is a fundamental question what information these flow categories carry.
We refer to  \cite{abouzaid2021arnold}, \cite{abouzaid2024foundation}, \cite{asplund2024spectral}, \cite{fourel2026morse} and \cite{lipshitz2014khovanov}, just to name a few references of many.

Considering Morse homology on closed manifolds one new tool is the \emph{Morse homology with differential graded coefficients} by Barraud, Damian, Humilière and Oancea in \cite{barraud2025morse}.
We briefly describe the setup of \cite{barraud2025morse} which uses a Morse function $f\colon M\to \mathbb{R}$ together with some additional data. 
Consider the Moore loop space $\Omega M$ with respect to a chosen basepoint $x_0\in M$.
The strictly associative concatenation of loops makes the cubical chains $\mathrm{C}_{\bullet}(\Omega M)$ into a differential graded algebra.
For a differential graded right-module $\mathcal{F}_{\bullet}$ over $\mathrm{C}_{\bullet}(\Omega M)$ define the DG Morse complex as
$$   \mathrm{C}_{\bullet}(M;\mathcal{F}_{\bullet}) :=  \mathcal{F}_{\bullet}\otimes \mathbb{Z}\langle \mathrm{Crit}(f)\rangle .     $$
The differential of this complex uses a collection of chains $\{m_{x,y}\in \mathrm{C}_{\bullet}(\Omega M)\}_{x,y\in\mathrm{Crit}(f)}$ which yield a \emph{twisting cocycle}.
Notice that the differential therefore incorporates information about pairs of critical points with arbitrary index difference and thus makes use of the higher-dimen-sional moduli spaces.
The typical source of DG left-modules $\mathcal{F}_{\bullet}$ comes from fibrations $E\to M$ over $M$.
Indeed, a transitive lifting function induces an $\Omega M$-action on the typical fiber $F = E_{x_0}$ which in turn induces the module structure.
Barraud, Damian, Humilière and Oancea show that the homology of the DG Morse complex $\mathrm{C}_{\bullet}(M;\mathrm{C}_{\bullet}(F))$ is isomorphic to the singular homology of the total space $E$, see \cite[Theorem 7.2]{barraud2025morse}.
This construction has found applications in string topology, see \cite{clivio2025goresky}, \cite{riegel2024chain} and \cite{riegel2025path}.
Using similar ideas and constructions, Barraud, Damian, Humilière and Oancea also study Floer homology with DG coefficients and provide applications to symplectic geometry, see \cite{barraud2024floer}.
The construction of DG Morse homology and DG Floer homology is based on a construction by Barraud and Cornea in \cite{barraud2007lagrangian}.

In this article we consider Morse homology with coefficients in a differential graded module $\mathcal{F}_{\bullet}$ over the differential graded algebra $\mathrm{C}_{\bullet}(G)$ where $G$ is a topological group.
The group multiplication turns the cubical chains $\mathrm{C}_{\bullet}(G)$ into a differential graded algebra.
In order to obtain a twisting cocycle $\{m_{x,y}\}_{x,y\in\mathrm{Crit}(f)}$  in $\mathrm{C}_{\bullet}(G)$ we consider \emph{transport functions}.
A $G$-valued transport function is a continuous functor from the flow category $\mathcal{M}_f$ induced by the Morse function $f$ to the topological category with one object and with morphism space $G$.
Thus a transport function is a collection of maps $\Phi_{x,y}\colon \overline{\mathcal{L}}(x,y)\to G$ with $\overline{\mathcal{L}}(x,y)$ the compactification of the moduli space of flow lines between the critical points $x,y\in\mathrm{Crit}(f)$ such that
\begin{equation}\label{eq_transport_conditionA}
    \Phi_{x,y}(u_1\circ u_2) = \Phi_{x,z}(u_1) \cdot \Phi_{z,y}(u_2)    
\end{equation}
for $u_1\in \overline{\mathcal{L}}(x,z)$ and $u_2\in\overline{\mathcal{L}}(z,y)$. 
As we shall see, transport functions with values in $G$ arise from $G$-principal bundles $q\colon E\to M$.
By certain choices of local trivializations of $E$ we construct a transport function $\Phi_q\colon \mathrm{mor}(\mathcal{M}_f) \to G$.
Conversely, a transport function $\Phi$ induces a $G$-valued cocycle which gives rise to a $G$-principal bundle $q_{\Phi}\colon \mathcal{E}\to M$.

\begin{TheoremA}[Theorem \ref{theorem_main_result_transport-functions} and Theorem \ref{theorem_isotopic_tr_functions_yield_isomorphic_bundles}]
	Let $(M,f,X)$ be Morse data and let $G$ be a metrizable topological group.
    \begin{enumerate}
        \item     A transport function $\{\Phi_{x,y}\colon \overline{\mathcal{L}}(x,y)\to G\}$ induces a $G$-principal bundle $q_{\Phi}\colon \mathcal{E}\to M$.
    Conversely, a $G$-principal bundle $\pi\colon E'\to M$ induces a transport function $\{\Phi_{\pi}\colon \overline{\mathcal{L}}(x,y)\to G\}$ such that the following holds:
    if $q\colon E\to M$ is a $G$-principal bundle, then the induced bundle $q_{\Phi_q}\colon \mathcal{E}\to M$ is isomorphic to $q$.
    \item If $G$ is path-connected, then two transport functions $\{\Phi_{x,y}\},\{\Psi_{x,y}\}$ induce isomorphic $G$-principal bundles $q_{\Phi}$ and $ q_{\Psi}$ if and only if $\{\Phi_{x,y}\}$ and $\{\Psi_{x,y}\}$ are isotopic transport functions.  
    Hence, there is a bijection between isotopy classes of $G$-valued transport functions on $\mathcal{M}_f$ and isomorphism classes of $G$-principal bundles over $M$.
    \end{enumerate}
\end{TheoremA}
The construction of the first part was sketched by Voigt in \cite{voigt2014diss}.
We expand Voigt's construction and fill in many details in this article.
By an \emph{isotopy of transport function} between $\{\Phi_{x,y}\}$ and $\{\Psi_{x,y}\}$ we mean a collection of maps $\mathcal{H}_{x,y}\colon \overline{\mathcal{L}}(x,y)\times [0,1]\to G$ with $\mathcal{H}_{x,y}(\cdot ,0) = \Phi_{x,y}$, $\mathcal{H}_{x,y}(\cdot, 1) = \Psi_{x,y}$ and such that $\mathcal{H}_{x,y}(\cdot,\tau)$ satisfies condition \eqref{eq_transport_conditionA} for each $\tau\in [0,1]$.

Using a $G$-valued transport function $\Phi$ induced by a principal bundle $q\colon E\to M$ and a DG-module $\mathcal{F}_{\bullet}$ over $\mathrm{C}_{\bullet}(G)$ we define a DG Morse complex $\mathrm{C}_{\bullet}(M;(\Phi,\mathcal{F}_{\bullet}))$.
The twisting cocycle $\{m_{x,y}\in \mathrm{C}_{\bullet}(G)\}$ is obtained by applying the induced map of $\Phi_{x,y}\colon \overline{\mathcal{L}}(x,y)\to G$ to a chain $s_{x,y}\in \mathrm{C}_{\bullet}(\overline{\mathcal{L}}(x,y))$ representing the fundamental class of $\overline{\mathcal{L}}(x,y)$.
We show that isotopic transport functions $\Phi$ and $\Psi$ induce quasi-isomorphic chain complexes in Theorem \ref{theorem_isotopic_tr_induces_quasi-iso}.

The main source of $\mathrm{C}_{\bullet}(G)$-modules comes from group actions $\psi\colon F\times G\to F$.
In this situation one can construct the \emph{associated bundle} $F\times_G E\to M$ which is a fiber bundle over $M$ with typical fiber $F$.
Note that the unstable manifolds $W^u(x)$, $x\in\mathrm{Crit}(f)$ admit compactifications $\overline{W}^u(x) \cong \mathbb{D}^{|x|}$ and a map $g_x\colon \overline{W}^u(x)\to M$.
We consider sections of the pullback bundle $g_x^* E\to \overline{W}^u(x)$ which satisfy a boundary condition on $\partial \overline{W}^u(x)$ and which we call the \emph{compatible sections}.
In case that the compatible sections exist for all $x\in\mathrm{Crit}(f)$ we show the following.

\begin{TheoremB}[Theorem \ref{theorem_assoc_bundle}]
    Let $(M,f,X,o,\{s_{x,y}\})$ be Morse data and let $q\colon E\to M$ be a $G$-principal bundle.
    Let $\Phi\colon \mathrm{mor}(\mathcal{M}_f)\to G$ be a transport function for $q\colon E\to M$ and assume that compatible sections exist for all critical points $x\in \mathrm{Crit}(f)$.    
    Furthermore, let $F$ be a right $G$-space.
    Then there is a quasi-isomorphism $t\colon \mathrm{C}_{\bullet}(M;(\Phi,\mathrm{C}_{\bullet}(F)))\to \mathrm{C}_{\bullet}(F\times_G E)$.
\end{TheoremB}

The compatible sections exist in many situations.
In particular if $\Phi\colon \mathrm{mor}(\mathcal{M}_f)\to G$ is a transport function for a smooth $G$-principal bundle with $G$ a Lie group, then compatible sections exist, see Corollary \ref{cor_lie_group_all_comp_sections_exist}.

In the case of a smooth $G$-principal bundle $q\colon E\to M$ with a given principal connection $\omega$ on $E$ we can use the holonomy with respect to $\omega$ to define a transport function $\Phi^{\omega}\colon \mathrm{mor}(\mathcal{M}_f)\to G$.
This map is essentially given by actual parallel transport along the flow lines.
It is a natural question how this geometrically defined transport function can be compared to the transport functions from Section \ref{sec_transport_function}.
We show the following.

\begin{TheoremC}[Theorem \ref{theorem_parallel_transport_induces_og_bundle} and Corollary \ref{cor_equivalence_to_parallel_transport_function}]
    Let $(M,f,X)$ be Morse data and let $G$ be a Lie group.
    Assume that $q\colon E\to M$ is a smooth $G$-principal bundle with connection $\omega$.
    Then the principal $G$-bundle induced by the transport function $\Phi^{\omega}\colon \mathrm{mor}(\mathcal{M}_f)\to G$ defined through parallel transport is isomorphic to the bundle $q\colon E\to M$.
    Hence, any transport function $\Phi\colon \mathrm{mor}(\mathcal{M}_f)\to G$ which induces the isomorphism class of $q$ is equivalent to $\Phi^{\omega}$ as a transport function.
\end{TheoremC}
The notion of equivalence of transport functions which we use here is explained in Section \ref{subsec_equivalence}. In case that $G$ is connected this equivalence relation is the same as the isotopy relation which was introduced above.

Next, we want to compare the DG Morse complex with coefficients in a $\mathrm{C}_{\bullet}(G)$-module to the DG Morse complexes studied in \cite{barraud2025morse}.
In \cite{barraud2025morse} the twisting cocycle $m_{x,y}\in \mathrm{C}_{\bullet}(\Omega M)$ is defined by collapsing a tree $\mathcal{Y}$ in $M$ and therefore certain paths occurring in the construction are not guaranteed to be smooth.
Instead of a twisting cocycle in the DGA $\mathrm{C}_{\bullet}(\Omega M)$ we consider the closely related DG category with objects $\mathrm{Crit}(f)$ and morphism spaces the chain groups $\mathrm{C}_{\bullet}(P'_{x\to y} M)$ for $x,y\in\mathrm{Crit}(f)$ where $P'_{x\to y}M$ is the space of piecewise smooth paths between $x$ and $y$ in $M$.
Let $q\colon E\to M$ be a fibration with transitive lifting function $\Xi$.
We consider a chain complex $       \mathrm{C}_{\bullet}'(M; E) := \bigoplus_{x\in\mathrm{Crit}(f)} \mathrm{C}_{\bullet}(E_x)\otimes R\langle x\rangle .  $
The differential is defined as in \cite{barraud2025morse} but the twisting cocycle lives in the morphism spaces $\mathrm{C}_{\bullet}(P'_{x\to y}M)$ instead of $\mathrm{C}_{\bullet}(\Omega M)$.
If $q\colon E\to M$ is a smooth $G$-principal bundle and $F$ is a right $G$-space then the fibration $F\times_G E\to M$ admits a transitive lifting function induced by parallel transport in $E$.
We can therefore compare the chain complex $\mathrm{C}_{\bullet}'(M;F\times_G E)$ with the complex $\mathrm{C}_{\bullet}(M;(\Phi^{\omega},\mathrm{C}_{\bullet}(F)))$ and show the following.
\begin{TheoremD}[Theorem \ref{theorem_isomorphic_complexes_connection}]
    Let $(M,f,X,o,\{s_{x,y}\})$ be Morse data and let $q\colon E\to M$ be a smooth $G$-principal bundle with principal connection $\omega$.
    Let $F$ be a right $G$-space and let $\mathcal{E} = F\times_G E\to M$ be the associated bundle.
    The chain complexes $\mathrm{C}_{\bullet}(M;(\Phi^{\omega},\mathrm{C}_{\bullet}(F)))$ and $\mathrm{C}_{\bullet}'(M;\mathcal{E})$ are isomorphic.
\end{TheoremD}
Eventually, we turn to functoriality aspects of our constructions with a particular eye on shriek maps in the fiber of the DG Morse complexes $\mathrm{C}_{\bullet}(M;(\Phi,\mathrm{C}_{\bullet}(F)))$.
We show in Theorem \ref{theorem_gysin_map_chain_map} and Proposition \ref{prop_invariant_thom_cc_exists_for_Gamma} that in certain situations one can find a DG Morse chain model for shriek maps which enables the computation of the shriek map in homology via a spectral sequence, see Corollary \ref{cor_spectral_sequence}.
We conclude by discussing the delicate topic of pullbacks of transport functions and their relevance for direct maps in the base in the DG Morse complexes.

\medskip

This article is organized as follows.
In Section \ref{sec_preliminaries} we collect some facts about Morse theory as well as about principal and associated bundles. We also review the construction of Morse homology with differential graded coefficients by \cite{barraud2025morse}.
Section \ref{sec_transport_function} is devoted to the proof of Theorem A, i.e. the Morse-theoretic description of principal bundles through transport functions.
In Section \ref{sec_4} we introduce the DG Morse complex using transport functions and consider some examples and particular cases.
Theorem B is proved in Section \ref{sec_assoc_bundles} where we also discuss the existence of the compatible sections.
In Section \ref{sec_comparison} we consider smooth principal bundles as well as the transport functions induced by parallel transport. We prove Theorem C and compare the DG Morse complex which uses the transport function to a DG Morse complex in the style of \cite{barraud2025morse} which leads to the proof of Theorem D.
Finally, in Section \ref{sec_functoriality} we consider some aspects of the functoriality of our constructions.

\medskip

\noindent \textbf{Acknowledgments:} 
The author wants to thank Jonathan Clivio, Robin Riegel, Misha Temkin and Roland Voigt for helpful discussions and suggestions regarding this manuscript.

\medskip

\section{Preliminaries}\label{sec_preliminaries}

\subsection{Morse theory and flow categories}\label{subsec_morse_flow}

We recall the Morse theoretical concepts that we shall need in this paper.
We refer to \cite{audin2014morse} for a thorough introduction to Morse theory.
Let $M$ be a closed manifold and let $f\colon M\to \mathbb{R}$ be a Morse function with negative pseudo-gradient vector field $X$.
This means that the stable and unstable manifolds intersect transversely.
We call the triple $(M,f,X)$ \emph{Morse data} on $M$.
We shall write $\mathrm{Crit}(f)$ for the set of all critical points which we consider as a graded set and we shall sometimes write $\mathrm{Crit}_j(f)$ for the critical points of index $j\in \mathbb{N}_0$.
We shall write $|x| \in \mathbb{N}_0$ for the Morse index of a critical point.
For critical points $a,b\in \mathrm{Crit}(f)$ with $|b|<|a|$ we denote the moduli space of gradient flow lines by 
$$     \mathcal{L} (a,b) = W^u (a)\cap W^s(b) / \sim    $$
where $W^u(a)$ is the \emph{unstable manifold} of $a$, $W^s(b)$ is the \emph{stable manifold} of $b$ and the equivalence relation is induced by the free $\mathbb{R}$-action on the parametrized space of unbroken flow lines $W^u (a)\cap W^s(b)$.
We further note that $\mathcal{L}(a,b)$ admits a compactification 
$$    \overline{\mathcal{L}}(a,b) = \mathcal{L}(a,b) \,\cup \,   \bigcup_{\substack{c_1,\ldots, c_{\ell} \in \mathrm{Crit}(f)\\ |b| < |c_{\ell}|<\ldots < |c_1|  <|a|}}   \mathcal{L}(a,c_1)\times \mathcal{L}(c_1,c_2)\times \ldots \times \mathcal{L}(c_{\ell},b) ,   $$
see e.g. \cite{latour1994existence} and \cite{qin2010moduli}.
The space $\overline{\mathcal{L}}(a,b)$ can be endowed with the structure of a manifold with corners.
In this manuscript we will refer to elements of the compactification as \emph{flow lines} and if we want to stress that a flow line is unbroken we shall call it \emph{unbroken flow line}.

Note that Morse data $(M,f,X)$ can be used to define a category $\mathcal{M}_f$ with objects $\mathrm{ob}(\mathcal{M}_f) = \mathrm{Crit}(f)$ the set of critical points and the morphism spaces being $\mathrm{mor}_{\mathcal{M}_f} (a,b) = \overline{\mathcal{L}}(a,b)$ for $a,b\in\mathrm{Crit}(f)$ with $|b|<|a|$.
This category is a topological category.
We say that a (broken) flow line $u = u_1\circ u_2\circ \ldots \circ u_{\ell}\in \overline{\mathcal{L}}(x,y)$ is \emph{maximally decomposed} if none of the flow lines $u_i$, $i\in\{1,\ldots,\ell\}$ are the constant flow line and if every flow line $u_i$ is unbroken for $i\in\{1,\ldots, \ell\}$.
Hence, without adding identity morphisms a maximally decomposed flow line $u$ cannot be decomposed further.
 
We shall now see that the homeomorphism type of the manifold $M$ can be recovered by gluing together the compactified moduli spaces.
Let $(M,f,X)$ be Morse data.
We reparametrize the flow lines as follows.
Let $u\in \overline{\mathcal{L}}(x,y)$ with $|x| \geq |y|$.
We define $\gamma_u\colon [f(y),f(x)]\to M$ by
$$    \gamma_u(t) = \overline{u} \cap f^{-1}( \{  t \})\quad \text{for}\,\,\, t\in [f(y),f(x)]    $$
where $\overline{u}$ denotes the closure of $\mathrm{im}(u)$.
Note that we have $f\circ \gamma_u(t) = t$ and thus the values of $f$ increase along $\gamma_u$, i.e. we run in opposite direction to the negative gradient flow.
Also note that for $x = y$ and the trivial flow line $u_x\in \mathcal{L}(x,x)$ we have $\gamma_{u_x}\colon \{f(x)\}\to M, f(x)\mapsto x$.
The advantage of this parametrization is that for $u_1\in \overline{\mathcal{L}}(x,y)$ and $u_2\in \overline{\mathcal{L}}(y,z)$ with $|x| > |y| > |z|$ we have $\gamma_{u_1}\colon [f(y),f(x)]\to M$ and $\gamma_{u_2}\colon [f(z),f(y)]\to M$ and therefore $\gamma_{u_1\circ u_2}$ satisfies
$$   \gamma_{u_1\circ u_2}\colon [f(z),f(x)]\to M,\quad \gamma_{u_1\circ u_2}(t) = \begin{cases}
    \gamma_{u_2}(t) , \quad f(z) \leq t \leq f(y) \\
    \gamma_{u_1}(t),  \quad f(y)\leq t \leq f(x) .
\end{cases}     $$

Define $N_0$ to be the discrete space
$$        N_0 =   \bigsqcup_{x\in\mathrm{Crit}(f)}  \{ (x, f(x)) \in \mathcal{L}(x,x) \times \mathbb{R}\}   .  $$
For $k\in\mathbb{N}$, $k\geq 1$ let
$$     N_k =  \bigsqcup_{\substack{x,y \in \mathrm{Crit}(f)\\ |x|-|y| = k }}  \overline{\mathcal{L}}(x,y) \times [f(y),f(x)]       $$
where $\overline{\mathcal{L}}(x,y)$ is equipped with the topology which makes it into a manifold with corners as discussed above.
We define an equivalence relation on the disjoint union 
$     \widetilde{K} = N_0\sqcup N_1 \sqcup \ldots \sqcup N_n   $
where $n = \mathrm{dim}(M)$.
Let $\sim$ be the equivalence relation generated by
$$     (u,f(x) )  \sim (x,f(x)) \quad \text{and}\quad (u,f(y)) \sim (y,f(y)) $$ for $u\in\overline{\mathcal{L}}(x,y)       $
as well as 
$   (u,t) \sim (u_i,t)$ for maximally decomposed $$u = u_1\circ u_2\circ \ldots \circ u_{\ell} \in \mathcal{L}(x,c_1)\times \mathcal{L}(c_1,c_2) \times \ldots \times \mathcal{L}(c_{{\ell-1}},y) \subseteq \overline{\mathcal{L}}(x,y)  $$ with $i\in \{1,\ldots, \ell\}$ and $ t\in [f(c_{i+1},f(c_i)] $.
Consider the continuous surjective map 
$$   \pi\colon \widetilde{K} \to M ,\quad (u,t) \mapsto  \gamma_u(t)  .    $$
One checks that $\pi(u,t) = \pi(v,s)$ if and only if $(u,t)\sim (v,s)$ for $(u,t),(v,s)\in\widetilde{K}$.
Therefore, $\pi$ induces a homeomorphism $K:=\widetilde{K} \,/\sim \xrightarrow[]{\cong} M$.
Moreover, for $i= 0,\ldots, n$ define spaces $\widetilde{K}_i = N_0\sqcup N_1\sqcup \ldots \sqcup N_i$.
We define $K_i = \pi(\widetilde{K}_i)$ and obtain a filtration $ \mathrm{Crit}(f) \cong K_0\subseteq K_1 \subseteq K_2\subseteq \ldots \subseteq K_n \cong M$.

\subsection{Morse homology with differential graded coefficients}\label{subsec_prelim_morse_DG}

In this subsection we summarize the construction of DG Morse homology in \cite{barraud2025morse}.
Let $R$ be a commutative unital ring.
We shall take cubical chains with coefficients in $R$ and we shall write $\mathrm{C}_{\bullet}(X)$ for the cubical chain complex $\mathrm{C}_{\bullet}^{\mathrm{cubical}}(X;R)$.
Recall that cubical chain complex is defined by quotienting out the degenerate cubes.

Let $(M,f,X)$ be Morse data as in the previous subsection.
As in \cite{barraud2025morse} we choose orientations $o_x$ of the unstable manifolds $W^u(x)$ and denote this collection of orientations by $o = \{o_x\}_{x\in\mathrm{Crit}(f)}$.  
We further choose a \emph{representing chain system} which is a collection of cubical chains $s_{x,y}\in \mathrm{C}_{|x|-|y|-1}(\overline{\mathcal{L}}(x,y))$ for $|y| < |x|$ such that $s_{x,y}$ represents the fundamental class of $(\overline{\mathcal{L}}(x,y),\partial\overline{\mathcal{L}}(x,y))$ and such that
$$      \partial s_{x,y}  =  \sum_{\substack{z \in \mathrm{Crit}(f)\\ |y| < |z|<|x|}}  (-1)^{|x|-|z|} s_{x,z}\times s_{z,y} .  $$
Barraud, Damian, Humilière and Oancea show that a representing chain system exists, see \cite[Proposition 5.6]{barraud2025morse} and that it is unique up to certain data, see \cite[Proposition 5.8]{barraud2025morse}.
The collection $(M,f,X,o,\{s_{x,y}\})$ will be called \emph{enhanced Morse data}.

\begin{definition}
    Let $\mathbf{R}_{\bullet}$ be a differential graded algebra over $R$ and let $(M,f,X)$ be Morse data.
     A collection of chains $\{m_{x,y}\in \mathbf{R}_{\bullet}\}_{x,y\in\mathrm{Crit}(f),|y|<|x|}$ is called a \emph{twisting cocycle} if for all $x,y\in\mathrm{Crit}(f)$ with $|y|<|x|$ we have
    \begin{equation}
       \label{eq_twisting_cocycle}     \partial m_{x,y} =     \sum_{\substack{z \in \mathrm{Crit}(f)\\ |y| < |z|<|x|}} (-1)^{|x|-|z|}  m_{x,z} \cdot m_{z,y}  .  
    \end{equation}
\end{definition}
Below we shall see that the representing chain system $\{s_{x,y}\}$ and certain maps out of the moduli spaces $\overline{\mathcal{L}}(x,y)$ will give rise to twisting cocycles.

Assume that we are given a twisting cocycle $\{m_{x,y}\}$ and that $\mathcal{F}_{\bullet}$ is a differential graded $\mathbf{R}_{\bullet}$ right-module via a map
$   \Psi\colon\mathcal{F}_{\bullet} \otimes \mathbf{R}_{\bullet} \to \mathcal{F}_{\bullet}     $.
We define the Morse complex with coefficients in $\mathcal{F}$ to be the family of $R$-modules
$$      \mathrm{C}_{\bullet}(M; \mathcal{F}_{\bullet}) :=    \mathcal{F}_{\bullet}  \otimes  R\langle \mathrm{Crit}_{\bullet}(f) \rangle      $$
together with the differential
$$    \partial ( \alpha \otimes x) :=   (\partial \alpha)\otimes x  + (-1)^{|\alpha|}  \sum_{\substack{y \in \mathrm{Crit}(f)\\ |y| < |x|}}    \Psi(\alpha\otimes m_{x,y}) \otimes y  $$
for $x\in\mathrm{Crit}(f)$ and $\alpha\in\mathcal{F}_{\bullet}$.
Using equation \eqref{eq_twisting_cocycle} one checks that this is indeed a differential.
We mention the \emph{canonical filtration} of the twisted complex, see \cite[Section 4.2]{barraud2025morse}.
The complex $\mathrm{C}_{\bullet} := \mathrm{C}_{\bullet}(M;\mathcal{F}_{\bullet})$ is filtered by
$$   F_p (\mathrm{C}_{\bullet}) := \bigoplus_{i\leq p} \mathcal{F}_{\bullet}\otimes R\langle \mathrm{Crit}_i(f)\rangle .    $$
The first page of the associated \emph{canonical spectral sequence} $E_{p,q}$ is given by
$$  E^1_{p,q} = \mathrm{H}_q(\mathcal{F}_{\bullet}) \otimes R\langle \mathrm{Crit}_p(f)\rangle  .  $$
Furthermore one introduces the \emph{lifted complex} given by
$    \widetilde{C}_i = \mathrm{H}_0(\mathbf{R}_{\bullet}) \otimes_R \mathrm{R}\langle \mathrm{Crit}_i(f) \rangle     $
which is an $\mathrm{H}_0(\mathbf{R}_{\bullet})$-left module equipped with the differential $\delta( x) = \sum_{|y| =|x|-1}   [m_{x,y}] x $ with $[m_{x,y}]\in \mathrm{H}_0(\mathbf{R}_{\bullet})$.
Barraud, Damian, Humilière and Oancea show in \cite[Lemma 4.3]{barraud2025morse} that the second page of the canonical spectral sequence of the DG Morse complex $\mathrm{C}_{\bullet}(M;\mathcal{F}_{\bullet})$ can be described as follows.
Consider $\mathrm{H}_{\bullet}(\mathcal{F}_{\bullet})$ as an $\mathrm{H}_0(\mathbf{R}_{\bullet})$-module and fix $q\in \mathbb{N}_0$. 
The complex
$     \mathrm{H}_{q}(\mathcal{F}_{\bullet})\otimes_{\mathrm{H}_0(\mathbf{R}_{\bullet})} \widetilde{C}_{\bullet}      $ 
has the differential
$  d([\alpha]\otimes x) = \sum_{|y| = |x|-1} [\alpha] \cdot [m_{x,y}]\otimes x    $ and there is an isomorphism $$E^2_{p,q} \cong \mathrm{H}_p(\widetilde{C}_{\bullet}; \mathrm{H}_q(\mathcal{F}_{\bullet})) .  $$

In \cite{barraud2025morse} the authors mainly consider the DGA $\mathbf{R}_{\bullet} = \mathrm{C}_{\bullet}(\Omega M)$ with $\Omega M$ the based Moore loop space.
Moreover, the typical example of DG modules $\mathcal{F}_{\bullet}$ is obtained from fibrations.
We shall first recall how the twisting cocycle in \cite{barraud2025morse} is constructed and then summarize how the DG modules over $\mathrm{C}_{\bullet}(\Omega M)$ arise from fibrations.

We consider the \emph{evaluation maps} $q_{x,y}\colon \overline{\mathcal{L}}(x,y)\to \Omega M $ as in \cite{barraud2025morse}.
Fix a basepoint $x_0\in M$ and let $\mathcal{Y}$ be a tree in $M$ with root the basepoint $x_0$ and with vertices the critical points $\mathrm{Crit}(f)$.
Let $p\colon M\to M/\mathcal{Y}$ be the canonical map to the quotient $M/\mathcal{Y}$ which collapses the tree $\mathcal{Y}$.
This is a homotopy equivalence and hence there is a homotopy inverse $\theta\colon M/\mathcal{Y}\to M$.
Let $\Gamma_{x,y}\colon \overline{\mathcal{L}}(x,y)\to C^0([0,f(x)-f(y)], M)$ be the map given by associating to $u\in \overline{\mathcal{L}}(x,y)$ the corresponding flow line parametrized as $\Gamma_{x,y}(u)(t) = \overline{u}\cap f^{-1}(\{f(x)-t\})$.
This induces a map $q_{x,y}\colon \overline{\mathcal{L}}(x,y)\to \Omega M$ by setting
$    q_{x,y}(u) =  \theta \circ p \circ (\Gamma_{x,y}(u)) .    $
For a broken flow line $u = u_1\circ u_2$ with $u_1\in \overline{\mathcal{L}}(x,z)$, $u_2\in \overline{\mathcal{L}}(z,y)$ we have
\begin{equation}\label{eq_transport_condition_evaluation_maps}
       q_{x,y}(u_1 \circ u_2 ) =q_{x,z}(u_1)\circ q_{z,y}(u_2) .   
\end{equation}
The chains
$$    m_{x,y}^{\Omega} := (q_{x,y})_* (s_{x,y})\in \mathrm{C}_{|x|-|y|-1}(\Omega M)        $$
yield a twisting cocycle where $s_{x,y}\in \mathrm{C}_{\bullet}(\overline{\mathcal{L}}(x,y))$ is the representing chain system as above.
The DG Morse homology with coefficients in a DG $\mathrm{C}_{\bullet}(\Omega M)$-module is then defined with the twisting cocycle $m^{\Omega}_{x,y}$.

Next, we recall from \cite{barraud2025morse} how DG modules over the DGA $\mathrm{C}_{\bullet}(\Omega M)$ arise from fibrations.
Let $q\colon \mathcal{E}\to M$ be a Hurewicz fibration.
Denote by $\mathcal{P}M$ and $\mathcal{P}E$ be the Moore path spaces of $M$ and $\mathcal{E}$, respectively. 
A \emph{lifting function} is a map
$$   \xi\colon \mathcal{E} \tensor[{_{q}}]{{\times}}{_{\mathrm{ev}_0}}  \mathcal{P}M \to \mathcal{P}\mathcal{E}  \quad \text{such that}\quad q\circ \xi = \mathrm{pr}_2 \quad\text{and}\quad \mathrm{ev}_0\circ \xi = \mathrm{pr}_1 .  $$
A \emph{transitive lifting function} is a map $$\Xi\colon \mathcal{E}\tensor[{_{q}}]{{\times}}{_{\mathrm{ev}_0}}  \mathcal{P}M\to \mathcal{E}  \quad \text{such that}\,\, q\circ \Xi = \mathrm{ev}_1\circ \mathrm{pr}_2  $$ 
and such that the following holds. For each $e\in \mathcal{E}$ we have $\Xi(e, c_{q(e)}) = e$ for $c_{q(e)}\in \mathcal{P}M$ the constant path at $q(e)\in M$ and for concatenable paths $\gamma,\sigma\in \mathcal{P}M$ we have
$   \Xi(\Xi(e,\gamma),\sigma) = \Xi(e, \gamma\circ \sigma) .    $
It is shown in \cite{dyer1969some} that any fibration is fiber homotopy equivalent to a fibration that admits a transitive lifting function.
Note further that for the fiber $F = \mathcal{E}_{x_0}$ over the basepoint $x_0\in M$ a transitive lifting function restricts to a map 
$$    \Xi|_{F\times \Omega M}\colon F\times \Omega M\to F    $$
which is an $\Omega M$ action.
This induces a $\mathrm{C}_{\bullet}(\Omega M)$ module structure on $\mathrm{C}_{\bullet}(F)$.
In this setup Barraud, Damian, Humilière and Oancea show the following.
\begin{theorem}[Theorem 7.2 \cite{barraud2025morse}]
    Let $(M,f,X,o,\{s_{x,y}\})$ be enhanced Morse data and let $\mathcal{Y}$ be a rooted tree in $M$ with vertices $\mathrm{Crit}(f)$.
    Let $E\to M$ be a fibration with typical fiber $F$ and consider the $\mathrm{C}_{\bullet}(\Omega M)$-module $\mathrm{C}_{\bullet}(F)$ where the module structure is induced by a transitive lifting function.
    Then the homology of the DG Morse complex $\mathrm{C}_{\bullet}(M;\mathrm{C}_{\bullet}(F))$ is isomorphic to the singular homology of the total space $E$.
\end{theorem}

\subsection{Principal bundles and their associated bundles}\label{subsec_principal_bundles}

In this subsection we review principal bundles and associated bundles.
While the contents of this subsection are standard, we include this subsection for completeness.
Moreover - in contrast to most of the literature - we consider principal bundles with left actions instead of right actions, see Remark \ref{remark_left-action} below, and we therefore want to spell out the basic definitions and constructions in order to avoid confusion.

\begin{definition}
    Let $B$ be a paracompact Hausdorff space and let $G$ be a topological group.
    A $G$-\emph{principal bundle} $q\colon E\to B$ is a fiber bundle such that there is a fiberwise left-action of $G$ on $E$ which is free and such that there are local trivializations $\tau_U\colon E|_{U}\to G\times U$ with the $G$-action being the obvious $G$-left action under the identification $\tau_U$.  
    Similarly, one defines a \emph{smooth} $G$\emph{-principal bundle} $q\colon E\to M$ for smooth manifolds $E$ and $M$, a Lie group $G$ and where the map $q$ as well as the action and the local trivializations are assumed to be smooth.
\end{definition}

\begin{remark}\label{remark_left-action}
    It is more common to define principal bundle via \emph{right actions}.
    For the associated bundles which we shall encounter below one then takes \emph{left actions}.
    We do not stick to this convention in this article for the following reason.
    Barraud, Damian, Humilière and Oancea \cite{barraud2025morse} consider right actions of $\Omega M$ on the fiber of a fibration.
    Since we are adapting some construction of \cite{barraud2025morse} to our setup and since we eventually want to compare the different DG Morse complexes in Section \ref{sec_comparison} we shall stick to considering principal bundles with left actions in this manuscript.
\end{remark}

Let $q\colon E\to B$ be a $G$-principal bundle and let $\varphi\colon F\times G\to F$ be a $G$-right action.
There is a diagonal left action $G\times (F\times E)\to F\times E$ given by
$$    (g, (f,u)) \mapsto  (f\cdot g^{-1}, g\cdot u) .   $$
This action is again free and the quotient $(F\times E)/G \to B, [f,u]\mapsto q(u)$ is a bundle over $B$ with typical fiber $F$ which we call the \emph{associated bundle}.
We shall write $F\times_G E := (F\times E)/G$ for the total space.
If a fiber bundle $X\to B$ arises as an associated bundle we call $G$ the \emph{structure group}.

\begin{example}
	Let $\pi\colon E\to B$ be a complex vector bundle of rank $k$ with a Hermitian fiber metric.
	It is well-known that the unitary frame bundle $U(E)\to B$ is a $U(k)$-principal fiber bundle.
	Consider the standard action $ \mathbb{C}^k \times U(k)\to \mathbb{C}^k$ given by $(z,A)\mapsto A^{-1}\cdot z$.
	The associated fiber bundle $ \mathbb{C}^k \times_{U(k)} U(E) \to B$ is naturally isomorphic to the original vector bundle $E\to B$.
	Moreover, one can view the projectivization of $E$ as the associated bundle $ \mathbb{C}P^{k-1}  \times_{U(k)} U(E)\to B$ where $ \mathbb{C}P^{k-1}\times U(k) \to \mathbb{C}P^{k-1}$ is the canonical action induced by the standard action $\mathbb{C}^k\times U(k)\to \mathbb{C}^k$.
\end{example}
\begin{example}\label{example_free_loop_space} 
	If $G$ is a Lie group and $H\subseteq G$ is a closed subgroup, then the canonical map $p\colon G\to H \backslash G =M $ to the homogeneous space $H\backslash G$ is a smooth $H$-principal fiber bundle.
	We want to see that the free loop space $L M = \mathrm{Map}(\mathbb{S}^1, M)$ can be understood as an associated bundle to this $H$-principal bundle.
	Note that there is a $G$-action $\psi\colon M \times G\to M$ given by $\psi([k],g) = [kg]$.
	This restricts to an $H$-action $\psi'\colon M\times H\to M$ which keeps the basepoint $p_0  =[e]\in M$ fixed.
	Consider the based loop space
	$    \Omega_{p_0} M = \{\gamma\in LM \,|\, \gamma(0) = p_0 \} .    $	
    There is an induced continuous action $\Psi\colon  \Omega_{p_0}M\times H\to \Omega_{p_0}M$ given by $\Psi(\gamma,h)(t) = \psi'(\gamma(t),h)$.
	We claim that the associated bundle $q\colon E =  \Omega_{p_0}M  \times_H G\to M$ is fiberwise isomorphic to the free loop fibration $ \mathrm{ev}\colon LM \to M, \,\mathrm{ev}(\gamma) = \gamma(0)$.
	Indeed, define maps $f\colon E\to LM$ and $h\colon LM\to E$ by 
	$$     f([\gamma,g])(t) =  \psi(\gamma(t),g)    , \qquad  
	    h(\sigma) =   [  \sigma g_{\sigma(0)}^{-1} , g_{\sigma(0)}]      $$
	for $g\in G,\gamma\in \Omega_{p_0}M $ and $\sigma\in LM$ where $g_{\sigma(0)}\in G$ is such that $[g_{\sigma(0)}] = \sigma(0)\in M$.
	One checks that $f$ and $h$ are well-defined fiberwise maps and both $f$ and $h$ are continuous and inverse to each other.
\end{example}
\begin{example}
    More generally, if $M$ is an arbitrary closed manifold, then the free loop fibration $\mathrm{ev} \colon LM\to M$ is a fiber bundle with structure group the basepoint-preserving diffeormorphisms $\mathrm{Diff}_{p_0}(M)$ and typical fiber $\Omega_{p_0}M$, see \cite[Theorem 1.1]{chataur2015basics}. 
\end{example}

We now recall the well-known local description of $G$-principal bundles via cocycles since this concept is crucial for the link between transport functions and principal bundles in the next section.
Let $q\colon E\to M$ be a $G$-principal bundle and let $\tau_U\colon E|_U \to G\times U$ and $\tau_V \colon E|_V\to G\times V$ be a local trivializations with $U\cap V\neq \emptyset$.
Since the trivializations are $G$-equivariant one checks that there is a map $\omega_{UV}\colon U\cap V\to G$ such that
$$   \tau_U \cap \tau_V^{-1} (g,p) = (g\cdot \omega_{UV}(p),p) \quad \text{for}\,\,\, g\in G,\,\,p\in U\cap V.     $$
Explicitly, we have $\omega_{UV}(p)  = \mathrm{pr}_1( \tau_U\circ \tau_V^{-1}(e,p))$ for $p\in U\cap V$.
Further one can see that on triple intersections $U\cap V\cap W\neq \emptyset$ the identity 
\begin{equation}\label{eq_cocycle_identity}
         \omega_{UV}(p ) = \omega_{WV}(p) \cdot \omega_{UW}(p), \quad p\in U\cap V\cap W 
\end{equation}
holds.
As mentioned above one usually considers principal bundles endowed with right-actions.
In this case the factors on the right hand side of equation \eqref{eq_cocycle_identity} appear in the opposite order.
Equation \eqref{eq_cocycle_identity} is often referred to as \emph{cocycle identity} the and it leads to the well-known definition of $G$-valued cocycles.

\begin{definition}
Let $\{U\}_{U\in\mathcal{U}}$ be an open cover of a topological space $B$ and let $G$ be a topological group.
\begin{enumerate}
    \item A system of maps $\{\omega_{UV}\colon U\cap V\to G\}_{U,V\in\mathcal{U}}$ satisfying equation \eqref{eq_cocycle_identity} is called a $G$\emph{-valued cocycle}.
    \item Two $G$-valued cocycles $\{\omega_{UV}\},\{\omega'_{UV}\}$ are called \emph{equivalent} if there is a family of maps $\{r_U\colon U\to G\}_{U\in\mathcal{U}}$ such that 
    $$     \omega_{UV}(p) =r_V(p) \cdot \omega'_{UV}(p) \cdot r_U(p)^{-1} \quad \text{for}\,\,\, p\in U\cap V.    $$
\end{enumerate}
\end{definition}
Again note that the order of the terms $r_V(p)$ and $r_U(p)$ is opposite to the usual convention.
This is due to our use of $G$-principal bundles with left actions.
The following result is the well-known characterization of isomorphism classes of principal bundles through $G$-valued cocycles.
\begin{theorem}
    Let $G$ be a topological group and let $B$ be a topological space.
    There is a bijection between equivalence classes of $G$-valued cocycles and isomorphism classes of $G$-principal bundles over $B$.
\end{theorem}
We refer to \cite[Theorem 5.2.7 and Theorem 5.3.2]{husemoller1966fibre} for a proof.
In order to go from a principal bundle to a cocyle one uses the transition functions $p\mapsto \tau_U\cap \tau_V^{-1}(e,p)$ as above.

\medskip
Note that in the topological category every fiber bundle over a paracompact space appears as an associated bundle.
In fact, if $X\to B$ is a fiber bundle with typical fiber $F$, then the local trivializations can be used to define a $\mathrm{Homeo}(F)$-valued cocycle and thus a $\mathrm{Homeo(F)}$-principal bundle $q\colon E\to B$.
The corresponding associated bundle $F\times_{\mathrm{Homeo(F)}} E\to B$ is isomorphic as a fiber bundle to $X\to B$, see \cite{steenrod1999topology}.
If $X\to M$ is a smooth fiber bundle, then the transition functions induce a $\mathrm{Diffeo(F)}$-principal bundle over $M$. 
In this case the corresponding associated bundle agrees with the original one only as a fiber bundle in the topological sense, since the $\mathrm{Diffeo}(F)$-principal bundle over $M$ does not carry a differentiable structure in general.
In many cases the \emph{structure group} can be chosen to be much smaller than the homeomorphism or the diffeomorphism group.
If $X\to M$ is a smooth bundle and the structure group can be chosen to be a Lie group, then going through the process sketched above, one obtains an associated bundle which is isomorphic to $X\to M$ as a smooth bundle.

\section{Transport functions and principal bundles} \label{sec_transport_function}

In this section we introduce transport functions and study their relationship to principal bundles.
We will show that every transport function induces a principal bundle and vice versa and that these two constructions are inverse to each other up to equivalence.
We note that this section builds upon Voigt's PhD thesis \cite[Section 4.2]{voigt2014diss}.
Since many details in Voigt's thesis are only sketched and since we extend Voigt's construction we give a self-contained and complete account of the connection between transport functions and principal bundles.


\begin{definition}
    Let $G$ be a topological group.
    We say that a collection of maps $\Phi_{x,y}\colon \overline{\mathcal{L}}(x,y)\to G$ for $x,y\in\mathrm{Crit}(f)$ with $|x|\geq |y|$ is a \emph{transport function} if $\Phi_{x,y}$ is continuous for all $x,y\in\mathrm{Crit}(f)$ and if 
    \begin{equation}\label{eq_transport_function}
             \Phi_{x,y}( u_1 \circ u_2) = \Phi_{x,z}( u_1) \cdot \Phi_{z,y}(u_2)        
    \end{equation}
    for each pair of broken flow lines $u_1\in \overline{\mathcal{L}}(x,z)$, $u_2\in\overline{\mathcal{L}}(z,y)$.
\end{definition}

As we shall see in Section \ref{sec_comparison} one can obtain transport functions from parallel transport in a smooth principal bundle with connection which justifies the name.

\begin{remark}
    Recall that Morse data $(M,f,X)$ determines the flow category $\mathcal{M}_f$ which is a topological category.
    If we let $\mathcal{G}$ be the topological category with one object and with morphism space being the group $G$, then a transport function is simply a continuous functor $\mathcal{M}_f\to \mathcal{G}$ between the topological categories $\mathcal{M}_f$ and $\mathcal{G}$. 
    For this reason we shall usually denote transport functions as maps $\Phi\colon \mathrm{mor}(\mathcal{M}_f)\to G$ where $\mathrm{mor}(\mathcal{M}_f) = \bigsqcup_{x,y\in\mathrm{Crit}(f)} \overline{\mathcal{L}}(x,y)$.
\end{remark}

\begin{example}\label{example_clutching}
Consider the height function $f\colon \mathbb{S}^n\to \mathbb{R}$ on the sphere, i.e. $f(x_1,\ldots, x_{n+1}) = x_{n+1}$.
Denote the critical points by $N = (0,\ldots, 0,1)$, $S = (0,\ldots, 0,-1)$.
The manifold of gradient flow lines $\mathcal{L}(N,S)$ is diffeomorphic to the equator $\mathbb{S}^{n-1} = \{(x_1,\ldots, x_n,0)\in\mathbb{S}^n\}$.
Let $G$ be a topological group.
A transport function $\Phi\colon \mathcal{L}(N,S)\to G$ is thus the same as a map $\Phi\colon \mathbb{S}^{n-1}\to G$ and this is nothing else but the classical \emph{clutching function}.
It is well-known that homotopy classes of clutching functions $[\mathbb{S}^{n-1},G]$ are in bijection to isomorphism classes of $G$-principal bundles over $\mathbb{S}^n$.
We shall see that transport functions for arbitrary Morse data $(M,f,X)$ generalize this phenomenon and yield all the isomorphism classes of $G$-principal bundles over $M$.
\end{example}

The goal for the rest of this section is to show that transport functions with values in $G$ describe $G$-principal bundles over $M$.
A Morse function $f\colon M\to \mathbb{R}$ is called \emph{self-indexing} if $f(p) = |p|$ for all $p\in\mathrm{Crit}(f)$.
If $(M,f,X)$ is Morse data on $M$, then there exists a self-indexing Morse function $g\colon M\to \mathbb{R}$ such that $\mathrm{Crit}(g) = \mathrm{Crit}(f)$ and such that the index of $x\in\mathrm{Crit}(f)$ is the same with respect to $f$ and with respect to $g$.
Moreover, $(M,g,X)$ is Morse data, i.e. the pseudo-gradient vector field $X$ remains unchanged.
See \cite[Corollary 2.36]{nicolaescu2007invitation} for a proof of this statement.
In particular we note that since the pseudo-gradient $X$ remains unchanged, the moduli spaces $\overline{\mathcal{L}}(x,y)$ are the same for $(M,f,X)$ as for $(M,g,X)$.
It is therefore no loss of generality to work with self-indexing Morse functions for the purpose of this section.

\medskip
\emph{For the rest of this section we fix Morse data $(M,f,X)$ and we assume that $f$ is self-indexing. Furthermore, we only consider metrizable topological groups $G$.}

\subsection{Contractible neighborhoods of the critical points}\label{subsec_separation_functions}

We begin by constructing an open cover $\{V_z\}_{z\in\mathrm{Crit}(f)}$ of the manifold $M$ such that the open set $V_z$ can be contracted to the critical point $z$ for each $z\in\mathrm{Crit}(f)$.
We shall further define certain subsets $R_z\subseteq V_z$ which will be crucial for the construction of the transport functions.

Recall that for each pair of critical points $x,y\in\mathrm{Crit}(f)$ the compactified moduli space $\overline{\mathcal{L}}(x,y)$ can be given the structure of a compact manifold with corners. 
Topologically, a manifold with corners is the same as a topological manifold with boundary.
We can therefore choose collar neighborhoods $U_{x,y}\subseteq \overline{\mathcal{L}}(x,y)$ of the boundary $\partial \overline{\mathcal{L}}(x,y)$, i.e. there are homeomorphisms $U_{x,y}\cong \partial\overline{\mathcal{L}}(x,y) \times [0,2 \epsilon)$ for some $\epsilon >0$.
We fix an $\epsilon>0$ as well as the collar neighborhoods $U_{x,y}$.
We further define the closed set $V_{x,y} \cong \partial\overline{\mathcal{L}}(x,y) \times [0,\epsilon] \subseteq U_{x,y}$. 

Next, we would like to choose a family of functions $\rho^j_{x,y}\colon \overline{\mathcal{L}}(x,y)\to \mathbb{R}$ for each $x,y\in\mathrm{Crit}(f)$ with $|y|<|x|$ and $j\in\mathbb{N}$ with $|y|+1\leq j \leq |x|$ with the following properties
\begin{enumerate}[(i)]
    \item The function $\rho^j_{x,y}$ is continuous and takes values in $(|y|,|x|)$.
    \item If $|y|< i < j \leq |x|$, then $\rho^i_{x,y}(u) \leq \rho^j_{x,y}(u)$ for all $u\in\overline{\mathcal{L}}(x,y)$.
    \item If $|y|< i < j \leq |x|$ and $u \in \overline{\mathcal{L}}(x,y)\setminus V_{x,y}$, then $\rho^i_{x,y}(u) = \rho^j_{x,y}(u) = \tfrac{|x|+|y|}{2}$. 
    \item If $u = u_0 \circ u_1\circ u_2\in\partial \overline{\mathcal{L}}(x,y)$, i.e. $u_0\in\overline{\mathcal{L}}(x,z_0)$, $u_1\in \overline{\mathcal{L}}(z_0,z_1), u_2\in\overline{\mathcal{L}}(z_1,y)$ with $|x|\geq |z_0|> |z_1| \geq |y|$ and if $j \in \{|z_1|+1, \ldots, |z_0|\}$ then $\rho^j_{x,y}(u) = \rho^j_{z_0,z_1}(u_1)$.
    \item 
    Let $\widehat{u}\in \partial\overline{\mathcal{L}}(x,y)$ and write $u = (\widehat{u},t)\in V_{x,y}$. The map $t\mapsto \rho_{x,y}^i(\widehat{u},t)$ is an affine linear map in $t$.
\end{enumerate}
We call such a family of functions a \emph{system of separation functions}.
The next lemma shows that there is a unique system of separation functions.
We shall use this system to divide the manifold $M$ into smaller chunks later.

\begin{lemma}\label{lemma_separation_functions_exist}
    There exists a unique system of separation functions $\{\rho_{x,y}^j\}$.
\end{lemma}
\begin{proof}
    In order to show existence we define the functions $\rho^j_{x,y}$ by induction over the relative index $|x|-|y|$.
    If $|x|-|y| = 1$, we only need to define $\rho^{|x|}_{x,y}\colon {\mathcal{L}}(x,y)\to \mathbb{R}$ and by property (iii) we need to define $       \rho_{x,y}^{|x|} (u) = \frac{|x| + |y| }{2}  $.

    Now let $k\geq 2$ and assume that for all pairs of critical points $(x,y)$ with $|x|-|y| < k$ the functions $\rho_{x,y}^j$ are defined for $j\in \{ |y|+1,\ldots, |x|\}$.
    Let $a,b\in \mathrm{Crit}(f)$ with $|a|-|b| = k$ and let $j\in \{|b| + 1,\ldots, |a|\}$.
    If the boundary of $\overline{\mathcal{L}}(a,b)$ is empty, we have to define $\rho_{a,b}^j$ to be the constant function with value $\tfrac{|a|+|b|}{2}$ by property (iii).
    
    Assume that the boundary $\partial\overline{\mathcal{L}}(a,b)$ is non-empty.
    Let $j\in \{|a|+1,\ldots, |b|\}$.
    Since the boundary consists of broken gradient flow lines, property (iv) already fixes the value of $$\widehat{\rho}_{a,b}^j := \rho_{a,b}^j|_{\partial \overline{\mathcal{L}}(a,b)} \colon \partial \overline{\mathcal{L}}(a,b)\to \mathbb{R}$$ on the boundary. 
    Indeed, assume that $u = u_1\circ \ldots \circ u_{\ell}\in \partial\overline{\mathcal{L}}(a,b)$ is maximally decomposed, then there is an $m\in \{1,\ldots, \ell\}$ such that $u_m\in \mathcal{L}(c,d)$ for critical points $c,d$ and such that $j\in \{|d|+1,\ldots, |c|\}$.
    By property (iv) we have to choose $\widehat{\rho}^j_{a,b}(u)  = \rho^j_{c,d}(u_m)$ and since the index difference $|c|-|d|$ is smaller than $|a|-|b|$, $\rho^j_{c,d}$ has already been defined.
    Let $(\widehat{u},t)\in U_{a,b}$ with $0\leq t\leq \epsilon$ then by property (v) we need to set
    $$     \rho_{a,b}^j (\widehat{u},t) = 
        \Big(1-\frac{t}{\epsilon}\Big) \cdot \widehat{\rho}_{{a,b}}^j (\widehat{u}) + \frac{t}{\epsilon}\cdot \frac{|a|+|b|}{2} .
     $$
    For $u\in \overline{\mathcal{L}}(a,b)\setminus V_{{a,b}}$ we have to set $      \rho_{a,b}^j ( u ) =\frac{|a|+|b|}{2}         $ by property (iii).
    The corresponding function is clearly continuous and one checks that properties (i)-(v) are satisfied.
    As we have seen we could not make any choices, therefore the system of separation functions is unique.
\end{proof}
Note that the system of separation functions does of course depend on the choice of the collar neighborhoods, but we have fixed this choice for this section.

Let $x,y\in\mathrm{Crit}(f)$ with $|y|< |x|$.
Using the system of separation functions we define the following subsets of $\overline{\mathcal{L}}(x,y)\times [|y|,|x|]$.
Choose an $\epsilon_0 >0$ with $\epsilon_0 < \tfrac{1}{4}$ and set
\begin{eqnarray*}
    \widetilde{U}_{x,y}^{|x|} &=& \{ (u,s)\in\overline{\mathcal{L}}(x,y)\times [|y|,|x|] \,|\,  \rho_{x,y}^{|x|} - \epsilon_0 < s  \} \\
    \widetilde{U}_{x,y}^{|y|}  &=&  \{(u,s) \in \overline{\mathcal{L}}(x,y) \times [|y|,|x|] \,|\,   s < \rho_{x,y}^{|y|+1}(u) + \epsilon_0 \} \quad \text{and} \\
    \widetilde{U}_{x,y}^{j} &=& \Big\{ ((\widehat{u},a),s) \in U_{x,y}\times [|y|,|x|] \,|\, \rho_{x,y}^j((\widehat{u},a)) - \epsilon_0 < s < \rho_{x,y}^{j+1}((\widehat{u},a)) + \epsilon_0,  \\   &  & \hphantom{errgjer} \rho_{x,y}^{j}(\widehat{u}) < \rho_{x,y}^{j+1}(\widehat{u}), \,\, a < \tfrac{4}{3}\epsilon \Big\} 
\end{eqnarray*}
for $j \in \{|y| + 1 ,|y|+2,\ldots, |x|-1\} $.
Note that all sets $\widetilde{U}_{x,y}^{j}$ are open in $\overline{\mathcal{L}}(x,y)\times [|y|,|x|]$.
We further set
\begin{eqnarray*}
\widetilde{P}_{x,y}^{|x|} &=& \{(u,s) \in \overline{\mathcal{L}}(x,y) \times [|y|,|x|] \,|\, \rho_{x,y}^{|x|}(u) \leq s  \} \\
\widetilde{P}_{x,y}^{|y|}  &=& \{ (u,s)\in \overline{\mathcal{L}}(x,y) \times [|y|,|x|]   \,|\, s \leq \rho_{x,y}^{|y|+1}(u)\} \quad \text{and}\quad \\
\widetilde{P}_{x,y}^{j} &=& \{ ((\widehat{u},a),s) \in U_{x,y} \times [|y|,|x|]\,|\,    \rho_{x,y}^j((\widehat{u},a))\leq s \leq \rho_{x,y}^{j+1}((\widehat{u},a))  \\ &  & \hphantom{errgjer}  \rho_{x,y}^{j}(\widehat{u}) < \rho_{x,y}^{j+1}(\widehat{u}), \,\, a\leq \epsilon    \}
\end{eqnarray*}
for $j \in \{|y| + 1 ,|y|+2,\ldots, |x|-1\} $.
We refer to Figure \ref{fig_sets_pxyi} for a sketch of these sets.
It is immediate that $\widetilde{P}_{x,y}^{j} \subseteq \widetilde{U}_{x,y}^j$ for $j\in \{|y|,\ldots ,|x|\}$.
We further define $U_{x,y}^j = \pi(\widetilde{U}_{x,y}^j)\subseteq M$ as well as $P_{x,y}^j = \pi(\widetilde{P}_{x,y}^j) \subseteq M$ for $j\in\{|y|,\ldots, |x|\}$ where $\pi$ is the quotient map $\pi\colon \widetilde{K}\to M$ as in Section \ref{subsec_morse_flow}.
The next lemma describes the most important properties of the sets $U_{x,y}^j$ and $P_{x,y}^j$.

\begin{figure}[t]
\centering
\includegraphics[scale=0.43]{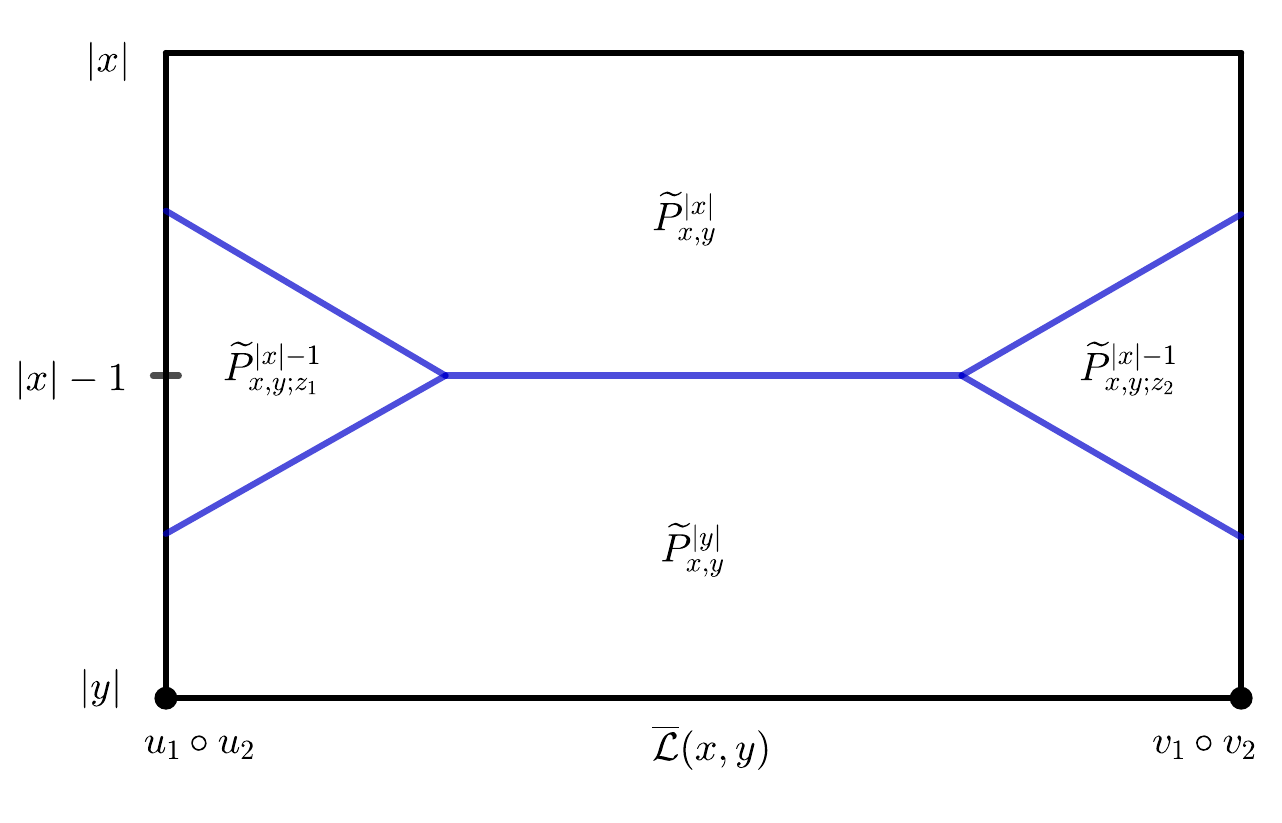}
\caption{The sets $\widetilde{P}_{x,y}^i$ for critical points $x,y\in\mathrm{Crit}(f)$ with $|x|-|y| = 2$. The boundary of the moduli space consists of two broken flow lines $\partial\overline{\mathcal{L}}(x,y) = \{ u_1\circ u_2,v_1\circ v_2\}$ with $u_1\in \mathcal{L}(x,z_1)$ and $v_1\in \mathcal{L}(x,z_2)$. The image of the components of $\widetilde{P}_{x,y}^{|x|-1}$ under $\pi$ can be contracted to $z_1$ and $z_2$, respectively.}
\label{fig_sets_pxyi}
\end{figure}

\begin{lemma}\label{prop_contractibility_pxyj}
    Let $(M,f,X)$ be self-indexing Morse data.
    \begin{enumerate}
        \item For each $j\in \{|y|,\ldots, |x|\}$ the set $U_{x,y}^j$ strongly deformation retracts onto $P_{x,y}^j$.
        \item The sets $P_{x,y}^{|y|}$ and $P_{x,y}^{|x|}$ can be contracted to $y$ and $x$, respectively.
        \item For $j\in \{|y|+1,\ldots, |x|-1\}$ the set $P_{x,y}^j$ consists of the following components.
        There are critical points $z_1,\ldots, z_k\in \mathrm{Crit}(f)$ of index $j$ and non-empty components $P_{x,y;z_i}^j\subseteq P_{x,y}^j$ for $i = 1,\ldots, k$ such that $P_{x,y;z_i}^j$ is contractible to $z_i$.
        Moreover, each component of $P_{x,y}^j$ is of this form, i.e.  we have
        $$   P_{x,y}^j = P_{x,y;z_1}^j \sqcup P_{x,y;z_2}^j \sqcup \ldots \sqcup P_{x,y;z_k}^j   .   $$
    \end{enumerate}
\end{lemma}
\begin{proof}
    For the first part, we begin by considering $j = |x|$ and define $H\colon \widetilde{U}_{x,y}^{|x|}\times [0,1]\to \widetilde{U}_{x,y}^{|x|}$ by
    $$   H((u,s),\mu ) = \begin{cases}
        (u,s) , & s\geq \rho_{x,y}^{|x|}(u) \\
        (u,  (1-\mu)\cdot s + \mu\cdot \rho_{x,y}^{|x|}(u)), & s \leq \rho_{x,y}^{|x|}(u) .
    \end{cases}        $$
    One checks that this induces a well-defined map $H'\colon U_{x,y}^{|x|}\times [0,1]\to U_{x,y}^{|x|}$.
    The map $H'(\cdot ,1)$ is a strong deformation retraction $U_{x,y}^{|x|}\to P_{x,y}^{|x|}$. 
    Analogously one obtains a strong deformation retraction $U_{x,y}^{|y|}\to P_{x,y}^{|y|}$.
    For $j\in \{|y|+1,\ldots, |x|-1\}$ we define a map
    $H\colon \widetilde{U}_{x,y}^{j}\times [0,1]\to \widetilde{U}_{x,y}^j$ by
    $$   H((\widehat{u},a),s,\mu) = \begin{cases}
        ((\widehat{u},a),s) , & ((\widehat{u},a),s)\in \widetilde{P}_{x,y}^j \\
        ((\widehat{u},a), (1-\mu)\cdot s + \mu \cdot  \rho_{x,y}^{j}((\widehat{u},a)) , &    s \leq \rho_{x,y}^j((\widehat{u},a)), \,\ a \leq \epsilon \\
        ((\widehat{u},a), (1-\mu)\cdot s + \mu \cdot  \rho_{x,y}^{j+1}((\widehat{u},a)) , &    s \geq \rho_{x,y}^{j+1}((\widehat{u},a)), \,\ a \leq \epsilon \\
        ((\widehat{u},(1-\mu)a + \mu \epsilon ), (1-\mu)s + \mu \tfrac{|x|+|y|}{2} ) , &    a \geq \epsilon  . 
    \end{cases}    $$
    One checks that this yields a strong deformation retraction $U_{x,y}^j\to P_{x,y}^j$.

    For the second part we want to show the contractibility of $P_{x,y}^{|x|}$.
    Define a map $H\colon {P}_{x,y}^{|x|}\times [0,1]\to  M$ by setting
    $$  H( \pi(u,s), t) = \pi(u, t|x| + (1-t)s) \quad \text{for}\,\,\, (u,s)\in \widetilde{P}_{x,y}^{|x|}, \,\, t\in [0,1] .    $$
    We need to show that $H$ is well-defined.
    Let $(u,s),(u',s') \in \widetilde{P}_{x,y}^{|x|}$ such that $\pi(u,s) = \pi(u',s')$ and $(u,s)\neq (u',s')$.
    By construction we have $s = s' \geq \rho_{x,y}^{|x|}(u)$.
    In case that $ s = s' = |x|$ then $t\mapsto H((u,s),t)$ is the constant path at the point $x$.
    If $s< |x|$ then we must have that $u$ and $u'$ are broken flow lines.
   Take the maximal decompositions $u = u_1\circ \ldots \circ u_{\ell}$ and $u' = u'_1 \circ \ldots \circ u'_k$.
   We have $u_1\in \mathcal{L}(x,z)$ with $|z|<|x|$ and $u_1'\in \mathcal{L}(x,z')$ with $|z'|<|x|$.    
    By property (v) of the separation functions and our specific construction in Lemma \ref{lemma_separation_functions_exist} it holds that $s \geq \rho_{x,y}^{|x|}(u) = \rho_{x,z}^{|x|} (u_1) = \tfrac{|x|+|z|}{2}$ since $u_1$ is unbroken.
    Similarly we have $s \geq \rho_{x,y}^{|x|}(u)  = \rho_{x,z'}^{|x|}(u_1) = \tfrac{|x| + |z'|}{2}$.
    Since $u_1(s) = \pi(u,s) = \pi(u',s) = u_1'(s)$ we must have $u_1 = u_1'$.
    Hence, we see that $H(\pi(u,s),t) = u_1( t|x| + (1-t)s) = H(\pi(u',s,t))$  and thus $H$ is well-defined.
    By the quotient construction we furthermore obtain that $H$ is continuous.
    Finally, we have $H(\pi(u,s),1) = x$ for all $\pi(u,s)\in P_{x,y}^{|x|}$.
    The statement for $j = |y|$ can be shown completely analogously.

    For the third part, let $x,y\in\mathrm{Crit}(f)$ and let $j\in \{|y|+1,\ldots, |x|-1\}$.
    We shall define a map $\widetilde{k}_{x,y}^j\colon \widetilde{P}_{x,y}^j \to \partial\overline{\mathcal{L}}(x,y)$ such that $\widetilde{P}_{x,y}^j$ is sent to $\widetilde{P}_{x,y}^j\cap \partial\overline{\mathcal{L}}(x,y)$ under $\widetilde{k}_{x,y}^j$ and such that $\widetilde{k}_{x,y}^j$ restricts to the identity on $\widetilde{P}_{x,y}^j \cap \partial\overline{\mathcal{L}}(x,y)$.
    We set
    $$    \widetilde{k}_{x,y}^j ((\widehat{u},a),s) = \big(  (\widehat{u},0),   s + \tfrac{\rho_{x,y}^j(\widehat{u}) + \rho_{x,y}^{j+1}(\widehat{u}) - (|x|+|y|)}{2\epsilon}\cdot a   \big)  .    $$
    Using the explicit form of the separation functions $\rho_{x,y}^j$ one checks that given a point $((\widehat{u},a),s)\in \widetilde{P}_{x,y}^j$ its image $k_{x,y}^j((\widehat{u},a),s)$ lies indeed in $\widetilde{P}_{x,y}^j\cap \partial\overline{\mathcal{L}}(x,y)$.
    The map $\widetilde{k}_{x,y}^j$ is continuous and its restriction to $\widetilde{P}_{x,y}^j \cap \partial\overline{\mathcal{L}}(x,y)$ is the identity.
    Let $\{K_j\}_{j=0}^n$ be the filtration of $M$ as in Section \ref{subsec_morse_flow}.
    We obtain an induced map $k_{x,y}^j\colon P_{x,y}^j\to K_{|x|-|y|-1}$.
    Moreover, let $k_{x,y}^{|x|}\colon {P}_{x,y}^{|x|}\to M$ be the map $p\mapsto x$ and similarly, let $k_{x,y}^{|y|}\colon P_{x,y}^{|y|}\to M$ be the map $p\mapsto y$.
    We now recursively define subsets $P_{x,y;z}^j\subseteq P_{x,y}^j$ for $x,y\in\mathrm{Crit}(f)$ and $z\in\mathrm{Crit}(f)$.

    First, let $a,b\in \mathrm{Crit}(f)$ with index difference $|b|-|a| = 1$. We set $P_{a,b;a}^{|a|} = P_{a,b}^{|a|}$ and $P_{a,b;b}^{|b|}  = P_{a,b}^{|b|}$.
    Moreover, for $c\in \mathrm{Crit}(f)$ with $c\not\in \{a,b\}$ set $P_{a,b;c}^j = \emptyset$ for $j\in \{|a|,|b|\}$.
    For $z\in\mathrm{Crit}(f)$ we define the sets $K_{0,z} = \{z\}\subseteq K_0$ and $$  K_{1,z} =  \bigcup_{|x|-|y| = 1} P_{x,y;z}^{|z|} \subseteq K_1.   $$
    Furthermore, let $\xi_{1,z} \colon K_{1,z}\to K_{0,z}$ be the map $\xi_z(p) = z$ and note that this is a strong deformation retraction.

    Now, assume that for all pairs of critical points with index difference $r = |x|-|y| \leq m$ we have defined sets $P_{x,y;z}^j$ for all $z\in\mathrm{Crit}(f)$ as well as sets $K_{r,z}$ together with maps $\xi_{r,z}\colon K_{r,z}\to K_{r-1,z}$ which are strong deformation retractions.
    Let $x,y\in\mathrm{Crit}(f)$ with $|x|-|y|  = m+1$.
    Define $P_{x,y;x}^{|x|} = P_{x,y}^{|x|}$ and $P_{x,y;y}^{|y|} = P_{x,y}^{|y|}$.
    Let $j \in \{|y|+1,\ldots, |x|-1\}$ and consider the map $k_{x,y}^j\colon P_{x,y}^j\to K_m$.
    For $z\not\in \{x,y\}$ we define $P_{x,y;z}^j = (k_{x,y}^j)^{-1}(K_{m,z}) $.
    We set
    $$ K_{m+1,z} =   \bigcup_{|x|-|y| = m+1}P_{x,y;z}^{|z|}    $$
    and we define a map $\xi_{m+1,z}\colon K_{m+1,z}\to K_{m,z}$ by
    $   \xi_{m+1,z}|_{P_{x,y;z}^j} := k_{x,y}^j|_{P_{x,y;z}^j}      $.
    This yields a well-defined continuous map which is a strong deformation retraction. The composition
    $$   \xi_{1,z}\circ \xi_{2,z}\circ \ldots \circ \xi_{m+1,z}\colon K_{m+1,z} \to M   $$
    is thus a strong deformation retraction as well by the inductive hypothesis.
    Of course, this map is just the constant map to $z$.
    It remains to show that the set $P_{x,y}^j$ decomposes into sets
    $   P_{x,y}^j =  P_{x,y;z_1 } \sqcup \ldots \sqcup P_{x,y;z_k}     $
    for critical points $z_1,\ldots, z_k\in \mathrm{Crit}_j(f)$.
    This statement is true for $|x|-|y| = 1$.
    Therefore, let $|x|-|y| \geq 2$ and assume that we have shown the statement for all pairs of critical points with smaller index difference.
    
    By construction we have $P_{x,y}^j = \cup_{z\in \mathrm{Crit}_j(f)} P_{x,y;z}^j$.
    Let $P_{x,y;z}^j$ and $P_{x,y;a}^j$ be non-empty with $z\neq a$ and assume that there is a point $p = \pi((\widehat{u},a),s)\in P_{x,y;z}^j\cap P_{x,y;a}^j$.
    This means that $q = k_{x,y}^j(p)  \in K_{r-1,z}\cap K_{r-1,a}$.
    Now, we have
    $$   K_{r-1,z} = \bigcup_{|x'|-|y'| = r-1}   P_{x',y';z}^{j} \quad \text{and}\quad     K_{r-1,a} = \bigcup_{|x'|-|y'| = r-1}   P_{x',y';a}^{j} .  $$
    Hence, we have $q\in P_{\widetilde{x},\widetilde{y};z}^j \cap P_{\overline{x},\overline{y};a}^j$ for pairs of critical points $(\widetilde{x},\widetilde{y})$ and  $(\overline{x},\overline{y})$ with index difference $r-1$.
    If $(\widetilde{x},\widetilde{y}) = (\overline{x},\overline{y})$, then we get a contradiction to the inductive hypothesis.
    Hence, assume that $(\widetilde{x},\widetilde{y})\neq (\overline{x},\overline{y})$.
    One checks that we have $P_{\widetilde{x},\widetilde{y}}^j\cap P_{\overline{x},\overline{y}}^j \subseteq K_{r-2}$ and therefore $q\in P_{x',y';z}^j \cap P_{x'',y'';a}^j$ for critical points $x',y',x'',y''\in\mathrm{Crit}(f)$ with $|x'|-|y'| =  |x''|- |y''| = r-2$.
    A downward induction then proves the claim.
\end{proof}

For $z\in \mathrm{Crit}(f)$ we define $R_z = K_{n,z}\subseteq M$ which is contractible to $z$ by the previous proposition.
Moreover, note that by construction we have $P_{x,y;z}^j\subseteq R_z$ for critical points $x,y\in\mathrm{Crit}(f)$.
See Figure \ref{fig_theothersphere1} for a sketch of the sets $P_{x,y}^j$ and $R_z$.

\begin{figure}[t]
\centering
\includegraphics[scale=1.]{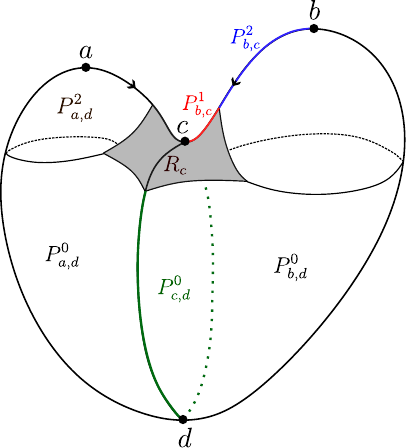}
\caption{A Morse function on the $2$-sphere with four critical points.
Some of the sets $P_{x,y}^j$ are shown as well as the neighborhood $R_c$ of the critical point $c$.}
\label{fig_theothersphere1}
\end{figure}

\begin{lemma}\label{lemma_contractibility_of_Rz}
    Let $z\in\mathrm{Crit}(f)$. The set $R_z\subseteq M$ admits an open neighborhood $V_z\subseteq M$ which strongly deformation retracts onto $R_z$ and is thus contractible as well.
    The family of open sets $\{V_z\}_{z\in\mathrm{Crit}(f)}$ is an open cover of $M$.
\end{lemma}
\begin{proof}
    As shown in Lemma \ref{prop_contractibility_pxyj} we have a homotopy equivalence $U_{x,y}^j\simeq P_{x,y}^j$.
    Hence we obtain components $U_{x,y;z}^j$ for each $z\in\mathrm{Crit}_j(f)$ with $P_{x,y;z}^j\neq \emptyset$.
    Recall that 
    $$ R_z = K_{n,z} =   \bigcup_{|x|-|y| = n} P_{x,y;z}^{|z|} .  $$
    We define 
    $$     V_z =  \bigcup_{|x|-|y| = n} U_{x,y;z}^{|z|}  . $$
    One checks from the explicit expression in Lemma \ref{prop_contractibility_pxyj} that the retractions $ U_{x,y;z}^j \to P_{x,y;z}^j$ fit together such that we obtain a strong deformation retraction $r_z\colon V_z\to R_z$.
    
    We claim that $V_z$ is an open set.  
    Let $p\in V_z$ with $p = \pi(u,s) \in U_{x,y;z}^{|z|}$ for critical points $x,y\in \mathrm{Crit}(f)$ with $|x|-|y| = n$.
    Assume that $u\not\in \partial\overline{\mathcal{L}}(x,y)$.
    Since the restriction of $\pi$ to ${{\mathcal{L}}(x,y)\times (0,n)}$ is a homeomorphism onto its image, it follows from the openness of $\widetilde{U}_{x,y}^{|z|}$ that $p$ has an open neighborhood in $V_z$.
    Next, suppose that $u\in \partial\overline{\mathcal{L}}(x,y)$ and denote the maximal decomposition of $u$ by $u = u_1\circ \ldots \circ u_k$.
    There are now three cases.
    In the first case there is an $\ell\in \{1,\ldots, k\}$ such that $u_{\ell}\in \mathcal{L}(z_1,z_2)$ with $|z_1| > |z| > |z_2|$ and $p = \pi(u,s) = \pi( u_{\ell},s)$ with $s\in (|z_2|,|z_1|)$.
    In the second case we have $u_{\ell}\in \overline{\mathcal{L}}(z_1,z)$ and $s\in [|z|,|z_1|)$ and in the third case we have $u_{\ell}\in \overline{\mathcal{L}}(z,z_1)$ and $s\in (|z_1|,|z|]$.

    We treat the first case, the second and third one can be dealt with using a similar argument. 
    Suppose that $p$ admits no neighborhood in $V_z$, then there is a sequence of points $q_k\in M\setminus V_z$ which converges to $p$.
    Write $q_k = \pi(u_k,s_k)$ for flow lines $u_k$ and $s_k\in [0,n]$.
    Since $f(q_k) = s_k$ we conclude that $s_k\to s$.
    We can always choose the flow lines $u_k$ such that $u_k\in\overline{\mathcal{L}}(x_k,y_k)$ with $|x_k|-|y_k| = n$.
    By the compactness of $M$ as well as by the compactness of the moduli spaces there are critical points $x_*, y_*\in\mathrm{Crit}(f)$ with $|x_*|  - |y_*| = n$ and a convergent subsequence $(u_{k_{\ell}})_{\ell\in \mathbb{N}}$ in $\overline{\mathcal{L}}(x_*,y_*)$.
    Denote the limit by $u_*\in\overline{\mathcal{L}}(x_*,y_*)$.
    Clearly, we have $\pi(u_*,s) = p$ which implies that $u_*$ has a decomposition $u_* = v\circ u_{\ell} \circ w$ with $v\in \overline{\mathcal{L}}(x_*,z_1)$, $w\in \overline{\mathcal{L}}(z_2,y_*)$ and $u_{\ell}$ the flow line as above.
    Since $\rho_{x_*,y_*}^m(u_*) = \rho_{z_1,z_2}^m(u_{\ell})$ for $m\in \{|z|,|z|+1\}$ we have that $p\in U_{x_*,y_*}^{|z|}$.
    Note that $r_z(p)\in R_z$ and thus $p\in U_{x_*,y_*;z}^{|z|}$.
    The point $(u_*,s)\in\widetilde{U}_{x_*,y_*}^{|z|}$ has a neighborhood $W\subseteq U_{x_*,y_*;z}^j$ and thus $\pi(W)\subseteq V_z$.
    For some number $N\in \mathbb{N}$ it holds that $(u_{k_{\ell}},s_{k_{\ell}}) \in W$ for $\ell \geq N$.
    Hence, we have $q_{k_{\ell}}\in \pi(W)\subseteq V_z$ for $\ell \geq N$ which contradicts our assumption. 
\end{proof}

\subsection{From local trivializations to transport functions}\label{subsec_3-2}

We have now constructed a cover of $M$ by contractible open sets, one set for each critical point.
If we have a $G$-principal bundle $q\colon E\to M$, the restriction to a contractible open set admits a local trivialization.
We will be interested in the corresponding transition functions in the sense of Section \ref{subsec_principal_bundles}.
Therefore we need to better understand the overlaps between the sets $R_z$, respectively the sets $P_{x,y}^j$.

Recall that for given $x,y\in\mathrm{Crit}(f)$ with $|y| < |x|$ the separation functions $\rho_{x,y}^j\colon \overline{\mathcal{L}}(x,y)\to \mathbb{R}$ are defined for $j= |y|+1,|y|+2,\ldots, |x|$.
For the formulation of the next lemma we define $\rho_{x,y}^{|y|},\rho_{x,y}^{|x|+1}\colon \overline{\mathcal{L}}(x,y)\to \mathbb{R}$ to be the constant functions 
$$   \rho_{x,y}^{|y|} \equiv |y|-1 \quad \text{and}\quad \rho_{x,y}^{|x|+1} \equiv |x|+1 .     $$

For $(i,j)\neq (|y|,|x|)$ consider the set
    \begin{equation}\label{eq_sxyij}
              \widetilde{S}_{x,y}^{\,i,j} = \{ (\widehat{u},a)\in U_{x,y} \,|\,  \rho_{x,y}^i(\widehat{u}) < \rho_{x,y}^{i+1}(\widehat{u}) = \rho_{x,y}^j(\widehat{u}) < \rho_{x,y}^{j+1}(\widehat{u}), \,\, a\leq \epsilon \} .       
    \end{equation}
    For $(i,j) = (|y|,|x|)$ we set $\widetilde{S}_{x,y}^{|y|,|x|} = \{u\in \overline{\mathcal{L}}(x,y) \,|\,  u\not\in \mathring{V}_{x,y}\}$.
    Furthermore, set
    $$    S_{x,y}^{i,j}= \pi\Big\{ (u,\rho_{x,y}^{i+1}(u))\in \overline{\mathcal{L}}(x,y) \times [|y|,|x|]\,\Big|\, u\in \widetilde{S}_{x,y}^{\,i,j}\Big\} \subseteq M  .   $$
    Then for all $i,j\in \{|y|,\ldots, |x|\}$ with $i< j$ it holds that  \begin{equation}\label{eq_s_contained_in_p}
         S_{x,y}^{i,j} \subseteq P_{x,y}^i\cap P_{x,y}^j
    \end{equation}
    as a direct verification shows.
Next, we want to relate transport functions and transition functions.
The technical intermediary construction is that of a pre-transport function which we now define.

\begin{definition}\label{definition_pre-transport-function}
    Let $(M,f,X)$ be Morse data and let $G$ be a topological group.
\begin{enumerate}
    \item 
    A family of continuous maps $\varphi_{x,y}^{i,j} \colon \widetilde{S}_{x,y}^{\,i,j} \to G  $ given for every pair of critical points $x,y\in\mathrm{Crit}(f)$ with $|y|<|x|$ and all $i,j\in\{ |y|,\ldots, |x|\}$ with $i<j$ is called \emph{pre-transport function} if the following conditions are satisfied.  
    \begin{enumerate}[(i)]
        \item Let $u = u_a\circ u_b\circ u_c$ with $u_a\in \overline{\mathcal{L}}(x,x_0), u_b\in\overline{\mathcal{L}}(x_0,y_0), u_c\in \overline{\mathcal{L}}(y_0,y)$ and $|x| \geq |x_0| > |y_0|\geq |y|$.
        Let $i,j\in \{|y_0|,\ldots, |x_0|\}$ and assume that $u_b\in \widetilde{S}_{x_0,y_0}^{i,j}$, then 
        $$   \varphi_{x,y}^{i,j}(u) = \varphi_{x_0,y_0}^{i,j}(u_b) .   $$
        \item Let $(u_n)_{n\in \mathbb{N}}$ be a convergent sequence of flow lines in $\overline{\mathcal{L}}(x,y)$ with $u_n\to u_*\in \overline{\mathcal{L}}(x,y)$.
        Let $|y| \leq i_1 < i_2 < \ldots < i_m \leq |x|$ be a sequence of indices and assume that $u_n\in \widetilde{S}_{x,y}^{\,i_1, i_2}\cap \widetilde{S}_{x,y}^{\,i_2, i_3}\cap \ldots \cap \widetilde{S}_{x,y}^{\,i_{m-1}, i_m}$ for all $n\in\mathbb{N}$ as well as $u_*\in \widetilde{S}_{x,y}^{\,i_1 , i_m}$.
        Then it holds that $$\lim_{n\to\infty} \varphi_{x,y}^{i_{m-1},i_m}(u_n) \cdot \ldots \cdot  \varphi_{x,y}^{i_1,i_2}(u_n) = \varphi_{x,y}^{i_1 i_m}(u_*)  . $$   
    \end{enumerate}
    \item 
    If $\Phi\colon \mathrm{mor}(\mathcal{M}_f)\to G$ is a transport function and $\{\varphi_{x,y}^{i,j}\}$ is a pre-transport function, then we say that $\{\varphi_{x,y}^{i,j}\}$ is \emph{adapted} to $\Phi$ if for $u\in \overline{\mathcal{L}}(x,y)$ and $|y| = j_0<j_1< \ldots < j_m = |x|$ it holds that
    $$    \varphi_{x,y}^{j_{m-1},j_m}(u) \cdot \ldots \cdot \varphi_{x,y}^{j_0,j_1}(u) = \Phi(u)    $$
    whenever the left hand side is defined.
\end{enumerate}
\end{definition}

\begin{example}
In Figure \ref{fig_theothersphere2} we sketch some of the sets $S_{x,y}^{i,j}$ on the $2$-sphere for a Morse function $f\colon \mathbb{S}^2\to \mathbb{R}$ with four critical points.
Assume that $\{\varphi_{x,y}^{i,j}\}$ is a pre-transport function adapted to a transport function $\Phi\colon \mathrm{mor}(\mathcal{M}_f)\to G$.
Note that the flow line $u$ lies both in the set $\widetilde{S}_{b,d}^{\,0,1}$ as well as in $\widetilde{S}_{b,d}^{\,1,2}$.
For the flow line $v$ we have group elements $\varphi_{b,d}^{0,2}(v),\varphi_{b,d}^{0,1}(v),\varphi_{b,d}^{1,2}(v)\in G$ and they have to satisfy
$$    \varphi_{b,d}^{1,2}(v) \cdot \varphi_{b,d}^{0,1}(v) = \varphi_{b,d}^{0,2}(v) = \Phi(v) .       $$
The first equality has to hold by condition (ii) and the second one since $\{\varphi_{x,y}^{i,j}\}$ is adapted to $\Phi$.
\end{example}

\begin{figure}[t]
\centering
\includegraphics[scale=1.]{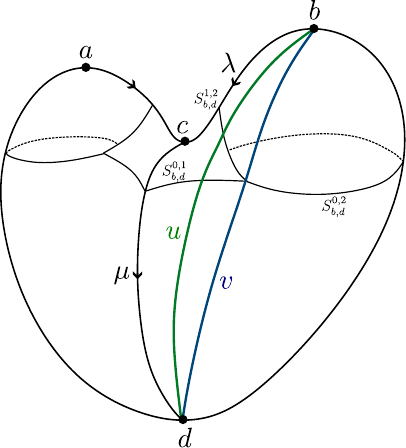}
\caption{Some of the sets $S_{x,y}^{i,j}$ on the $2$-sphere with four critical points.
The flow line $u$ lies in the sets $\widetilde{S}_{b,d}^{\,0,1}$ and $\widetilde{S}_{b,d}^{\,1,2}$ while the flow line $v$ lies in the sets  $\widetilde{S}_{b,d}^{\,0,2}, \widetilde{S}_{b,d}^{\,0,1}$ and $\widetilde{S}_{b,d}^{\,1,2}$. }
\label{fig_theothersphere2}
\end{figure}

\begin{lemma}\label{lemma_trivializations_induce_pre-transport}
    Let $(M,f,X)$ be Morse data and let $q\colon E\to M$ be a $G$-principal bundle.
    A choice of trivializations $\{E|_{V_z} \xrightarrow[]{\cong} G\times V_z\}_{z\in \mathrm{Crit}(f)}$ induces a pre-transport function $\{\varphi_{x,y}^{i,j}\}$ such that $\varphi_{x,y}^{i,j}$ is the corresponding change of trivializations on $S_{x,y}^{i,j} \cap P_{x,y;z_1}^i\cap P_{x,y;z_2}^j$ for suitable $z_1,z_2\in \mathrm{Crit}(f)$.
\end{lemma}
\begin{proof} 
    Let $z\in\mathrm{Crit}(f)$ with $|z| = k$.
    The set $V_z$ is contractible by Lemma \ref{lemma_contractibility_of_Rz}.
    We choose and fix a trivialization $\psi_z\colon E|_{V_z}\xrightarrow[]{\cong} G \times  V_z$ which restricts to a map $\psi_z\colon E|_{R_z}\xrightarrow[]{\cong} G\times R_z$.
    Let $x,y\in\mathrm{Crit}(f)$ and recall that we have $P_{x,y;z}^k\subseteq R_z$.
    We define restrictions $\psi_{x,y;z}^{k}\colon E|_{P_{x,y;z}^{k}} \to G\times P_{x,y;z}^{k}$ by restricting $\psi_z$ to $E|_{P_{x,y;z}^{k}}$.
    For critical points $z,a\in \mathrm{Crit}(f)$ with $|z| = i$, $|a| = j$ and $i< j$ we consider the compositions $\psi^i_{x,y;z}\circ (\psi^j_{x,y;a})^{-1}$.
    These are defined on $G\times (P_{x,y;z}^i\cap P_{x,y;a}^j)$ and we have
    $$   \psi^i_{x,y;z}\circ (\psi^j_{x,y;a})^{-1}|_{G\times (P_{x,y;z}^i\cap P^j_{x,y;a})} = \big( \psi_z |_{E_{P_{x,y;z}^i\cap P^j_{x,y;a}}}  \big) \circ \big(  \psi_a |_{E_{P_{x,y;z}^i\cap P^j_{x,y;a}}}  \big)^{-1}  . $$
    This map is of the form $(h,p)\mapsto ( g\cdot \psi_{x,y;z,a}^{i,j}(p),p)$ for some map $\psi_{x,y;z,a}^{i,j}\colon  P_{x,y;z}^j \cap P_{x,y;a}^i\to G$.
    Explicitly we have
    $$  \psi_{x,y;z,a}^{i,j} (p)  =   \mathrm{pr}_1 \circ \big( \psi^i_{x,y;z}\circ (\psi^j_{x,y;a})^{-1}|_{G\times (P_{x,y;z}^i\cap P^j_{x,y;a})} \big) (e,p) .       $$
    Recall that there is an inclusion $S_{x,y}^{\, i,j}\subseteq P_{x,y}^i\cap P_{x,y}^j $.
    Moreover there is a map $\widetilde{\pi}_{x,y}^{\, i,j}\colon \widetilde{S}_{x,y}^{\, i,j}\to S_{x,y}^{i,j}$ given by $u\mapsto \pi(u,\rho_{x,y}^{i+1}(u))$.
    The set ${S}_{x,y}^{i,j}$ decomposes into components $P_{x,y;z}^i\cap P_{x,y;a}^j$ and similarly the set $\widetilde{S}_{x,y}^{\, i,j}$ has components $\widetilde{S}_{x,y;z,a}^{\, i,j}$ which are the preimages of the sets ${S}_{x,y;z,a}^{i,j}$ under $\widetilde{\pi}_{x,y}^{\, i,j}$.
    Hence, $\widetilde{\pi}_{x,y}^{\, i,j}$ maps $\widetilde{S}_{x,y;z,a}^{\, i,j}$ into $P_{x,y;z}^i\cap P_{x,y;a}^j$.
    On the component $\widetilde{S}_{x,y;z,a}^{\, i,j}$ we define $\varphi_{x,y}^{i,j}\colon \widetilde{S}_{x,y;z,a}^{\, i,j}\to G$ to be the pullback of $\psi_{x,y;z,a}^{i,j}$ along $\widetilde{\pi}_{x,y}^{\,i,j}$, i.e. 
    $$     \varphi_{x,y}^{i,j}(u) :=  \psi^{i,j}_{x,y;z,a}(\pi( u,\rho_{x,y}^{i+1}(u))      \quad \text{for}\,\,\, u\in\widetilde{S}_{x,y;z,a}^{\, i,j} . $$
    In order to show that condition (ii) of a pre-transport function holds let $(u_n)_{n\in \mathbb{N}}$ be a converging sequence in $\overline{\mathcal{L}}(x,y)$ with $u_n\to u_*\in \overline{\mathcal{L}}(x,y)$.
    Assume that for all $n\in \mathbb{N}$ we have $u_n\in \widetilde{S}_{x,y}^{\, i_1, i_2}\cap \ldots \cap \widetilde{S}_{x,y}^{\, i_{m-1}, i_m}$ for indices $|y| \leq i_1 < \ldots < i_m \leq |x|$.
    Furthermore assume that $u_*\in \widetilde{S}_{x,y}^{\, i_1, i_m}$.
    There are critical points $a,b\in \mathrm{Crit}(f)$ with $ p_* = \widetilde{\pi}_{x,y}^{\, i,j}(u_*) \in S_{x,y;a,b}^{i,j}$.
    In particular $p_*\in V_a\cap V_b$.
    Moreover, there are critical points $c^k_n\in \mathrm{Crit}(f)$ with $n\in \mathbb{N}$, $k\in \{1,\ldots, m-1\}$ such that $q^k_n := \widetilde{\pi}_{x,y}^{\, i_k, i_{k+1}}(u_n) \in V_{c^k_n}\cap V_{c^{k+1}_n}$.
    Since $ \lim_{n\to\infty}q_n = \lim_{n\to\infty} \widetilde{\pi}_{x,y}^{\,i_k,i_{k+1}}(u_n) =  p_* $ there is $N\in \mathbb{N}$ such that $q_n\in V_a\cap V_b$ for $n\geq N$.
    Thus, for $n\in \mathbb{N}$, $n\geq N$ we have $q_n \in V_a \cap V_b \cap V_{c^1_n}\cap V_{c^2_n}\cap \ldots \cap V_{c^m_n}$.
    Since $$\varphi_{x,y}^{i_k,i_{k+1}}(u_n) = \psi_{x,y;c^k_nc^{k+1}_n}^{i,j} (\widetilde{\pi}_{x,y}^{\, i,j}(u_n))   $$
    and since the $\psi_{x,y;z,a}^{i,j}$ satisfy the cocycle condition we obtain that
    $$     \varphi_{x,y}^{i_{m-1},i_{m}}(u_n)\cdot \ldots \cdot \varphi_{x,y}^{i_1,i_2}(u_n) = \psi_{x,y;a,b}^{i_1,i_m}(q_n^k) .           $$
    Since $\psi_{x,y;a,b}^{i_1i_m}$ is continuous we conclude that
    $$  \lim_{n\to\infty} \varphi_{x,y}^{i_{m-1},i_{m}}(u_n)\cdot \ldots \cdot \varphi_{x,y}^{i_1,i_2}(u_n) = \lim_{n\to\infty}\psi_{x,y;a,b}^{i_1,i_m}(q^k_n) = \psi_{x,y;a,b}^{i_1,i_m}(q_*) = \varphi_{x,y}^{i_1,i_m} (u_*)    .    $$
    
    In order to show that condition (i) holds let $u = u_a\circ u_b\circ u_c$ with $u_a \in \overline{\mathcal{L}}(x,x_0)$, $u_b\in\overline{\mathcal{L}}(x_0,y_0)$ and $u_c\in\overline{\mathcal{L}}(y_0,y)$ for critical points $y,y_0,x_0,x\in\mathrm{Crit}(f)$ with $|y|\leq |y_0| < |x_0|\leq |x|$.
    We have
    \begin{eqnarray}\label{eq_transition_functions}
        \varphi_{x,y}^{i,j}(u) &=& \psi_{x,y;z,a}^{i,j}(\pi(u,\rho_{x,y}^{i+1}(u))) \\
        &=& \psi_{x,y;z,a}^{i,j}(\pi(u_b,\rho_{x_0,y_0}^{i+1}(u_b))) \nonumber \\ 
        &=& \mathrm{pr}_1 \big(  \psi_{x,y;z}^i|_{E|_{P_{x,y;z}^i\cap P_{x,y;z}^j}} \circ (  \psi_{x,y;a}^j|_{E|_{P_{x,y;z}^i\cap P_{x,y;z}^j}})^{-1}   \big) (e,\pi(u_b,\rho_{x_0,y_0}^{i+1}(u_b))) . \nonumber
    \end{eqnarray}
    By construction we have $P_{x_0,y_0;z}^i \subseteq P_{x,y;z}^i$ and $P_{x_0,y_0;a}^j\subseteq P_{x,y;a}^j$.
    Since the local trivialization $\psi^i_{x,y;z}\colon E|_{P_{x,y;z}^i}\to P_{x,y;z}^i\times G$ is the restriction of $\psi_z\colon E|_{R_z}\to R_z\times G$ and analogously for $\psi^j_{x_0,y_0;a}$ we have that
    $$       \psi_{x_0,y_0;z}^i =   \psi_{x,y;z}^i |_{E|_{P_{x_0,y_0;z}^i}}   \quad \text{and}\quad    \psi_{x_0,y_0;a}^j =   \psi_{x,y;a}^j |_{E|_{P_{x_0,y_0;a}^j}}  \, .    $$
    Inserting this into equation \eqref{eq_transition_functions} yields $\varphi_{x,y}^{i,j}(u) = \varphi_{x_0,y_0}^{i,j}(u_b)$ which shows condition (i).
\end{proof}

We now want to show that a pre-transport function induces a transport function.

\begin{prop}\label{prop_pre-transport_induces_transport}
    Let $(M,f,X)$ be Morse data and let $G$ be a topological group.
    A pre-transport function $\{\varphi_{x,y}^{i,j}\colon \widetilde{S}_{x,y}^{\, i,j}\to G\}$ induces a transport function $\Phi\colon \mathrm{mor}(\mathcal{M}_f)\to G$ such that $\{\varphi_{x,y}^{i,j}\}$ is adapted to $\Phi$.
\end{prop}
\begin{proof}
    Let $x,y\in\mathrm{Crit}(f)$ with $|y|\leq |x|$.
    If $x= y$ there is only the constant flow line $c_x\in \mathcal{L}(x,x)$ and we set $\Phi(c_x) = e\in G$.
    Assume now that $x\neq y$ and let $u\in \overline{\mathcal{L}}(x,y)$ be a flow line.
    Consider the maximal sequence of indices $|y| = j_0<j_1< \ldots < j_m = |x|$ such that $u\in \widetilde{S}_{x,y}^{\, j_k,j_{k+1}}$ for $k = 0,1,\ldots, m-1$.
    From the definition it follows that if $u\in \widetilde{S}_{x,y}^{\, i,j}$ for some indices $i,j\in \{|y|,|y|+1,\ldots, |x|\}$, then either $i = j_p, j = j_{p+1}$ for some $p\in \{0,\ldots, m-1\}$ or $(i,j) = (|y|,|x|)$.   
    Define 
    $$    \Phi(u) =  \varphi_{x,y}^{j_{m-1},j_m}(u) \cdot \varphi_{x,y}^{j_{m-2},j_{m-1}}(u) \cdot \ldots \cdot  \varphi_{x,y}^{j_{0},j_1}(u).   $$
    For the continuity, let $(u_n)_{n\in\mathbb{N}}$ be a sequence in $\overline{\mathcal{L}}(x,y)$ and assume that $\lim_{n\to\infty} u_n = u\in \overline{\mathcal{L}}(x,y)$.
    First, assume that for all $n\in\mathbb{N}$ the flow lines $u_n$ have the maximal sequence of indices $|y| = j_0<j_1< \ldots < j_m = |x|$, i.e. $u_n\in \widetilde{S}_{x,y}^{\, j_k,j_{k+1}}$ for $k = 0,\ldots, m-1$ for all $n\in\mathbb{N}$.
   By continuity of the separation functions the maximal sequence of indices for $u$ is a subset of the indices for the $u_n$, i.e. there is a sequence $0 = k_0 < k_1 < \ldots < k_{\ell} = m$ such that $|y| = j_{k_0} < j_{k_1} < \ldots < j_{k_{\ell}} = |x|$ is the maximal sequence for $u$.
    Therefore, by definition 
    $$   \Phi(u) =  \varphi_{x,y}^{j_{k_{\ell-1}},j_m}(u) \cdot \varphi_{x,y}^{j_{k_{\ell-2}},j_{k_{\ell-1}}}(u) \cdot \ldots \cdot  \varphi_{x,y}^{j_{0},j_{k_1}}(u)  . $$
    Since $\varphi$ is a pre-transport function we have
    $$   \lim_{n\to\infty}   \varphi_{x,y}^{j_{k_i-1},j_{k_i}}(u_n) \cdot \ldots \cdot \varphi_{x,y}^{j_{k_{i-1}},j_{k_{i-1}+1}}(u_n) = \varphi_{x,y}^{j_{k_{i-1}} j_{k_i} } (u)    $$
    for all $i = 0,\ldots ,\ell-1$.
    Therefore we have
    \begin{eqnarray*}
            \lim_{n\to\infty} \Phi(u_n)  &=&  \lim_{n\to\infty} \varphi_{x,y}^{j_{m-1},j_m}(u_n) \cdot \varphi_{x,y}^{j_{m-2},j_{m-1}}(u_n) \cdot \ldots \cdot  \varphi_{x,y}^{j_{0},j_1}(u_n)     \\
            &=&  \lim_{n\to\infty}( \varphi_{x,y}^{j_{m-1},j_m}(u_n) \cdot \ldots \cdot \varphi_{x,y}^{j_{k_{\ell-1}},j_{k_{\ell-1}+1}}(u_n)) \cdot \ldots   \\  &  & \hphantom{ergoi }  \cdot \lim_{n\to\infty} (\varphi_{x,y}^{j_{k_1 -1},j_{k_1}}(u_n)\cdot \ldots \cdot \varphi_{x,y}^{j_0,j_1}(u_n)) \\
            &=&   \varphi_{x,y}^{j_{k_{\ell-1}},j_m}(u) \cdot \varphi_{x,y}^{j_{k_{\ell-2}},j_{k_{\ell-1}}}(u) \cdot \ldots \cdot  \varphi_{x,y}^{j_{0},j_{k_1}}(u) = \Phi(u)       .    
    \end{eqnarray*}
    More generally, if $(u_n)_{n\to\infty}$ is an arbitrary converging sequence in $\overline{\mathcal{L}}(x,y)$, we obtain a decomposition into finitely many subsequences, each of which has the same maximal sequence of indices for all elements of this subsequence.
    The above argument can be repeated for each subsequence to show that the limit of $\Phi(u_n)$ exists and agrees with $\Phi(u)$.

    We still need to show that $\Phi$ is a transport function, i.e. that it respects compositions of flow lines.
    Let $u = u_1\circ u_2$ for $u_1\in \overline{\mathcal{L}}(x,z)$ and $u_2\in \overline{\mathcal{L}}(z,y)$.
    Let $|y| = j_0 < j_1< \ldots <j_m = |z|$ as well as 
    $|z|=  \ell_0 < \ell_1 < \ldots < \ell_n = |x|$ be the maximal sequences of indices for $u_1$ and $u_2$, respectively.
    Then the sequence $j_0 < j_1 < \ldots  < j_m = \ell_0 < \ell_1 < \ldots < \ell_n$ is the maximal sequence for the flow line $u$.
    We have
    \begin{eqnarray*}
        \Phi_{x,y}(u) &=&  \varphi_{x,y}^{\ell_{n-1},\ell_n}(u) \cdot \ldots \cdot   \varphi_{x,y}^{\ell_0,\ell_1}(u)  \cdot \varphi_{x,y}^{j_{m-1},j_m}(u)  \cdot \ldots \cdot  \varphi_{x,y}^{j_0,j_1}(u)  
        \\
        &=&  \varphi_{x,z}^{\ell_{n-1},\ell_n}(u_1)  \cdot \ldots \cdot   \varphi_{x,z}^{\ell_0,\ell_1}(u_1)  \cdot \varphi_{z,y}^{j_{m-1},j_m}(u_2)  \cdot \ldots \cdot  
        \varphi_{z,y}^{j_0,j_1}(u_2)
        \\
        &=& \Phi_{x,z}(u_1) \cdot \Phi_{z,y}(u_2) 
    \end{eqnarray*}
    by condition (i) of the pre-transport functions.
    This shows that $\Phi$ is a transport function and it is adapted to $\{\varphi_{x,y}^{i,j}\}$ by construction.
\end{proof}

\subsection{From transport functions to principal bundles}

In this subsection we want to see that one can recover the isomorphism class of a $G$-principal bundle from the transport function.
The idea is that a $G$-valued transport function induces a $G$-cocycle and thus a $G$-principal bundle.
We will then show that the transport function which we defined above using the local trivializations induces a bundle which is isomorphic to the one that we started with.

\begin{theorem}\label{theorem_main_result_transport-functions}
    Let $(M,f,X)$ be Morse data and let $q\colon E\to M$ be a $G$-principal bundle.
    There is an induced transport function $\Phi_q\colon \mathrm{mor}(\mathcal{M}_f)\to G$.
    Conversely, any transport function $\Psi\colon \mathrm{mor}(\mathcal{M}_f)\to G$ yields a $G$-principal bundle $p\colon E^{\Psi}\to M$ such that the principal bundle $p\colon E^{\Phi_q}\to M$ associated to the transport function $\Phi_q$ is isomorphic to the initial bundle $q\colon E\to M$.
\end{theorem}
The rest of this subsection is devoted to the proof of this Theorem.
The first goal is to better understand the intersections of the sets $V_a$, $a\in\mathrm{Crit}(f)$.
Note that $V_a\cap V_b\neq \emptyset$ implies that $|a| \neq |b|$, hence intersections can only happen for the open sets $V_x$ for critical points with different index.

In Section \ref{subsec_3-2} we related the intersection $P_{x,y}^i \cap P_{x,y}^j$ to the set of flow lines $\widetilde{S}_{x,y}^{i,j}$.
We now introduce an analogue for the open sets $U_{x,y}^i\cap U_{x,y}^j$.
For $(i,j)\neq (|y|,|x|)$ we define
\begin{eqnarray*}
        \widetilde{W}_{x,y}^{i,j} = \{ (\widehat{u},a)\in U_{x,y} \,|\, \rho_{x,y}^j(\widehat{u},a) -\epsilon_0 < \rho_{x,y}^{i+1}(\widehat{u},a) + \epsilon_0 , \,\, \rho_{x,y}^i(\widehat{u}) < \rho_{x,y}^{i+1}(\widehat{u}) , \\
        \rho_{x,y}^j(\widehat{u}) < \rho_{x,y}^{j+1}(\widehat{u}), \,\, a < \tfrac{4}{3}\epsilon \}
\end{eqnarray*}
and for $(i,j) = (|y|,|x|)$ we set
$$    \widetilde{W}_{x,y}^{|y|,|x|} = \{u\in \overline{\mathcal{L}}(x,y) \,|\, \rho_{x,y}^{|x|}(u) -\epsilon_0 < \rho_{x,y}^{|y|+1}(u) + \epsilon_0 \} .            $$

\begin{lemma}\label{lemma_maximal_sequences}
    Let $x,y\in\mathrm{Crit}(f)$ be critical points and let $i,j\in \{|y|,\ldots ,|x|\}$ with $i< j$.
    \begin{enumerate}
        \item 
        The sets $\widetilde{W}_{x,y}^{i,j}\subseteq \overline{\mathcal{L}}(x,y)$ are open and if $p = \pi(u,s)\in U_{x,y}^i \cap U_{x,y}^j$, then $u\in \widetilde{W}_{x,y}^{i,j}$.
        \item If $\widehat{u}\in\partial\overline{\mathcal{L}}(x,y)$ with $\rho_{x,y}^i (\widehat{u}) < \rho_{x,y}^{i+1}(\widehat{u}) = \rho_{x,y}^{j}(\widehat{u})< \rho_{x,y}^{j+1}(\widehat{u})$, then $(\widehat{u},a)\in \widetilde{W}_{x,y}^{i,j}$ for all $ a\in[0,\tfrac{4}{3}\epsilon)$.
        \item Let $u = (\widehat{u},a)\in U_{x,y}$ and assume that $|y| = i_0 < i_1 < \ldots < i_m = |x|$ is the maximal sequence of indices such that $\rho_{x,y}^{i_r}(\widehat{u}) < \rho_{x,y}^{i_{r}+1}(\widehat{u}) = \rho_{x,y}^{i_{r+1}}(\widehat{u}) < \rho_{x,y}^{i_{r+1}+1}(\widehat{u})$. 
        Then $\widehat{u}$ has the maximal decomposition $\widehat{u} = u_1\circ \ldots \circ u_k$ with $u_1\in\overline{\mathcal{L}}(x,z_1),u_2\in\overline{\mathcal{L}}(z_1,z_2),\ldots , u_k\in \overline{\mathcal{L}}(z_{k-1},y)$ and there are indices $m > r_1 > \ldots >  r_{k-1}> 0 $ such that $|z_{\ell}|  = i_{r_{\ell}}$ for $\ell\in \{1,\ldots, k-1\}$.
        \item It holds that $\widetilde{S}_{x,y}^{i,j}\subseteq \widetilde{W}_{x,y}^{i,j}$ where $\widetilde{S}_{x,y}^{i,j}$ is the set defined in equation \eqref{eq_sxyij}.
    \end{enumerate}
\end{lemma}
\begin{proof}
    The first and the second claim are direct verifications using the definitions.
    For the third part let $u = (\widehat{u},a)$ and assume that $\widehat{u} = u_1\circ \ldots \circ u_k$ is the maximal decomposition with $u_1\in\overline{\mathcal{L}}(x,z_1),u_2\in\overline{\mathcal{L}}(z_1,z_2),\ldots , u_k\in \overline{\mathcal{L}}(z_{k-1},y)$.
	We need to show that $|z_i| \in \{i_1,\ldots, i_{m-1}\}$ for $i = 1,\ldots, k-1$.
	Assume this is not true, i.e. there is an $i\in \{1,\ldots, k-1\}$ and an $r\in \{0,\ldots, m-1\}$ such that $i_r <|z_i| < i_{r+1}$ and hence $i_r +1 \leq |z_i| < i_{r+1}$.
    We thus have
    $$  \rho_{x,y}^{i_{r +1}}(\widehat{u}) = \rho_{z_{i-1},z_i}^{i_{r+1}}(u_{i-1}) > |z_i|     $$
    and on the other hand
    $$    \rho_{x,y} {i_r +1}(\widehat{u}) = \rho_{z_i,z_{i+1}}^{i_r+1}(u_{i+1}) < |z_i| .$$
    However, by assumption we have $\rho_{x,y}^{i_r+1}(\widehat{u}) = \rho_{x,y}^{i_{r+1}}(\widehat{u})$ which is a contradiction to the inequalities above.
    Hence, the maximal decomposition of $\widehat{u}$ has the claimed form.	
\end{proof}

Let $a,b\in \mathrm{Crit}(f)$ with  $i = |a|$, $j = |b|$ and $i< j$.
We choose trivializations $\psi_a \colon E|_{V_a}\to G\times V_a$ and define
$$  \psi_{a,b}\colon V_a\cap V_b \to G    $$
to be the change of trivializations, i.e. explicitly we have
$$   \psi_{a,b}(p)  = \mathrm{pr}_2 \big(  \psi_a|_{E_{V_a\cap V_b}} \circ \psi_b|_{E|_{V_a}\cap E|_{V_b}}^{-1}   \big) (e,p)  .  $$
This constitutes a cocycle $\{\psi_{a,b}\}$ and the $G$-principal bundle $E^{\psi}$ induced by this cocycle is isomorphic to the original bundle $E$.
For $U_{x,y;a}^i \cap U_{x,y;b}^j$ we define $\psi_{x,y}^{i,j}\colon U_{x,y;a}^i \cap U_{x,y;b}^j\to G$ by restricting $\psi_{a,b}$.
This defines a continuous map $\psi_{x,y}^{i,j}\colon U_{x,y}^i\cap U_{x,y}^j\to G$.
Further we define maps
$$   \mathring{\omega}_{x,y}^{i,j}\colon \widetilde{W}_{x,y}^{i,j}\to G     $$
recursively as follows.
For $j = i+1$ with $\widetilde{W}_{x,y}^{i,j} \neq \emptyset$ we set
$$     \mathring{\omega}_{x,y}^{i,j}(u) = \psi_{x,y}^{i,j}(\pi(u, \rho_{x,y}^{i+1}(u)))  .  $$
Now, let $ j> i+1$ and assume that all maps $\mathring{\omega}_{x,y}^{k,\ell}$ have been defined for $\ell - k < j-i$.
Let $u\in \widetilde{W}_{x,y}^{i,j}$ and let $|y| = i_0 < i_1 <  \ldots < i_m = |x| $ be the maximal sequence for $u$ as in Lemma \ref{lemma_maximal_sequences}.(3).
If $\rho_{x,y}^{i+1}(u) = \rho_{x,y}^j(u)$, then we have $i = i_r$ and $j= i_{r+1}$ for some $r\in \{0,\ldots, m-1\}$ and we set
$$   \mathring{\omega}_{x,y}^{i,j}(u) = \psi_{x,y}^{i,j}(\pi(u, \rho_{x,y}^{i+1}(u))) .     $$
Otherwise, let $i= i_r <  \ldots < i_{r+\ell} = j$ and set
$$    \mathring{\omega}_{x,y}^{i,j}(u) =  \mathring{\omega}_{x,y}^{i_{r+\ell-1},j}(u) \cdot \ldots \cdot \mathring{\omega}_{x,y}^{i,i_{r+1}}(u) .     $$
Note that this means in particular that 
$$   \mathring{\omega}_{x,y}^{|y|,|x|}(u) =  \begin{cases}
    \psi_{x,y}^{|y|,|x|}(\pi (u, \tfrac{|y|+|x|}{2})), & u\not\in V_{x,y} \\
    \mathring{\omega}_{x,y}^{i_{m-1},|x|}(u)   \cdot \ldots \cdot \mathring{\omega}_{x,y}^{|y|,i_1}(u),& \text{else}
\end{cases}       $$
for $u\in \widetilde{W}_{x,y}^{|y|,|x|}$.
Finally, we define $\omega_{x,y}^{i,j}\colon U_{x,y}^i\cap U_{x,y}^j\to G$ with $\omega_{x,y}^{i,j}(\pi(u,s))) = \mathring{\omega}_{x,y}^{i,j}(u)$.

\begin{definition}
    Let $(M,f,X)$ be Morse data and let $G$ be a topological group.
    A family of continuous maps $\varphi_{x,y}^{i,j} \colon \widetilde{W}_{x,y}^{i,j} \to G  $ for all pairs of critical points $x,y\in\mathrm{Crit}(f)$ with $|y|<|x|$ and for all $i,j\in\{ |y|,\ldots, |x|\}$ with $i<j$ is called \emph{extended pre-transport function} if the following conditions are satisfied.  
    \begin{enumerate}[(i)]
        \item If $u = u_a\circ u_b\circ u_c$ with $u_a\in \overline{\mathcal{L}}(x,x_0)$, $u_b\in\overline{\mathcal{L}}(x_0,y_0)$ and $u_c\in \overline{\mathcal{L}}(y_0,y)$ and $i,j\in \{|y_0|,\ldots, |x_0|\}$ and $u_b\in \widetilde{W}_{x_0,y_0}^{i,j}$, then 
        $$   \varphi_{x,y}^{i,j}(u) = \varphi_{x_0,y_0}^{i,j}(u_b) .   $$
        \item Let $u\in\overline{\mathcal{L}}(x,y)$ and $i,j,k\in\{|y|,\ldots, |x|\}$ with $i<j<k$.
        If $u\in \widetilde{W}_{x,y}^{i,j} \cap \widetilde{W}_{x,y}^{j,k} \cap \widetilde{W}_{x,y}^{i,k}$ then
        $$     \varphi_{x,y}^{i,k}(u)  = \varphi_{x,y}^{j,k}(u) \cdot \varphi_{x,y}^{i,j}(u)      . $$
    \end{enumerate}
\end{definition}

\begin{lemma}
    The maps $\mathring{\omega}_{x,y}^{i,j}$ form an extended pre-transport function.
\end{lemma}
\begin{proof}
    We need to show that conditions (i) and (ii) of an extended pre-transport function are satisfied as well as the continuity of the $\mathring{\omega}_{x,y}^{i,j}$.
    For the first condition, let $u = u_a\circ u_b\circ u_c\in \overline{\mathcal{L}}(x,y)$ with $u_b\in\overline{\mathcal{L}}(x_0,y_0)$ and such that $u_b\in \widetilde{W}_{x_0,y_0}^{i,j}$. 
    Assume that $\rho_{x,y}^{i+1}(u) = \rho_{x,y}^j(u)$ and note that by construction we have $\rho_{x,y}^k(u) =\rho_{x_0,y_0}^k(u_b)$ for $k\in \{|y_0|+1,\ldots ,|x_0|\}$.
    This implies that $\pi(u,\rho_{x,y}^{i+1}(u))= \pi(u_b,\rho_{x_0,y_0}^{i+1}(u_b))$.
    Consequently, 
    $$    \mathring{\omega}_{x,y}^{i,j}(u) = \psi_{x,y}^{i,j}(\pi(u,\rho_{x,y}^{i+1}(u))) = \psi_{x_0,y_0}^{i,j}(\pi(u_b,\rho_{x_0,y_0}^{i,j}(u_b))) = \mathring{\omega}_{x_0,y_0}^{i,j}(u_b) .     $$
    Condition (i) for other indices $i,j\in \{|x_0|+1,\ldots, |y_0|\}$ then follows immediately.
    The second condition follows from the cocycle condition and the construction.
    The continuity can then be shown with the same ideas as in the proof of Lemma \ref{lemma_trivializations_induce_pre-transport}.
\end{proof}

\begin{lemma}\label{lemma_extended_pre-transport_induces_cocycle}
    An extended pre-transport function $\{\mathring{\tau}_{x,y}^{i,j}\colon \widetilde{W}_{x,y}^{i,j}\to G\}$ induces a cocycle $\{\tau_{a,b}\colon V_a\cap V_b\to G\}$.
\end{lemma}
\begin{proof}
    Let $a,b\in\mathrm{Crit}(f)$ with $V_a\cap V_b \neq \emptyset$.
    The non-empty intersection implies that $|a| \neq |b|$ so without loss of generality we will assume $|a|<|b|$.
    Recall that $V_a = \bigcup U_{x,y;a}^{|a|}$ and similarly for $V_b$.  
    We therefore write $V_a = U_1\cup \ldots \cup U_{\ell}$ and $V_b = W_1\cup \ldots \cup W_m$ such that
     $    U_i = U_{x_i,y_i;a}^{|a|} $ and $ W_i = U_{x_i,y_i;b}^{|b|}      $
     for $i\in \{1,\ldots, k\}$ for some $k\leq \max\{\ell,m\}$
    and such that for $j > k$ we have $U_j = U_{x_j,y_j;a}^{|a|}$ and $W_j = U_{\overline{x}_j,\overline{y}_j;b}^{|b|}$ with $(x_j,y_j)\neq (\overline{x}_m,\overline{y}_m)$ for all $j\in \{1,\ldots, \ell\}$ and $m\in\{\ell+1,\ldots, m\}$.
     The intersection $V_a\cap V_b$ satisfies
     $$  V_a \cap V_b = \bigcup_{ i= 1}^k U_{x_i,y_i;a}^{|a|}\cap U_{x_i,y_i;b}^{|b|} .     $$
    Recall from Lemma \ref{lemma_maximal_sequences} that if $p = \pi(u,s)\in U_{x_i,y_i;a}^{|a|}\cap U_{x_i,y_i;b}^{|b|}$ then $u\in \widetilde{W}_{x_i,y_i}^{|a|,|b|}$.
    We define $\omega_{a,b}\colon V_a\cap V_b \to G$ by setting $\tau_{a,b}(p) = \mathring{\tau}_{x_i,y_i}^{|a|,|b|}(u)$ for $p = \pi(u,s)\in U_{x_i,y_i;a}^{|a|}\cap U_{x_i,y_i;b}^{|b|}$.
    We claim that this is a well-defined and continuous map and that the collection $\{\tau_{a,b}\}$ satisfies the cocycle condition.
    For the well-definedness let $p\in (U_{x_i,y_i;a}^{|a|} \cap U_{x_i,y_i;b}^{|b|}) \cap  (U_{x_j,y_j;a}^{|a|} \cap U_{x_j,y_j;b}^{|b|})$ and write $p = \pi(u_i,s) = \pi(u_j,s)$ for $u_i\in \overline{\mathcal{L}}(x_i,y_i)$, $u_j\in\overline{\mathcal{L}}(x_j,y_j)$ and $s\in \mathbb{R}$.
    We need to show that $\mathring{\tau}_{x_i,y_i}^{|a|,|b|}(u_i ) = \mathring{\tau}_{x_j,y_j}^{|a|,|b|}(u_j)$.
    We can write $u_i = w_1\circ w_2\circ w_3$ with $w_2\in \mathcal{L}(c,d)$ unbroken and such that $|d| < s < |c|$.
    Similarly, we write $u_j = v_1\circ v_2\circ v_3$ with $v_2\in \mathcal{L}(\overline{c},\overline{d})$ with $|\overline{d}| < s < |\overline{c}|$.
    Since $$\pi(w_2,s) = \pi(u_i,s)  = p  = \pi(u_j,s) = \pi(v_2,s) $$ we have $w_2 = v_2$ and thus $(c,d) = (\overline{c},\overline{d})$.
    Since $\mathring{\tau}_{x,y}^{i,j}$ is an extended pre-transport function we have $\mathring{\tau}_{x_i,y_i}^{|a|,|b|}(u_i) = \mathring{\tau}_{c,d}^{|a|,|b|}(w_2) = \mathring{\tau}_{x_j,y_j}^{|a|,|b|}(u_j)$ which we wanted to show for the well-definedness.
    With the same argument one can check the continuity of $\tau_{a,b}$ since all $\mathring{\tau}_{x_i,y_i}^{|a|,|b|}$ are continuous and as we have shown they glue together nicely.
    The cocycle condition can be seen as follows. 
    Assume that $V_a\cap V_b\cap V_c\neq \emptyset$ is a non-empty intersection.
    Then we have that $|a|,|b|$ and $|c|$ are pairwise distinct.
    The cocycle condition then follows from condition (ii) of an extended pre-transport function.    
\end{proof}

Next, we want to show that the cocycles $\{\psi_{a,b}\}$ and $\{\omega_{a,b}\}$ induce isomorphic $G$-principal bundles.
We will do this by constructing an isotopy between these two cocycles.

Let $x,y\in\mathrm{Crit}(f)$ and $i,j\in \{|y|,\ldots, |x|\}$ with $i < j$.
We shall define maps $\Omega_{x,y}^{i,j;t}\colon U_{x,y}^{i}\cap U_{x,y}^j \to G$ for $t\in [0,1]$ by a recursive definition over the difference $j-i$.
For $j = i+1$ we define
$$      \Omega_{x,y}^{i,j;t}(\pi(u,s)) = \psi_{x,y}^{i,j}(\pi( u ,t s + (1-t) \rho_{x,y}^{i+1}(u) ))  .     $$
Now, let $ j> i$ and assume that all maps $\mathring{\Omega}_{x,y}^{k,\ell;t}$ have been defined for $\ell - k < j-i$ and $t\in [0,1]$.
Let $p = \pi(u,s) \in U_{x,y}^i \cap U_{x,y}^j$.
If $\rho_{x,y}^{i+1}(u) = \rho_{x,y}^j(u)$, we set
$$   {\Omega}_{x,y}^{i,j;t}(\pi(u,s)) = \psi_{x,y}^{i,j}(\pi( u , t s + (1-t) \rho_{x,y}^{i+1}(u))) .     $$
Otherwise, let $i= i_0 < i_1 < \ldots < i_m = j$ be the maximal sequence of indices such that $\rho_{x,y}^{i_p}(u) < \rho_{x,y}^{i_p + 1}(u) = \rho_{x,y}^{i_{p+1}}(u) < \rho_{x,y}^{i_{p+1}+1}(u) $ for $p = 1,\ldots, m$.
Then we set
$$    \Omega_{x,y}^{i,j;t}(p) =  \Omega_{x,y}^{i_{m-1},j;t}(p) \cdot \ldots \cdot {\Omega}_{x,y}^{i,i_1;t}(p) .     $$

\begin{lemma}\label{lemma_first_isotopy}
    For 
    $t\in [0,1]$ the maps $\Omega_{x,y}^{i,j;t}$ determine a cocycle $\{\Omega_{a,b}^t\colon V_a\cap V_b \to G\}$ such that $\Omega_{a,b}^0 = \omega_{a,b}$ and $\Omega_{a,b}^1 = \psi_{a,b}$. Consequently, the cocycles $\{\omega_{a,b}\}$ and $\{\psi_{a,b}\}$ induce isomorphic $G$-principal bundles.
\end{lemma}
\begin{proof}
    One checks that $\Omega_{x,y}^{i,j;t}$ induces a cocycle $\Omega_{a,b}^t$ as in the proof of Lemma \ref{lemma_extended_pre-transport_induces_cocycle}. 
    The maps $\Omega_{a,b}^{\cdot}\colon V_a\cap V_b\times [0,1]\to G$ are cocycles on the space $M\times [0,1]$ with respect to the covering $\{V_a\times [0,1]\}_{a\in\mathrm{Crit}(f)}$ and they induce a $G$-principal bundle $E^{\Omega}\to M\times [0,1]$.
    The restrictions of $\Omega_{a,b}^{\cdot}$ to $V_a\cap V_b\times \{0\}$ and $V_a\cap V_b\times \{1\}$ yield the cocycles $\{\omega_{a,b}\}$ and $\{\psi_{a,b}\}$, respectively.
    By the homotopy invariance of $G$-principal bundles, it follows that the bundles $E|_{M\times \{0\}}\to M\times \{0\}\cong M$ and $E|_{M\times \{1\}}\to M\times \{1\}\cong M$ are isomorphic.
    Since these bundles are isomorphic to $E^{\omega}$ and $E^{\psi}$, respectively the claim of the lemma is shown.
\end{proof}

Next, we shall see that any transport function $\Phi\colon \mathrm{mor}(\mathcal{M}_f)\to G$ induces a $G$-valued cocycle and hence a $G$-principal bundle.

\begin{lemma}\label{lemma_transport_induces_ept}
    A transport function $\Phi\colon \mathrm{mor}(\mathcal{M}_f)\to G$ induces a cocycle $\{\eta_{a,b}\colon V_a\cap V_b\to G\}$ such that $\eta_{a,b}(p) = \Phi(u)$ whenever $u\in \overline{\mathcal{L}}(a,b)\setminus V_{a,b}$ is an unbroken flow line and $p = \pi(u, \tfrac{|a|+|b|}{2})$.
\end{lemma}
\begin{proof}
    Let $\Phi\colon \mathrm{mor}(\mathcal{M}_f)\to G$ be a transport function.
    We want to define an extended pre-transport function $\mathring{\eta}_{x,y}^{i,j}\colon \widetilde{W}_{x,y}^{i,j}\to G$.
    We will do so recursively over the index difference $|x|-|y|$.
    If $|x|-|y| = 1$, we set
    $$    \mathring{\eta}_{x,y}^{|y|,|x|} (u) = \Phi(u) \quad \text{for}\,\,\, u\in \widetilde{W}_{x,y}^{|y|,|x|} .    $$
    Assume that the maps $\mathring{\eta}_{x',y'}^{i,j}$ have been defined for all $x',y'\in\mathrm{Crit}(f)$ with $|x'|-|y'| \leq k$ and let $x,y\in\mathrm{Crit}(f)$ with $|x|-|y| = k+1$.
    If $U_{x,y} = \emptyset$, we only need to define $\mathring{\eta}_{x,y}^{|y|,|x|}\colon\widetilde{W}_{x,y}^{|y|,|x|} \to G$ and we do so by setting $\mathring{\eta}_{x,y}^{|y|,|x|}(u)  = \Phi(u)$.
    Now assume that $U_{x,y}\neq \emptyset $.
    For $i = |y|,j = |x|$ and $u\in \widetilde{W}_{x,y}^{|y|,|x|}$ we set   
    $    \mathring{\eta}_{x,y}^{|y|,|x|}( u) =\Phi(u)     $.
    Let $i,j\in \{|y|,\ldots ,|x|\}$ with $i<j$ and $(i,j)\neq (|y|,|x|)$ and let $u \in\widetilde{W}_{x,y}^{i,j}$ be a flow line.
    By Lemma \ref{lemma_maximal_sequences} we have a maximal sequence $|y| = i_0 < i_1 < \ldots < i_m = |x|$.
    First, assume that $i = i_p$ and $j = i_{p+1}$ for some $p > 0$.
    Again by Lemma \ref{lemma_maximal_sequences} there is a maximal decomposition  $\widehat{u} = u_1\circ \ldots \circ u_k$ with $u_1\in\overline{\mathcal{L}}(x,z_1),u_2\in\overline{\mathcal{L}}(z_1,z_2),\ldots , u_k\in \overline{\mathcal{L}}(z_{k-1},y)$ and there are indices $m > r_1 > \ldots >  r_{k-1}> 0 $ such that $|z_{\ell}|  = i_{r_{\ell}}$ for $\ell\in \{1,\ldots, k-1\}$.
    Set $x = z_0$ and $ y = z_{k+1}$.
    There is an $\ell\in \{0,\ldots, k\}$ such that $ |z_{\ell}| \geq  i_{p+1} > i_p \geq |z_{\ell+1}|$.
    Since $\rho_{x,y}^s(\widehat{u}) = \rho_{z_{\ell},z_{\ell+1}}^{s}(u_{\ell+1})$ for $s\in \{|z_{\ell+1}|+1,\ldots , |z_{\ell}|\}$ we have
    \begin{equation}\label{eq_separation_fct_uell}
            \rho_{z_{\ell},z_{\ell+1}}^{i_p} (u_{\ell+1})  < \rho_{z_{\ell},z_{\ell+1}}^{i_p+1}(u_{\ell+1})    =\rho_{z_{\ell},z_{\ell+1}}^{i_{p+1}}(u_{\ell+1}) < \rho_{z_{\ell},z_{\ell+1}}^{i_{p+1}+1}(u_{\ell+1})  .    
    \end{equation}
    If $i_p = |z_{\ell+1}|$ and $i_{p+1} = |z_{\ell}|$, then one checks that we must have $u_{\ell+1}\in \mathcal{L}(z_{\ell},z_{\ell+1}) \setminus \mathring{V}_{z_{\ell},z_{\ell+1}}$ and thus $u_{\ell+1}\in \widetilde{W}_{z_{\ell},z_{\ell+1}}^{|z_{\ell+1}|,|z_{\ell}|}$.
    In this case we set $\mathring{\eta}_{x,y}^{i_p,i_{p+1}}(u) = \mathring{\eta}_{z_{\ell},z_{\ell+1}}^{|z_{\ell+1}|,|z_{\ell}|}(u_{\ell+1}) = \Phi(u_{\ell+1})$.
    Otherwise, write $u_{\ell+1} = (\widehat{v},b)\in U_{z_{\ell},z_{\ell+1}}$ and let $|z_{\ell+1}| = j_0 < j_1 < \ldots< j_N = |z_{\ell}|$ be the maximal sequence for $u_{\ell}$ as in Lemma \ref{lemma_maximal_sequences}.
    From equation \eqref{eq_separation_fct_uell} and the explicit expression of the separation functions we conclude that $b< \epsilon$ and that
    $$\rho_{z_{\ell},z_{\ell+1}}^{i_p}(\widehat{v})  < \rho_{z_{\ell},z_{\ell+1}}^{i_p+1}(\widehat{v}) = \rho_{z_{\ell},z_{\ell+1}}^{i_{p+1}}(\widehat{v}) < \rho_{z_{\ell},z_{\ell+1}}^{i_{p+1}+1}(\widehat{v})  . $$
    It follows that $i_p = j_{p^*}$ and $i_{p+1} = j_{p^*+1}$ for some $0\leq p^*< N$.
    Hence, $u_{\ell+1}\in \widetilde{W}_{z_{\ell},z_{\ell-1}}^{i_p,i_{p+1}}$ by Lemma \ref{lemma_maximal_sequences}(ii) and we define $\mathring{\eta}_{x,y}^{i_p,i_{p+1}}(u) = \mathring{\eta}_{z_{\ell},z_{\ell+1}}^{i_p,i_{p+1}}(u_{\ell+1})$.
    In general if $i = i_s ,j = i_t$ with $s,t\in\{1,\ldots, m\}$ and $t-s \geq 2$, then we define
    $$  \mathring{\eta}_{x,y}^{i_s,i_t}(u) = \mathring{\eta}_{x,y}^{i_{t-1},i_t}(u)\cdot \ldots \cdot \mathring{\eta}_{x,y}^{i_s,i_{s+1}}(u) .      $$
    Finally, we define $\mathring{\eta}_{x,y}^{i_0,i_1}\colon \widetilde{W}_{x,y}^{i_0,i_1}\to G$ by
    $$    \mathring{\eta}_{x,y}^{i_0,i_1}(u) =    (\mathring{\eta}_{x,y}^{i_1,i_2}(u))^{-1} \cdot \ldots \cdot (\mathring{\eta}_{x,y}^{i_N,i_{N+1}}(u))^{-1}\cdot \Phi(u)  .     $$
    We need to check that this defines an extended pre-transport function.

    Condition (ii) of an extended pre-transport function holds by construction.
    It remains to show that condition (i) holds and that the maps $\mathring{\eta}_{x,y}^{i,j}$ are continuous.
    Let $u = u^a\circ u^b\in \partial\overline{\mathcal{L}}(x,y)$ with $u^a\in \overline{\mathcal{L}}(x,z)$ and $u^b\in\overline{\mathcal{L}}(z,y)$.
    As before let $|y| = i_0 < i_1 < \ldots  < i_N < i_{N+1} = |x|$ be the maximal sequence of indices for $u$.
    By Lemma \ref{lemma_maximal_sequences}.(iii) there is $k\in \{1,\ldots ,N\}$ such that $|z| = i_k$.
    If $u^a$ is unbroken, then a maximal decomposition of $u$ is given by $u = u^a\circ (v_1\circ \ldots \circ v_l)$ with $u^b = v_1\circ \ldots \circ v_{\ell}$ the maximal decomposition of $u_b$.
    For $r\geq k$ we have by definition that
    $$   \mathring{\eta}_{x,y}^{i_r,i_{r+1}}(u) = \mathring{\eta}_{x,z}^{i_r,i_{r+1}}(u_a)     $$
    and thus condition (i) holds in this case.
    If $u^a$ is broken and has maximal decomposition $u^a = w_1\circ \ldots \circ w_k$, then the maximal decomposition of $u$ is given by $u = w_1\circ \ldots \circ w_k\circ v_1\circ \ldots \circ v_{\ell}$.
    If $r\geq k$ then by construction we have
    $$  \mathring{\eta}_{x,y}^{i_r,i_{r+1}}(u) = \mathring{\eta}_{z_1,z_2}^{i_r,i_{r+1}}(v_s) = \mathring{\eta}_{x,z}^{i_r,i_{r+1}}(u^a)    $$
    where $v_s\in \mathcal{L}(z_1,z_2)$ is the flow line such that $|z_2| \leq i_r < i_{r+1}\leq |z_1|$.
    With similar arguments one checks that 
    $$    \mathring{\eta}_{x,y}^{i_r,i_{r+1}}(u) = \mathring{\eta}_{z,y}^{i_r,i_{r+1}}(u^b)     $$
    for $r\in \{1,\ldots,k-1\}$.
    For $r =0$ we have
    \begin{eqnarray*}
        \mathring{\eta}_{x,y}^{i_0,i_1}(u) &=&    (\mathring{\eta}_{x,y}^{i_1,i_2}(u))^{-1} \cdot \ldots \cdot (\mathring{\eta}_{x,y}^{i_N,i_{N+1}}(u))^{-1}\cdot \Phi(u^a \circ u^b) \\
        &=&   (\mathring{\eta}_{y,z}^{i_1,i_2}(u^b))^{-1} \cdot \ldots \cdot (\mathring{\eta}_{y,z}^{i_{k-1},i_k}(u^b))^{-1} \cdot (\mathring{\eta}_{x,z}^{i_{k},i_{k+1}}(u^a))^{-1} \cdot \ldots  \\ &  &  \cdot (\mathring{\eta}_{x,z}^{i_N,i_{N+1}}(u^a))^{-1}\cdot \Phi(u^a) \cdot \Phi(u^b) \\
        &=& (\mathring{\eta}_{y,z}^{i_1,i_2}(u^b))^{-1} \cdot \ldots (\mathring{\eta}_{y,z}^{i_{k-1},i_k}(u^b))^{-1} \cdot \Phi(u^b) \\
        &=& \mathring{\eta}_{y,z}^{i_0,i_1}(u^b)
    \end{eqnarray*}
    which we wanted to show.
    Here we used the multiplicativity of a transport function and the fact that 
    $$     \Phi(u^a ) = \mathring{\eta}_{x,z}^{i_N,i_{N+1}}(u^a) \cdot \ldots \cdot \mathring{\eta}_{x,z}^{i_0,i_1}(u^a)    .  $$
    Lastly, we need to show the continuity of the $\mathring{\eta}_{x,y}^{i,j}$.
    For $|x|-|y| = 1$ continuity is clear since $\mathcal{L}(x,y)$ is discrete in this case.
    Now assume that we have shown continuity for all pairs of critical points of index difference less than $ k$ and let $x,y\in\mathrm{Crit}(f)$ with $|x|-|y| = k+1$.
    Let $i,j\in \{|y|,\ldots, |x|\}$ with $i< j$ and consider $\mathring{\eta}_{x,y}^{i,j}\colon \widetilde{W}_{x,y}^{i,j}\to G$.
    Note that for $(i,j)= (|y|,|x|)$ we have $\mathring{\eta}_{x,y}^{|y|,|x|}(u) = \Phi(u)$ and thus the continuity of $\mathring{\eta}_{x,y}^{|y|,|x|}$ follows from the continuity of $\Phi$.
    If $(i,j)\neq (|y|,|x|)$ and $i > |y|$ then we note that $\mathring{\eta}_{x,y}^{i,j}$ satisfies 
    $$    \mathring{\eta}_{x,y}^{i,j} (u) = \big( \mathring{\eta}_{x,y}^{i,j}|_{\widetilde{W}_{x,y}^{i,j}\cap \partial\overline{\mathcal{L}}(x,y)}  \circ \mathrm{pr}\,\big)  (u)     $$
    for $\mathrm{pr}\colon U_{x,y}\to \partial\overline{\mathcal{L}}(x,y)$ the projection of the collar neighborhood.
    Hence, the continuity of $\mathring{\eta}_{x,y}^{i,j}$ will follow from the continuity $\mathring{\eta}_{x,y}^{i,j}$ on the boundary.
    
    Let $(u^n)_{n\in\mathbb{N}}$ be a convergent sequence in $\widetilde{W}_{x,y}^{i,j}\cap \partial\overline{\mathcal{L}}(x,y)$.
    Assume for the moment that for all $u^n$ we have $u^n\in\mathcal{L}(x,z_1)\times \mathcal{L}(z_1,z_2)\times \ldots \times \mathcal{L}(z_k,y)$ and that all the $u^n$ have the maximal sequence of indices $|y| = i_0 < i_1 < \ldots < i_m = |x|$.
    Denote the limit by $u_* = \lim_{n\to\infty} u^n \in \widetilde{W}_{x,y}^{i,j}\cap \partial \overline{\mathcal{L}}(x,y)$.
    In particular we have $u_* \in \overline{\mathcal{L}}(x,z_1)\times \ldots \times \overline{\mathcal{L}}(z_k,y)$.
    By the continuity of the separation functions we have that the maximal sequence of indices for $u_*$, which we denote by $|y| = j_0 < j_1 < \ldots < j_r = |x|$, is a subsequence of $(i_0,\ldots, i_m)$.
    Let $i= i_p$ and $j = i_{p+1}$ for some $p\in \{1,\ldots, m-1\}$.
    Since $u_*\in \widetilde{W}_{x,y}^{i,j}$ we have that $i_p$ and $i_{p+1}$ are in the maximal sequence of indices for $u_*$, i.e. there is $s\in \{1,\ldots, r\}$ such that $i= j_s$ and $j = j_{s+1}$.
    By construction we have $\mathring{\eta}_{x,y}^{i,j}(u^n) = \mathring{\eta}_{z_{\ell},z_{\ell+1}}(u^n_{\ell+1})$ for some $\ell\in \{1,\ldots ,k\}$.
    Note that we have $u^n_{\ell+1}\to v_*\in \overline{\mathcal{L}}(z_{\ell},z_{\ell+1})$.
    By the inductive hypothesis the map $\mathring{\eta}_{z_{\ell},z_{\ell+1}}^{i,j}$ is continuous and therefore $\mathring{\eta}_{x,y}^{i,j}(u^n)\to \mathring{\eta}_{z_{\ell},z_{\ell+1}}^{i,j}(v_*)$ for $n\to \infty$.
    It remains to show that $\mathring{\eta}_{z_{\ell},z_{\ell+1}}^{i,j}(v_*) = \mathring{\eta}_{x,y}^{i,j}(u_*)$.
    We can write $v_* = v_1\circ \ldots \circ v_{\nu}$ for $v_1\in {\mathcal{L}}(z_{\ell},a_1)$, $v_2\in \mathcal{L}(a_1,a_2),\ldots v_{\nu}\in \mathcal{L}(a_{\nu-1},z_{\ell+1})$.
    Since $i =j_s$ and $j = j_{s+1}$ are in the maximal sequence of indices for $u_*$ we must have by Lemma \ref{lemma_maximal_sequences}.(iii) that there is $\tau\in \{1,\ldots, \nu\}$ such that $|a_{\mu}| \geq j_{s+1}> j_s \geq |a_{\mu+1}|$ for some $\mu \in \{0,\ldots, \nu-1\}$.
    Here, we set $a_0 = z_{\ell}$ and $a_{\nu} = z_{\ell+1}$.
    Therefore, we have by construction that $\mathring{\eta}_{x,y}^{i,j}(u_*) = \mathring{\eta}_{a_{\mu},a_{\mu+1}}^{i,j}(v_{\mu})$.
    By condition (i) of an extended pre-transport function we have $\mathring{\eta}_{z_{\ell},z_{\ell+1}}(v_* ) = \mathring{\eta}_{a_{\mu},a_{\mu+1}}^{i,j}(v_{\mu})$ and therefore $\mathring{\eta}_{x,y}^{i,j}(u_*) = \mathring{\eta}_{z_{\ell},z_{\ell+1}}^{i,j}(v_*)$ which we wanted to show.

    For an arbitrary $(u^n)_{n\in\mathbb{N}}$ convergent sequence in $\widetilde{W}_{x,y}^{i,j}\cap \partial\overline{\mathcal{L}}(x,y)$ one argues similarly by passing to subsequences which have fixed maximal sequences and which lie in fixed strata of $\partial\overline{\mathcal{L}}(x,y)$.

    Hence, the maps $\mathring{\eta}_{x,y}^{i,j}$ give an extended pre-transport function and by Lemma \ref{lemma_extended_pre-transport_induces_cocycle} this induces a cocycle $\{\eta_{a,b}\colon V_a\cap V_b\to G\}_{a,b\in\mathrm{Crit}(f)}$.
    It is a direct verification that the cocycle $\{\eta_{a,b}\}$ has the property stated in the claim of the lemma.
\end{proof}

We now consider the transport function $\Phi\colon \mathrm{mor}(\mathcal{M}_f)\to G$ which is induced by the local trivializations $\psi_a\colon E|_{V_a}\to G\times V_a$ as in Section \ref{subsec_3-2}.
Recall that for a flow line $u\in \widetilde{W}_{x,y}^{|y|,|x|}$ we have 
\begin{equation}\label{eq_ept_are_already_equal_for_Certain_sets}
\mathring{\eta}_{x,y}^{|y|,|x|}(u) = \Phi(u) = \psi_{x,y}^{|y|,|x|}(\pi(u,\tfrac{|x|+|y|}{2})) = \mathring{\omega}_{x,y}^{|y|,|x|}(u) .
\end{equation}
Hence, the extended pre-transport functions $\{\mathring{\eta}_{x,y}^{i,j}\}$ and $\{\mathring{\omega}_{x,y}^{i,j}\}$ already agree on a large part of the moduli spaces of flow lines.
In order to complete the proof of Theorem \ref{theorem_main_result_transport-functions} the goal is now to show that the $G$-valued cocycles associated to $\{\mathring{\eta}_{x,y}^{i,j}\}$ and $\{\mathring{\omega}_{x,y}^{i,j}\}$, respectively induce isomorphic $G$-principal bundles.

\begin{lemma}\label{lemma_second_isotopy}
    Let $\Phi\colon \mathrm{mor}(\mathcal{M}_f)\to G$ be the transport function induced by the local trivializations $\psi_a\colon E|_{V_a}\to G\times V_a$.
    Let $\{{\omega}_{a,b}\}$ be the cocycle as before and let $\{\eta_{a,b}\}$ be the cocycle induced from $\Phi$ as in Lemma \ref{lemma_transport_induces_ept}.
    There is an isotopy between the cocycles $\{\eta_{a,b}\}$ and $\{\omega_{a,b}\}$.
\end{lemma}
\begin{proof}
    We construct an isotopy of extended pre-transport functions by which we mean maps $\chi_{x,y}^{i,j;t}\colon \widetilde{W}_{x,y}^{i,j}\to G$ for $t\in [0,n]$ such that for each $t\in [0,n]$ the maps $\chi_{x,y}^{i,j;t}$ satisfy the conditions of an extended pre-transport function.
    It is easy to see that this yields an isotopy of cocycles between the cocycles defined by $\chi_{x,y}^{i,j;0}$ and $\chi_{x,y}^{i,j;n}$.

    We define the maps $\chi_{x,y}^{i,j;t}$ recursively over the index difference.
    Let $x,y\in\mathrm{Crit}(f)$ with $|x|-|y| = 1$.
    For $t\in [0,n]$ with $n = \mathrm{dim}(M)$ we set $\chi_{x,y}^{|y|,|x|;t}\colon \widetilde{W}_{x,y}^{|y|,|x|}\to G$ to be the map $\chi_{x,y}^{|y|,|x|;t}(u)  = \Phi(u)$, constant in $t$.
    Note that this implies $\chi_{x,y}^{i,j;0}(u) = \mathring{\omega}_{x,y}^{i,j}(u)$ and $\chi_{x,y}^{i,j;n}(u) = \mathring{\eta}_{x,y}^{i,j}(u)$ by equation \eqref{eq_ept_are_already_equal_for_Certain_sets}.

    Assume now that for all critical points $x,y\in\mathrm{Crit}(f)$ with $|x|-|y|\leq k$ we have constructed a continuous map $\chi_{x,y}^{i,j;t}\colon \widetilde{W}_{x,y}^{i,j}\to G$ for $t\in [0,n]$ such that $\chi_{x,y}^{i,j;t}$ is independent of $t$ for $t\in [0, n-k]$ and such that $\chi_{x,y}^{i,j;n-k}(u) = \mathring{\omega}_{x,y}^{i,j}(u)$ as well as $\chi_{x,y}^{i,j;n}(u) = \mathring{\eta}_{x,y}^{i,j}(u)$.
    Let $x,y\in\mathrm{Crit}(f)$ with $|x|-|y| = k+1$ and let $i,j\in\{|y|,\ldots,|x|\}$ with $i< j$.
    If $(i,j) = (|y|,|x|)$ we define $\chi_{x,y}^{|y|,|x|;t}(u) = \Phi(u)$ and again we have $\chi_{x,y}^{|y|,|x|;n-k-1}= \mathring{\omega}_{x,y}^{|y|,|x|}$ and $\chi_{x,y}^{|y|,|x|,n} = \mathring{\eta}_{x,y}^{|y|,|x|}$ by equation \eqref{eq_ept_are_already_equal_for_Certain_sets}.
    Moreover, $\chi_{x,y}^{|y|,|x|;t}$ is clearly constant in $t$ for $t\in [0,n-k-1]$.
    
    If $(i,j)\neq (|y|,|x|)$, let $u\in \widetilde{W}_{x,y}^{i,j}$ and let $|y| = i_0 < i_1 < \ldots < i_m = |x|$ be the maximal sequence of indices in the sense of Lemma \ref{lemma_maximal_sequences}.
    Assume that $i =i_p$ and $j = i_{p+1}$ and let $u = (\widehat{u},a)\in\widetilde{W}_{x,y}^{i_p,i_{p+1}}$.
    If the maximal decomposition of $\widehat{u}$ is given by $\widehat{u} = u_1 \circ \ldots \circ u_k$, then as in the proof of Lemma \ref{lemma_transport_induces_ept} there is $u_{\ell}\in \overline{\mathcal{L}}(z_{\ell-1},z_{\ell})$ with $u_{\ell}\in \widetilde{W}_{z_{\ell-1},z_{\ell}}^{i_p,i_{p+1}}$.
    We define $\chi_{x,y}^{i_p,i_{p+1};t}(u)$ to be equal to $\mathring{\omega}_{x,y}^{i_p,i_{p+1}}(u)$ for $t\in [0,n-k-1]$.
    For $t\in [n-k-1,n-k]$ we define
    $$  \chi_{x,y}^{i_p,i_{p+1};t}(u) = \psi_{x,y}^{i_p,i_{p+1}}((\widehat{u}, -at + a(n-k)), \rho_{x,y}^{i_p+1}(\widehat{u},-at + a(n-k))) .    $$
    Note that 
    $$\chi_{x,y}^{i_p,i_{p+1};n-k}(u) = \psi_{x,y}^{i_p,i_{p+1}}(\widehat{u},\rho_{x,y}^{i_p+1}(\widehat{u}))= \psi_{z_{\ell},z_{\ell-1}}^{i_p,i_{p+1}}(u_{\ell},\rho_{z_{\ell},z_{\ell-1}}^{i_p+1}(u_{\ell})) = \mathring{\omega}_{z_{\ell},z_{\ell-1}}^{i_p,i_{p+1}}(u_{\ell}) . $$
    For $t \in [n-k,n]$ we define $\chi_{x,y}^{i_p,i_{p+1};t}(u) = \chi_{z_{\ell},z_{\ell-1}}^{i_p,i_{p+1};t}(u_{\ell})$.
    We have  $\chi_{x,y}^{i_p,i_{p+1};n}(u) = \mathring{\eta}_{z_{\ell},z_{\ell-1}}^{i_p,i_{p+1}}(u_{\ell})$ and this is equal to $\mathring{\eta}_{x,y}^{i_p,i_{p+1}}(u)$ by condition (i) of an extended pre-transport function.
    Moreover, for $p= 0$ we set
    $$   \chi_{x,y}^{i_0,i_1;t} (u) = (\chi_{x,y}^{i_1,i_2;t}(u))^{-1}\cdot \ldots \cdot (\chi_{x,y}^{i_{m-1},i_m;t}(u))^{-1}\cdot \Phi(u) .      $$
    It is clear by construction that $\chi_{x,y}^{i_0,i_1;n}(u) = \mathring{\eta}_{x,y}^{i_0,i_1}(u)$.
    For $t \leq n-k-1$ we have 
    $$     \chi_{x,y}^{i_0,i_1;t}(u) = \mathring{\omega}_{x,y}^{i_1,i_2}(u))^{-1}\cdot \ldots \cdot (\mathring{\omega}_{x,y}^{i_{m-1},i_m}(u))^{-1}\cdot \Phi(u)   .    $$
    Since $\Phi(u) = \mathring{\omega}_{x,y}^{i_{m-1},i_m}(u) \cdot \ldots \cdot \mathring{\omega}_{x,y}^{i_0,i_1}(u)$ we have that $     \chi_{x,y}^{i_0,i_1;t}(u)  = \mathring{\omega}_{x,y}^{i_0,i_1}(u)$ is constant in $t$ for $t \in [0, n-k-1]$.
    As usual for arbitrary $i= i_s$ and $j = i_v$ with $v > s$ we set
    $$    \chi_{x,y}^{i_s,i_v;t}(u) = \chi_{x,y}^{i_{v-1},i_v;t}(u) \cdot \ldots \cdot \chi_{x,y}^{i_s,i_{s+1};t}(u) .       $$
    By induction we have thus constructed continuous maps $\chi_{x,y}^{i,j;\cdot}\colon \widetilde{W}_{x,y}^{i,j}\times [0,n]\to G$ with the property that $\chi_{x,y}^{i,j;t}$ is constant in $t$ for $t\in [0,n-k-1]$ and $\chi_{x,y}^{i,j;n-k-1} = \mathring{\omega}_{x,y}^{i,j}$ as well as $\chi_{x,y}^{i,j;n} = \mathring{\eta}_{x,y}^{i,j}$.
    
    Overall we have hence constructed an isotopy between $\mathring{\omega}_{x,y}^{i,j}$ and $\mathring{\eta}_{x,y}^{i,j}$ and as explained in the beginning of the proof this yields an isotopy between $\{\omega_{a,b}\}$ and $\{\eta_{a,b}\}$.
\end{proof}

\begin{proof}[Proof of Theorem \ref{theorem_main_result_transport-functions}]
    By construction the cocycle $\{\psi_{a,b}\colon V_a\cap V_b\to G\}$ represents the isomorphism class of $q\colon E\to M$.
    The transport function $\Phi_q\colon \mathrm{mor}(\mathcal{M}_f)\to G$ induced from the cocycle $\{\psi_{a,b}\}$ induces a $G$-principal bundle $E^{\Phi_q}\to M$ through the cocycle $\{\eta_{a,b}\colon V_a\cap V_b\to G\}$ which we constructed in Lemma \ref{lemma_transport_induces_ept}.
    In Lemmas \ref{lemma_first_isotopy} and \ref{lemma_second_isotopy} we have shown that the cocycles $\{\eta_{a,b}\}$ and $\{\psi_{a,b}\}$ induce isomorphic $G$-principal bundles, hence the $G$-principal bundles $q\colon E\to M$ and $E^{\Phi_q}\to M$ are isomorphic.
\end{proof}

\subsection{Equivalence of transport functions}\label{subsec_equivalence}

We have shown in Theorem \ref{theorem_main_result_transport-functions} that in every isomorphism class of $G$-principal bundles there is at least one representative which is induced by a transport function.
It is a natural question to ask in what situations two transport functions $\Phi,\Psi\colon \mathrm{mor}(\mathcal{M}_f)\to G$ induce isomorphic $G$-principal bundles.
We characterize this property for a path-connected topological group $G$ through the notion of isotopic transport functions.

\begin{definition}
    Let $(M,f,X)$ be Morse data and let $G$ be a topological group.
    Consider the set of transport functions 
    $$\mathrm{T}(\mathcal{M}_f,G)  = \{\Phi\colon \mathrm{mor}(\mathcal{M}_f)\to G\,|\, \Phi \text{ is a transport function}\,\} . $$
    We say that $\Phi,\Psi\in \mathrm{T}(\mathcal{M}_f,G)$ are \emph{isotopic} if there is a homotopy $\mathcal{H}\colon \mathrm{mor}(\mathcal{M}_f)\times [0,1]\to G$ between $\Phi$ and $\Psi$ such that $\mathcal{H}(\cdot, t)$ is a transport function for each $t\in [0,1]$.
\end{definition}

\begin{theorem}\label{theorem_isotopic_tr_functions_yield_isomorphic_bundles}
    Let $(M,f,X)$ be Morse data and let $G$ be a topological group.
    \begin{enumerate}
        \item Isotopic transport functions induce isomorphic $G$-principal bundles.
        \item If $G$ is path-connected and $\Phi,\Psi\in T(\mathcal{M}_f,G)$ induce isomorphic $G$-principal bundles, then they are isotopic.
    \end{enumerate}
    For path-connected $G$ we thus obtain a bijection between isotopy classes of $G$-valued transport functions on $\mathcal{M}_f$ and isomorphism classes of $G$-principal bundles over $M$.
\end{theorem}
\begin{proof}
    Let $\Phi,\Psi\in T(\mathcal{M}_f,G)$ be transport functions.
    First, assume that $\Phi$ and $\Psi$ are isotopic via a map $\mathcal{H}\colon \mathrm{mor}(\mathcal{M}_f)\times [0,1] \to G$.
    By Lemmas \ref{lemma_extended_pre-transport_induces_cocycle} and \ref{lemma_transport_induces_ept} the family of transport functions $\mathcal{H}(\cdot, t)$ induces a family of cocycles $\{ h_{a,b}^t \colon V_a\cap V_b\to G\}$ for $t\in [0,1]$ such that $h_{a,b}^0$ is the cocycle induced by $\Phi$ and $h_{a,b}^1$ is the cocycle induced by $\Psi$.
    Consequently, $\Phi$ and $\Psi$ induce isomorphic $G$-principal bundles.

    Conversely, let $\Phi, \Psi\in T(\mathcal{M}_f,G)$ be transport functions and assume that they induce isomorphic $G$-principal bundles.
    By Lemmas \ref{lemma_extended_pre-transport_induces_cocycle} and \ref{lemma_transport_induces_ept} the transport functions $\Phi$ and $\Psi$ induce cocycles ${}^{\Phi}\eta_{a,b}, {}^{\Psi}\eta_{a,b} \colon V_a\cap V_b\to G$.
    Since the induced bundles are isomorphic by assumption, the two cocycles are equivalent, i.e. there are maps $r_z\colon V_z\to G$ for each $z\in \mathrm{Crit}(f)$ such that for $a,b\in \mathrm{Crit}(f)$ and $p\in V_a\cap V_b$ we have \begin{equation}\label{eq_equivalence_transport_isotopy} {}^{\Phi}\eta_{a,b} (p) =  r_b(p) \cdot {}^{\Psi}\eta_{a,b}(p)  \cdot r_a(p)^{-1} .
    \end{equation}    
    We claim that we can recover the transport function $\Phi$ from the cocycle ${}^{\Phi}\eta_{a,b}$.
    Recall from Lemma \ref{lemma_maximal_sequences} that $\widetilde{S}_{x,y}^{i,j}\subseteq \widetilde{W}_{x,y}^{i,j}$.
    Let $u\in \overline{\mathcal{L}}(x,y)$ be a flow line.
    If $u\in \overline{\mathcal{L}}(x,y)\setminus V_{x,y}$, then we have $\Phi(u) = {}^{\Phi}\eta_{x,y}(u)$ by construction of the cocycle ${}^{\Phi}\eta_{a,b}$.
    Now, assume that $u = (\widehat{u},a)\in U_{x,y}$
    and let $|y| = i_0 < i_1 < \ldots < i_m = |x|$ be the maximal sequence of indices with $\rho_{x,y}^{i_k}(\widehat{u}) < \rho_{x,y}^{i_k+1}(\widehat{u}) = \rho_{x,y}^{i_{k+1}}(\widehat{u}) < \rho_{x,y}^{i_{k+1}+1}(\widehat{u}) $.
    We have $u\in \widetilde{S}_{x,y}^{i_0,i_1} \cap \ldots \cap \widetilde{S}_{x,y}^{i_{m-1},i_m}$ and there are points $p_1\in S_{x,y;y,z_1}^{i_0,i_1}, \ldots , p_{m}\in S_{x,y;z_{m-1},x}^{i_{m-1},i_m}$ such that $p_r$ lies on the flow line $u$.
    Moreover, by the construction in Lemma \ref{lemma_extended_pre-transport_induces_cocycle} we have
    $$    \Phi_{x,y}(u) =  {}^{\Phi}\eta_{z_1,x}(p_1) \cdot {}^{\Phi}\eta_{z_2,z_1}(p_2) \ldots  \cdot {}^{\Phi}\eta_{y,z_k}(p_{k+1})   .   $$
    By equation \eqref{eq_equivalence_transport_isotopy} we see that the equation
    \begin{eqnarray}\label{eq_expression_for_Phi}
        \Phi_{x,y}(u) &=&   r_x(p_1) \cdot   {}^{\Psi}\eta_{z_1,x}(p_1) \cdot r_{z_1}(p_1)^{-1} \cdot r_{z_1}(p_2)\cdot {}^{\Psi}\eta_{z_2,z_1} \cdot r_{z_2}(p_2)^{-1} \ldots  \\ &  &   r_{z_k}(p_{k+1}) \cdot {}^{\Psi}\eta_{y,z_k}(p_{k+1}) \cdot r_y(p_{k+1})^{-1}   \nonumber
    \end{eqnarray}
    holds as well.
    Recall that $V_z$ is contractible for each $z\in \mathrm{Crit}(f)$. Hence, there is a homotopy $H_z\colon V_z\times [0,1]\to V_z$ with $H_z(\cdot, 0) = \mathrm{id}_{V_z}$ and $H_z(\cdot,1) = z$.
    Define a map $\mathcal{H}_{x,y}\colon \overline{\mathcal{L}}(x,y)\times [0,1]\to G$ by setting 
    \begin{eqnarray*}
           \mathcal{H}_{x,y}(u,t) &=& r_x(H_x(p_1,t)) \cdot   {}^{\Psi}\eta_{z_1,x}(p_1) \cdot r_{z_1}(H_{z_1} (p_1,t))^{-1} \cdot r_{z_1}(H_{z_1}(p_2,t)) \cdot   {}^{\Psi}\eta_{z_2,z_1}(p_2) \cdot  \\ &  &   r_{z_2}(H_{z_2}(p_2,t)) \cdot  \ldots \cdot  r_{z_k}( H_{z_k}(p_{k+1},t)) \cdot {}^{\Psi}\eta_{y,z_k}(p_{k+1}) \cdot r_y( H_y( p_{k+1},t))^{-1}  .      
    \end{eqnarray*} 
    By equation \eqref{eq_expression_for_Phi} we see that $\mathcal{H}_{x,y}(\cdot ,0) = \Phi$ and $\mathcal{H}_{x,y}(u,1) = r_x(x) \cdot \Psi_{x,y}(u)\cdot r_y^{-1}(y)$.
    Further, one checks that $\mathcal{H}$ is an isotopy of transport functions between $\Phi$ and $\Sigma = \mathcal{H}(\cdot,1)$.
    For $z\in\mathrm{Crit}(f)$ choose a continuous path $g_z\colon [0,1]\to G$ with $g_z(0) =z$ and $g_z(1) = e$.
    Those paths exist by the assumption that $G$ is path-connected.
    Define $\mathcal{H}'_{x,y}\colon \overline{\mathcal{L}}(x,y)\times [0,1]\to G$ by 
    $$   \mathcal{H}'_{x,y}(u,t) = g_x(t) \cdot \Psi_{x,y}(u)\cdot g_y(t)^{-1} .        $$
    This defines an isotopy of transport functions between $\Sigma$ and $\Psi$ and by transitivity it follows that $\Phi$ and $ \Psi$ are isotopic. 
\end{proof}

In the case that $G$ is not path-connected, there are non-isotopic transport functions which induce isomorphic $G$-principal bundles.
Therefore, we introduce the following relation.
\begin{definition}
    Let $(M,f,X)$ be Morse data and let $\Phi,\Psi\in T(\mathcal{M}_f,G)$ be transport functions.
    \begin{enumerate}
        \item We say that $\Phi$ and $\Psi$ are c-\emph{equivalent} if there are group elements $r_x\in G$ for $x\in\mathrm{Crit}(f)$ such that $\Phi_{x,y} = r_x \cdot \Psi_{x,y} \cdot r_y^{-1}$.
        \item The equivalence relation on $T(\mathcal{M}_f,G)$ generated by isotopy and $G$-equivalence will be denoted by $\sim_{\mathrm{tr}}$.
    \end{enumerate}
\end{definition}
The \emph{c} stands for \emph{cocycle}, since the condition of being c-equivalent resembles the notion of equivalence of cocycles.
\begin{example}\label{example_transport_z2}
    Let $f\colon \mathbb{S}^1\to \mathbb{R}, (x,y)\mapsto y$ be the standard height function on the circle and let $G = \mathbb{Z}_2$.
    The critical points are $N = (0,1)\in \mathbb{S}^1$ and $S = (0,-1)\in\mathbb{S}^1$ and the only non-trivial moduli space is the two-point space $\mathcal{L}(N,S) = \{\lambda_{\mathrm{left}},\lambda_{\mathrm{right}}\}$.
    There are precisely four transport functions $\Phi\colon \mathrm{mor}(\mathcal{M}_f)\to \mathbb{Z}_2$.
    We have the two constant transport functions $\Phi_{+1},\Phi_{-1}\colon \mathrm{mor}(\mathcal{M}_f)\to \mathbb{Z}_2$ given by $\Phi_{+1} \equiv +1$ and $\Phi_{-1} \equiv -1$.
   Furthermore, the map $\Psi_a\colon \mathcal{L}(N,S)\to \mathbb{Z}_2 $ given by $$\Psi_a(\lambda_{\mathrm{left}}) = -1  \quad \text{and}\quad \Psi_a(\lambda_{\mathrm{right}}) = +1  $$ is a transport function as is $\Psi_b\colon \mathcal{L}(N,S)\to \mathbb{Z}_2$ defined by $\Psi_b = - \Psi_a$.
    None of these transport functions is isotopic to another one.
    However $\Psi_b$ and $\Psi_a$ are c-equivalent transport functions and similarly, $\Phi_{+1}$ and $\Phi_{-1}$ are c-equivalent.
    There are thus two equivalence classes in $\mathrm{T}(\mathcal{M}_f,\mathbb{Z}_2)$ with respect to $\sim_{\mathrm{tr}}$.
    These classes correspond to the two isomorphism classes of $\mathbb{Z}_2$-principal bundle over $\mathbb{S}_1$.
\end{example}
\begin{remark}
    If $G$ is path-connected, then the equivalence relation $\sim_{\mathrm{tr}}$ is generated by the isotopy relation, i.e. c-equivalent transport functions $\Phi,\Psi$ are already isotopic. 
    In fact if there are group elements $r_x\in G$, $x\in\mathrm{Crit}(f)$ with $\Psi_{x,y}= r_x \cdot \Phi_{x,y}\cdot  r_y^{-1}$, then choose paths $g_x\colon [0,1]\to  G$ connecting $r_x$ to $e$ for each $x\in \mathrm{Crit}(f)$.
    As in the proof of Theorem \ref{theorem_isotopic_tr_functions_yield_isomorphic_bundles} one sees that $\Psi$ is isotopic to $\Phi$.
\end{remark}
Again, with the same idea as in the proof of Theorem \ref{theorem_isotopic_tr_functions_yield_isomorphic_bundles} one can see the following.
\begin{cor}\label{cor_same_bundle_equialent_tr_func}
    Let $(M,f,X)$ be Morse data and let $G$ be a topological group.
    If two transport functions $\Phi,\Psi\in T(\mathcal{M}_f,G)$ induce isomorphic $G$-principal bundles, then $\Phi\sim_{\mathrm{tr}} \Psi$.
\end{cor}
\begin{proof}
    Let $\Phi,\Psi\in T(\mathcal{M}_f,G)$ and assume that the induced $G$-principal bundles are isomorphic.
    Following the proof of Theorem \ref{theorem_isotopic_tr_functions_yield_isomorphic_bundles} one shows that there are elements $r_x\in G$ such that the transport function $\Phi$ is isotopic to the transport function $\Xi\colon \mathrm{mor}(\mathcal{M}_f)\to G$ given by $\Xi_{x,y} = r_x \cdot \Psi_{x,y}\cdot r_y^{-1}$.
    Clearly $\Xi$ and $\Psi$ are c-equivalent and therefore $\Phi\sim_{\mathrm{tr}}\Psi$.
\end{proof}


\section{Morse homology with DG coefficients using transport functions}\label{sec_4}

In this section we define a DG Morse complex given enhanced Morse data $(M,f,X,o,\{s_{x,y}\})$ as well as a transport function $\Phi\colon \mathrm{mor}(\mathcal{M}_f)\to G$ for a topological group $G$ and a DG module $\mathcal{F}_{\bullet}$ over $\mathrm{C}_{\bullet}(G)$.
We shall see that equivalent transport functions induce quasi-isomorphic chain complexes and we shall consider some particular cases and examples.

\subsection{Definition of the twisted complex}

Let $(M,f,X,o,\{s_{x,y}\})$ be enhanced Morse data.
We refer to Section \ref{subsec_prelim_morse_DG} for the definition of enhanced Morse data.
Further, let $R$ be a commutative unital ring and consider cubical chains with coefficients in $R$.
Let $G$ be a topological group with multiplication $\mu\colon G\times G\to G$.
The cubical chains $\mathrm{C}_{\bullet}(G)$ become a strictly associative differential graded algebra by setting
$$      c\cdot d := \mu_* (c\times d) \quad \text{for}\,\,\, c,d\in \mathrm{C}_{\bullet}(G) .         $$

Let $\Phi\colon \mathrm{mor}(\mathcal{M}_f)\to G$ be a transport function.
Recall that $\Phi$ consists of a collection of continuous maps $\Phi_{x,y}\colon \overline{\mathcal{L}}(x,y)\to G$.
We consider the collection of chains
\begin{equation}\label{eq_left-twisting-cocycle_lie_group}
       m_{x,y} := (\Phi_{x,y})_* s_{x,y} \in \mathrm{C}_{|x|-|y|-1}(G) , \quad x,y\in\mathrm{Crit}(f), |y|<|x| .     
\end{equation}
If we want to stress the transport function we shall also write $m^{\Phi}_{x,y}$ for $m_{x,y}$.
Using the property of the transport function that $\Phi_{x,y}(u_1\circ u_2)  = \Phi_{x,z}(u_1)\cdot \Phi_{z,y}(u_2)$ for $u_1\in\overline{\mathcal{L}}(x,z), u_2\in \overline{\mathcal{L}}(z,y)$ one checks the following.
\begin{lemma}
    Let $(M,f,X,o,\{s_{x,y}\})$ be enhanced Morse data and let $\Phi\colon \mathrm{mor}(\mathcal{M}_f)\to G$ a transport function.
    Then the chains $m_{x,y} = (\Phi_{x,y})_* (s_{x,y})\in \mathrm{C}_{|x|-|y|-1}(G)$ for $x,y\in \mathrm{Crit}(f)$ constitute a twisting cocycle. 
\end{lemma}

Now, assume that $\mathcal{F}_{\bullet}$ is a differential graded $\mathrm{C}_{\bullet}(G)$ right-module via a map
$$   \Psi\colon\mathcal{F}_{\bullet} \otimes \mathrm{C}_{\bullet}(G) \to \mathcal{F}_{\bullet} .    $$
We define the Morse complex with coefficients in $\mathcal{F}$ to be the graded $R$-module
$$      \mathrm{C}_{\bullet}(M; (\Phi,\mathcal{F}_{\bullet})) :=    \mathcal{F}_{\bullet}  \otimes  R\langle \mathrm{Crit}_{\bullet}(f) \rangle       $$
together with the differential
$$    \partial ( \alpha \otimes x) :=   (\partial \alpha)\otimes x  + (-1)^{|\alpha|}  \sum_{\substack{y \in \mathrm{Crit}(f)\\ |y| < |x|}}    \Psi(\alpha\otimes m_{x,y}) \otimes y  $$
for $x\in\mathrm{Crit}(f)$ and $\alpha\in\mathcal{F}_{\bullet}$.
We shall denote the homology of this complex by 
$   \mathrm{H}_{\bullet}( M;(\Phi,\mathcal{F}_{\bullet}))       $
since we want to stress the dependence on the transport function.
Note that the complex $\mathrm{C}_{\bullet}(M;(\Phi,\mathcal{F}_{\bullet}))$ also depends on the other choices of the enhanced Morse data but we choose to omit this from notation.

We would like to show that equivalent transport functions induce quasi-isomorphic chain complexes.
Since the proof that isotopic transport functions induce quasi-isomorphic chain complexes is technical and uses a construction which is not used otherwise in this paper, we complete the proof of the following theorem in Appendix \ref{sec_proof_isotopy_quasi}.
We will prove the isotopy part only with $\mathbb{Z}_2$-coefficients.
\begin{theorem}\label{theorem_isotopic_tr_induces_quasi-iso}
    Let $(M,f,X,o,\{s_{x,y}\})$ be enhanced Morse data and let $G$ be a topological group.
    Let $\mathcal{F}_{\bullet}$ be a DG $\mathrm{C}_{\bullet}(G)$-module and let $\Phi,\Psi\colon \mathrm{mor}(\mathcal{M}_f)\to G$ be transport functions.
    \begin{enumerate}
        \item If $\Phi$ and $\Psi$ are c-equivalent, then the DG Morse complexes $\mathrm{C}_{\bullet}(M;(\Phi,\mathcal{F}_{\bullet}))$ and $\mathrm{C}_{\bullet}(M;(\Psi,\mathcal{F}_{\bullet}))$ are isomorphic as chain complexes.
        \item If $\Phi$ and $\Psi$ are isotopic and we take chain complexes with $\mathbb{Z}_2$-coefficients then the DG Morse complexes  $\mathrm{C}_{\bullet}(M;(\Phi,\mathcal{F}_{\bullet}))$ and $\mathrm{C}_{\bullet}(M;(\Psi,\mathcal{F}_{\bullet}))$ are quasi-isomorphic.
    \end{enumerate}
\end{theorem}
\begin{proof}
    Assume that $\Phi $ and $ \Psi$ are c-equivalent, i.e. there are points $r_x\in G$ such that $\Phi_{x,y} = r_x \cdot \Psi_{x,y}\cdot r_y^{-1}$.
    Define a map $\Xi\colon \mathrm{C}_{\bullet}(M;(\Phi,\mathcal{F}_{\bullet})) \to \mathrm{C}_{\bullet}(M;(\Psi,\mathcal{F}_{\bullet}))$ by 
    $$   \Xi(\alpha\otimes x) =  \alpha\cdot r_x^{-1}\otimes x \quad \text{for}\,\,\,\alpha \in\mathcal{F}_{\bullet},\,\,x\in \mathrm{Crit}(f) .     $$
    Note that we have $m^{\Phi}_{x,y} = r_x \cdot m^{\Psi}_{x,y}\cdot r_y^{-1}$.
    A straight-forward computation shows that $\Xi$ is a chain map and it is clearly an isomorphism of chain complexes.
    The second part of the theorem is proved in Appendix \ref{sec_proof_isotopy_quasi}.
\end{proof}

\begin{remark}
   The complex $\mathrm{C}_{\bullet}(M;(\Phi,\mathcal{F}_{\bullet}))$ further depends on the choice of enhanced Morse data $(M,f,X,o,\{s_{x,y}\})$.
   We do not show independence of this data here but we believe that this can be shown in the style of \cite[Section 6]{barraud2025morse}.
\end{remark}

We note that our main source of $\mathrm{C}_{\bullet}(G)$-modules $\mathcal{F}_{\bullet}$ comes from right actions of $G$.
Let $\psi\colon  F\times G \to F$ be a continuous right action of $G$ on a topological space $F$.
The map $\Psi = \psi_* \circ \times \colon \mathrm{C}_{\bullet}(F) \otimes  \mathrm{C}_{\bullet}(G) \to \mathrm{C}_{\bullet}(F)$ makes $\mathrm{C}_{\bullet}(F)$ into a DG right module over $\mathrm{C}_{\bullet}(G)$.
In Section \ref{sec_assoc_bundles} we will show that under certain conditions the homology of the DG Morse chain complex $\mathrm{C}_{\bullet}(M; (\Phi,\mathrm{C}_{\bullet}(F)))$ is the homology of the total space of the associated bundle $F\times_G E$ if $\Phi\colon \mathrm{mor}(\mathcal{M}_f)\to G$ is a transport function induced by the $G$-principal bundle $q\colon E\to M$.
We will also encounter the following situation.
Assume that $F'\subseteq F$ is a $G$-invariant subspace, then the relative chains $\mathrm{C}_{\bullet}(F,F')$ are a DG left module over $\mathrm{C}_{\bullet}(G)$.

\begin{remark}
    We refer to Section \ref{subsec_prelim_morse_DG} for an overview over the construction of the DG Morse complex in \cite{barraud2025morse}.
    The setup in \cite{barraud2025morse} is more universal compared to our present construction since the evaluation maps $q_{x,y}\colon \overline{\mathcal{L}}(x,y)\to \Omega M$ are chosen once and for all.
    In contrast, the transport functions, resp. their equivalence class, depends on the particular principal bundle $q\colon E\to M$.
    We will come back to the comparison between the construction in this article and the constructions in \cite{barraud2025morse} in Section \ref{sec_comparison}.
\end{remark}

Before we turn to examples we recall a result of \cite{barraud2025morse} which will enable us to do explicit computations.
\begin{prop}[\cite{barraud2025morse}, Proposition 4.8]\label{prop_smaller_complex}
    Let $(M,f,X,o,\{s_{x,y}\})$ be enhanced Morse data.
    Let $\mathbf{R}_{\bullet},\mathbf{R}_{\bullet}'$ be DGAs and let $\mathcal{F}_{\bullet}$, and $\mathcal{F}_{\bullet}'$ be DG modules over $\mathbf{R}_{\bullet}$, resp. $\mathbf{R}_{\bullet}'$.
    Assume that $f\colon \mathbf{R}_{\bullet}\to \mathbf{R}_{\bullet}'$ is a morphism of DGAs and $g\colon \mathcal{F}_{\bullet}\to \mathcal{F}_{\bullet}'$ be a morphism of complexes such that 
    $$       g( a\cdot r) =  g(a)\cdot f(r) \quad \text{for}\,\,\,  a\in\mathcal{F}_{\bullet} , \,\, r\in \mathbf{R}_{\bullet}.   $$
    Further, let $\{m_{x,y}\}$ be a twisting cocycle in $\mathbf{R}_{\bullet} $ and set $m'_{x,y} = f(m_{x,y})$.
    Then $m'_{x,y}$ is a twisting cocycle in $\mathbf{R}'_{\bullet}$ and the map 
    $$   \widetilde{g}\colon  \mathcal{F}_{\bullet} \otimes R\langle \mathrm{Crit}_{\bullet}(f)\rangle  \to  \mathcal{F}_{\bullet}' \otimes R\langle \mathrm{Crit}_{\bullet}(f)\rangle , \quad \alpha\otimes x \mapsto g(\alpha)\otimes x     $$
    is a morphism of complexes.
    If $g$ is a quasi-isomorphism, then so is $\widetilde{g}$.
\end{prop}

\subsection{Examples and particular cases}\label{subsec_examples}

\hphantom{b}
\medskip
\newline
\textbf{Trivial bundles}
    Let $ E= M\times G$ be the trivial $G$-principal bundle over $M$ and take $R = \mathbb{Z}$.
    The trivial bundle $E$ is induced by the trivial transport function, i.e. $\Phi^{\mathrm{triv}}\colon \mathrm{mor}(\mathcal{M}_f)\to G$ defined by $\Phi(u) = e$ for all $u\in\mathrm{mor}(\mathcal{M}_f)$.
    We claim that $m_{x,y} = 0$ for all pairs of critical points $(x,y)$ with $|x|-|y| \geq 2$.
    Recall that $m_{x,y} = (\Phi_{x,y})_* s_{x,y} $ and $s_{x,y}\in \mathrm{C}_{|x|-|y|-1}(\overline{\mathcal{L}}(x,y))$ is a class of positive degree.
    However, $\Phi_{x,y}$ is constant, i.e. it factors through the inclusion of the unit $\overline{\mathcal{L}}(x,y) \to \{e\} \hookrightarrow G$.
    Hence, the chain $m_{x,y}\in \mathrm{C}_{|x|-|y|-1}(G)$ is trivial since cubical chains are taken modulo degenerate chains.
    Moreover, if $|x|-|y| = 1$ then $m_{x,y} = (\Phi_{x,y})_*(s_{x,y}) =  n(x,y)\cdot e$ where $n(x,y)\in \mathbb{Z}$ is the signed count of gradient flow lines connecting $x$ and $y$ as in ordinary Morse homology with $\mathbb{Z}$-coefficients. 
    Let $\mathcal{F}_{\bullet}$ be a DG $\mathrm{C}_{\bullet}(G)$-module.
    For $x\in\mathrm{Crit}(f)$ and $\alpha\in\mathcal{F}_{\bullet}$ one sees that the differential of the twisted complex is
    \begin{eqnarray*}
             \partial (\alpha \otimes x) &=&   (\partial \alpha) \otimes x  + (-1)^{|\alpha|} \sum_{|y| = |x|-1} \Psi_*(\alpha\otimes m_{x,y})    \otimes y  \\
             &=&   (\partial \alpha) \otimes x + (-1)^{|\alpha|} \sum_{|y| = |x|-1}  \alpha \otimes  n(x,y) y  \\ &=&
               (\partial \alpha) \otimes x  + (-1)^{|\alpha|} \alpha\otimes (\partial^{\mathrm{Morse}} y) .
    \end{eqnarray*}
    The complex $\mathrm{C}_{\bullet}(M; (\Phi^{\mathrm{triv}},\mathcal{F}))$ is thus simply the tensor product
    $ \mathcal{F}_{\bullet}  \otimes (\mathrm{C}_{\bullet}(M), \partial^{\mathrm{Morse}})       $.

\medskip
\textbf{Discrete groups as fibers}
Assume that $G = \Gamma$ a discrete group and let $\Phi\colon \mathrm{mor}(\mathcal{M}_f)\to \Gamma$ be a transport function.
By construction of the cubical chains we have that $\mathrm{C}_i(\Gamma) = \{0\}$ for $i>0$ and there are natural isomorphisms of algebras
$$    \mathrm{C}_{\bullet}(\Gamma) \cong \mathrm{C}_0(\Gamma) \cong R[\Gamma]      $$
with $R[\Gamma]$ the group ring of $\Gamma$.
In the twisting cocycle $\{m_{x,y}\}$ we therefore have $m_{x,y} = 0$ for $|x|-|y| \geq 2$.
Let $\mathcal{F}_{\bullet}$ be a differential graded right module over $\mathrm{C}_{\bullet}(\Gamma)$.
The differential in the twisted complex $\mathrm{C}_{\bullet}(M;(\Phi, \mathcal{F}_{\bullet}))$ is given by
\begin{equation}\label{eq_twisted_complex_discrete_group}   \partial ( \alpha\otimes x)  = (\partial \alpha) \otimes x +  (-1)^{|\alpha |} \sum_{|y| = |x|-1}  \Psi (\alpha \otimes m_{x,y})  \otimes y .    \end{equation}
\begin{prop}
    Let $(M,f,X,o,\{s_{x,y}\})$ be enhanced Morse data and let $\Gamma$ be a discrete group.
    Let $R$ be a commutative unital ring.
    Let $\Phi\colon \mathrm{mor}(\mathcal{M}_f)\to \Gamma$ be a transport function and let $\mathcal{F}_{\bullet}$ be a $\mathrm{C}_{\bullet}(\Gamma)$-module.
    The DG Morse complex $\mathrm{C}_{\bullet}(M;(\Phi,\mathcal{F}_{\bullet}))$ is isomorphic as a chain complex of $R$-modules to the chain complex $\mathcal{F}_{\bullet} \otimes_{R[\Gamma]} \widetilde{C}_{\bullet} $, which is the tensor product of chain complexes of $R[\Gamma]$-modules between $\mathcal{F}_{\bullet}$ and the lifted complex $\widetilde{C}_{\bullet}$, see Section \ref{subsec_prelim_morse_DG}.
\end{prop}
\begin{proof}
   	Recall from Section \ref{subsec_prelim_morse_DG} that $\widetilde{C}_{\bullet}:= \mathrm{H}_0(\Gamma)\otimes_R \mathrm{R}\langle \mathrm{Crit}(f)\rangle$ is a chain complex in $\mathrm{H}_0(\Gamma) = R[\Gamma]$-modules.
    We recall that the differential is
    $$     \delta(x) = \sum_{|y| = |x|-1} [m_{x,y}] y \quad \text{for}\,\,\, x \in\mathrm{Crit}(f) .       $$
    Since $\mathrm{C}_{\bullet}(\Gamma) = \mathrm{H}_0(\Gamma) = R[\Gamma]$ it follows that $\mathcal{F}_{\bullet}$ is a chain complex in $R[\Gamma]$-modules as well.
    As $R$-modules we have $\mathcal{F}_{\bullet}\otimes_{R[\Gamma]} \widetilde{C}_{\bullet} = \mathcal{F}_{\bullet}\otimes_{R[\Gamma]} R[\Gamma]\otimes_R R\langle \mathrm{Crit}(f) \cong \mathrm{C}_{\bullet}(M;(\Phi,\mathcal{F}_{\bullet}))$.
    Using equation \eqref{eq_twisted_complex_discrete_group} one checks that the differential in the complex $\mathcal{F}_{\bullet}\otimes_{R[\Gamma]} \widetilde{C}_{\bullet}$ agrees with the differential in the DG Morse complex $\mathrm{C}_{\bullet}(M;(\Phi,\mathcal{F}_{\bullet}))$.
\end{proof}

\begin{example}
    Consider the two-fold covering of the circle by itself $q\colon \mathbb{S}^1\to \mathbb{S}^1, z\mapsto z^2$ which is a $\mathbb{Z}_2$-principal bundle.
    Here, we consider $\mathbb{S}^1\subseteq \mathbb{C}\cong \mathbb{R}^2$ as the unit circle in $\mathbb{C}\cong \mathbb{R}^2$.
    Let $f\colon \mathbb{S}^1\to \mathbb{R}$ be the height function on the circle $f(x,y) = y$. 
    There are two gradient flow lines from $N = (0,1)$ to $S= (0,-1)$, which we denote by $\lambda_{\mathrm{left}},\lambda_{\mathrm{right}}\in \mathcal{L}(N,S)$.
    With a suitable choice of orientations we get $s_{N,S} = \lambda_{\mathrm{right}} - \lambda_{\mathrm{left}}\in \mathrm{C}_0(\mathcal{L}(N,S))$.
    A transport function induced by the bundle $q$ is given by the map $\Phi\colon \mathcal{L}(N,S)\to \mathbb{Z}_2$,
    $$    \Phi(\lambda_{\mathrm{right}}) = +1, \quad \Phi(\lambda_{\mathrm{left}})= -1 ,    $$
    see Example \ref{example_transport_z2}.
    As mentioned above we have $\mathrm{C}_{\bullet} (\mathbb{Z}_2)  = \mathrm{C}_0(\mathbb{Z}_2) \cong \mathbb{Z}[\mathbb{Z}_2]$ as a DGA.
    Hence, the twisting cocycle is 
    $$   m_{N,S} = \Phi_* ( s_{N,S}) = (+1) - (-1) \in \mathbb{Z}[\mathbb{Z}_2]  .    $$
    Now, let $ \psi\colon \mathbb{S}^1\times \mathbb{Z}_2 \to \mathbb{S}^1, (x,y)\mapsto (x,-y)$ be the action induced by reflection along the real axis.
    Consider the complex $\mathcal{F}_{\bullet}' = \mathrm{C}_{\bullet}(\mathbb{S}^1)$ with the induced $\mathrm{C}_{\bullet}(\mathbb{Z}_2)$-module structure.
    In order to compute the homology of the twisted complex we want to replace $\mathrm{C}_{\bullet}(\mathbb{S}^1)$ by a smaller complex using Proposition \ref{prop_smaller_complex}.
    Let $a_0 = (1,0)\in \mathbb{S}^1$.
    We denote the cubical $0$-simplex given by $a_0$ by $a_0\in \mathrm{C}_0(\mathbb{S}^1)$ as well.
    We have $ a_0 \cdot (-1) = a_0$ for $(-1)\in\mathbb{Z}_2$.
    Further, let $b_1 \in \mathrm{C}_1(\mathbb{S}^1)$ be the $1$-cube $b_1(t) =e^{2\pi i t}$ and let $c_1\in\mathrm{C}_1(\mathbb{S}^1)$ be defined by $c_1(t) = e^{-2\pi i t}$.
    The chains $b_1$ and $c_1$ are cycles and we have $b_1 \cdot (-1) = c_1$.
    Let $b_2 \in \mathrm{C}_2(\mathbb{S}^1)$ be a chain such that $\partial b_2 = b_1 + c_1$ and set $c_2 = b_2 \cdot (-1)$.
    We have that $b_2 - c_2$ is a cycle, so we need to introduce higher-dimensional chains $b_3,c_3\in\mathrm{C}_3(\mathbb{S}^1)$ to kill of $b_2-c_2$.
    This process goes on and we thus obtain a $\mathbb{Z}_2$-invariant subcomplex $\mathcal{F}_{\bullet}\subseteq \mathrm{C}_{\bullet}(\mathbb{S}^1)$ with $\mathcal{F}_0 = \mathbb{Z}[a_0]$, $\mathcal{F}_1=  \mathbb{Z}[b_1,c_1]$, $\mathcal{F}_2  =\mathbb{Z}[b_2,c_2]$ and so on.
    The inclusion $\mathcal{F}_{\bullet}\hookrightarrow \mathrm{C}_{\bullet}(\mathbb{S}^1)$ is a $\mathbb{Z}_2$-equivariant quasi-isomorphism.
    Hence, we can use the complex $\mathcal{F}_{\bullet}$ to compute the DG Morse homology by Proposition \ref{prop_smaller_complex}.
    One finds that the differential of the complex $\mathcal{F}_{\bullet}\otimes \mathbb{Z}\langle \mathrm{Crit}_{\bullet}(f)\rangle $ vanishes in degree $1$, i.e. we have
    $$   \mathrm{ker}(d_1) = \mathbb{Z}[ b_1\otimes S, c_1\otimes S, a_0 \otimes N ] .     $$
    By computing the differential $d_2$ we find
    $$   \mathrm{im}(d_2) = \mathbb{Z}[  (b_1+c_1)\otimes S,  (b_1-c_1)\otimes S]  .   $$
    It is straight-forward to check that there is an isomorphism
    $$    \mathrm{H}_1( \mathcal{F}_{\bullet}\otimes  \mathbb{Z}\langle \mathrm{Crit}_{\bullet}(f)\rangle, \partial) \cong \mathbb{Z}\oplus \mathbb{Z}_2 .     $$
    Furthermore, the zero'th homology group is $\mathbb{Z}$ and the higher homology groups vanish.
    Hence, the homology of the twisted complex $ \mathcal{F}_{\bullet} \otimes \mathbb{Z}\langle \mathrm{Crit}_{\bullet}(f)\rangle$ is isomorphic to the singular homology of the Klein bottle.
    The Klein bottle is the total space of the associated bundle $\mathbb{S}^1\times_{\mathbb{Z}_2} \mathbb{S}^1 \to \mathbb{S}^1$.
    Our findings are therefore in agreement with Theorem \ref{theorem_assoc_bundle}.
\end{example}

\medskip
\textbf{Principal bundles on spheres} Recall from Example \ref{example_clutching} that for the standard height function on the sphere $\mathbb{S}^n$ a transport function $\Phi\colon {\mathcal{L}}(N,S) \cong \mathbb{S}^{n-1}\to G$ can be understood as a clutching function.

\begin{example}\label{example_hopf_fibration_transport}
    For $\mathbb{S}^1$-principal bundles over the $2$-sphere, the transport function is a map $\Phi\colon \mathbb{S}^1\to \mathbb{S}^1$.
    Consider the map $\Phi_k(e^{2\pi it}) = e^{2\pi ikt}$ for an integer $k\in \mathbb{Z}$.
    Clearly, the maps $\Phi_k$, $k\in\mathbb{Z}$, yield all isotopy classes of transport functions as $k$ ranges over $\mathbb{Z}$.

    As a representing chain system we choose $\sigma = s_{N,S}\in \mathrm{C}_1(\mathbb{S}^1)$  for $\sigma\colon I\to \mathbb{S}^1$ the standard cycle $\sigma(t) = e^{2\pi i t}$.
    If $k = 1$ we have $m_{N,S} = \sigma$.
    Again the explicit computation of the homology is facilitated by Proposition \ref{prop_smaller_complex}.
    We take a DGA model $\mathbf{R}_{\bullet}$ for $\mathrm{C}_{\bullet}(\mathbb{S}^1)$ with 
    $$    \mathbf{R}_0 = \mathbb{Z}[e], \quad \mathbf{R}_1 = \mathbb{Z}[\sigma]      $$
    with $e\in \mathrm{C}_0(\mathbb{S}^1)$ the unit.
    We further need to add the chain $y= \sigma\cdot \sigma\in \mathrm{C}_2(\mathbb{S}^1)$ to $\mathbf{R}_2$. This is a cycle so we also need to add a chain $z\in \mathrm{C}_3(\mathbb{S}^1)$ with $\partial z = y$.
    Furthermore we need to have $\sigma^3 \in\mathbf{R}_3$, i.e. we let $u = \sigma^3\in\mathrm{C}_{\bullet}(\mathbb{S}^1)$ be in our DGA $\mathbf{R}_{\bullet}$.
    This procedure has to be continued.
    One obtains a DGA $\mathbf{R}_{\bullet}$ together with a quasi-isomorphism of DGAs $\mathbf{R}_{\bullet}\hookrightarrow \mathrm{C}_{\bullet}(\mathbb{S}^1)$.

    Since we explicitly know the DGA $\mathbf{R}_{\bullet}$ in low degrees we compute the kernel of the differential in degree $1$.
    One finds that
    $    \mathrm{ker} (\partial_1) = \mathbb{Z}[\sigma \otimes S]    $
    and since 
    \begin{equation}\label{eq_differential_hopf_fibration}      
        \partial_2 (e\otimes N)  = (\partial e) \otimes N - (e\cdot \sigma)\otimes S  = - \sigma\otimes S   
    \end{equation}
    we see that the first homology vanishes.
    Similarly, one finds that the second homology vanishes and that $\mathrm{H}_3(\mathbf{R}_{\bullet} \otimes \mathbb{Z}[N,S]) \cong \mathbb{Z}$, i.e. we recover the homology of the $3$-sphere.
    The $3$-sphere is the total space of the principal $\mathbb{S}^1$-bundle $q\colon \mathbb{S}^3\to \mathbb{S}^2$ given by the Hopf fibration.
    This corresponds to the transport function $\Phi_1$, therefore our computation of the homology is in agreement with Theorem \ref{theorem_assoc_bundle}.
    
\end{example}

\section{DG Morse homology and associated bundles}\label{sec_assoc_bundles}

In this section we investigate the relationship between the homology of the total space of an associated bundle $F\times_G E\to M$ and the homology of the DG Morse complex $\mathrm{C}_{\bullet}(M;(\Phi,\mathrm{C}_{\bullet}(F)))$ for a transport function $\Phi\colon \mathrm{mor}(\mathcal{M}_f)\to G$ as well as a $G$-space $F$.
Here, $\Phi$ is a transport function induced by a $G$-principal bundle $q\colon E\to M$.
We shall consider certain section of bundles obtained by pulling back $q\colon E\to M$ to the compactification of the unstable manifolds $\overline{W}^u(x)$.
We shall see that the existence of these sections only depends on the equivalence class of the transport function and that the obstruction to the existence can be understood via classical obstruction theory.
If these sections exist, then we can show that the DG Morse homology agrees with the singular homology of the total space of the associated bundle $F\times_G E$.

\medskip
Let $(M,f,X,o,\{s_{x,y}\})$ be enhanced Morse data.
We consider the compactified unstable manifold $\overline{W}^u(x)$ which is homeomorphic to a ball $\mathbb{D}^{|x|}$, see \cite{latour1994existence} and \cite{qin2010moduli}.
Note that the boundary of $\overline{W}^u(x)$ is given by
$$     \partial \overline{W}^u(x) = \bigcup_{\substack{y\in \mathrm{Crit}(f) \\  |y| < |x|}} \overline{\mathcal{L}}(x,y) \times W^u(y) .       $$
There is a continuous map $g_x\colon \overline{W}^u(x)\to M$ given by $g_x(p) = p$ for $p\in W^u(x)$ and $g_x(u,p) = p$ for $(u,p)\in\overline{\mathcal{L}}(x,y)\times W^u(y)$.
The map $g_x$ maps the interior $W^u(x)\subseteq M$ homeomorphically onto itself and is actually the cell attachment map of CW structure induced by $f$, i.e. $\partial\overline{W}^u(x)\cong \mathbb{S}^{|x|-1}$ is mapped into the $(k-1)$-skeleton.
See \cite[Section 4.9]{audin2014morse} for details.

\begin{definition}
    Let $(M,f,X)$ be Morse data and let $q\colon E\to M$ be a $G$-principal bundle over $M$.
    Let $\Phi\colon \mathrm{mor}(\mathcal{M}_f)\to G$ be an induced transport function.
    We say that a family of maps $\{\phi_x\colon \overline{W}^u(x)\to E\}_{x\in\mathrm{Crit(f)}}$
    is a collection of \emph{compatible sections} for $\Phi$ if for each $(u,a),\in \overline{\mathcal{L}}(x,y)\times W^u(y)\subseteq \overline{W}^u(x)$ we have
    \begin{equation} \label{eq_compatible_section}
     \phi_x( u,a) = \Phi_{x,y}(u)   \cdot \phi_y(a)             \end{equation}
     and if the following diagram commutes
     $$
     \begin{tikzcd}
         &  E \arrow[]{d}{q} \\
         \overline{W}^u(x) \arrow[]{ru}{\phi_x} \arrow[]{r}{g_x} & M  .
     \end{tikzcd}
     $$
\end{definition}

A compatible section $\phi_x\colon \overline{W}^u(x)\to E$ is thus nothing else but a global section of the pullback bundle $g_x^*E\to \overline{W}^u(x)$.
Since $\overline{W}^u(x)$ is contractible, global sections always exist, but the boundary condition \eqref{eq_compatible_section} is non-trivial.
Let $x\in\mathrm{Crit}(f)$ be a critical point and assume that there exist compatible sections $\phi_y\colon \overline{W}^u(y)\to E$ for all $y\in\mathrm{Crit}(f)$ with $|y|<|x|$.
Note that $\partial \overline{W}^u(x)\cong \mathbb{S}^{|x|-1}$.
Equation \eqref{eq_compatible_section} determines a map $h_x\colon \partial\overline{W}^u(x)\to E$.
\begin{lemma}\label{lemma_continuity_compatible_section_on_boundary}
    Let $(M,f,X)$ be Morse data, let $q\colon E\to M$ be a $G$-principal bundle and assume that $\Phi\colon \mathrm{mor}(\mathcal{M}_f)\to G$ is an induced transport function.
    Let $x\in\mathrm{Crit}(f)$ be a critical point and assume that the compatible sections $\phi_y\colon \overline{W}^u(y)\to E$ exist for all $y\in \mathrm{Crit}(f)$ with $|y| < |x|$.
    The map $h_x\colon \partial\overline{W}^u(x)\to E$ defined by equation \eqref{eq_compatible_section} is continuous.
\end{lemma}
\begin{proof}
    We refer to \cite[Section 4.9]{audin2014morse} for the construction of the topology on the compactifications $\overline{W}^u(x)$.
    We prove the assertion by induction over the index of $x$.
    For $|x| = 1$, there is nothing to show.
    Therefore let $|x| = k$ and assume that continuity has already been shown for all critical points of index less than $k$.
    The only case we need to consider is the following.
    Let $(\lambda_n,x_n)\in \overline{\mathcal{L}}(x,a)\times W^u(a)$ be a sequence with $(\lambda_n,x_n)\to (\lambda, (\mu,x))\in \overline{\mathcal{L}}(x,a)\times \overline{\mathcal{L}}(a,b)\times W^u(b) \subseteq \overline{\mathcal{L}}(x,b)\times W^u(b)$ for $n\to \infty$.
    Since $|a|< k$ we have that $\phi_a$ is continuous by the induction hypothesis and thus 
    $$   \Phi_{x,a}(\lambda_n)\cdot \phi_a(x_n) \to \Phi_{x,a}(\lambda)\cdot \phi_a(\mu,x) = \Phi_{x,a}(\lambda)\cdot \Phi_{a,b}(\mu)\cdot \phi_b(x) = \Phi_{x,b}(\lambda\circ \mu)\cdot \phi_b(x)    $$
    for $n\to \infty$.
    Here, we used that $\Phi$ is a transport function.
    This shows that $h_x(\lambda_n,x_n) \to h_x((\lambda\circ \mu),x)$ for $n\to \infty$ and thus the continuity of $h_x$.
\end{proof}

\begin{example}
        We recall the situation of Example \ref{example_transport_z2}.
           For the non-trivial $2$-fold covering $q\colon \mathbb{S}^1\to \mathbb{S}^1$ and the height function on $\mathbb{S}^1$ we considered the transport function $\Phi\colon \mathcal{L}(N,S)\to \mathbb{Z}_2$ with $\Phi(\lambda_{\mathrm{right}}) = +1$ and $\Phi(\lambda_{\mathrm{left}})= -1$.
          Note that the compactification $\overline{W}^u(N) \cong [-1,1]$ has boundary points $(\lambda_{\mathrm{left}},S)$ and $(\lambda_{\mathrm{right}},S)$.
          Hence for an arbitrary choice of $\phi_S(S) = y\in \mathbb{S}^1$ we must have
          $     \phi_N( (\lambda_{\mathrm{left}},S)) = \Phi_{N,S}(\lambda_{\mathrm{left}}) \cdot \phi_S(S)  = -y    $ while
          $     \phi_N( (\lambda_{\mathrm{right}},S)) = \Phi_{N,S}(\lambda_{\mathrm{right}}) \cdot \phi_S(S)  = y  .  $
          One checks that the map $$\phi_N\colon [-1,1]\to \mathbb{S}^1, t\mapsto e^{i(\tfrac{\pi}{2}(1-t))}y$$ is a compatible section.
\end{example}

The next lemma shows that the question of the existence of compatible sections can be understood through classical obstruction theory.

\begin{lemma}
    Let $(M,f,X)$ be Morse data, let $q\colon E\to M$ be a $G$-principal bundle and assume that $\Phi\colon \mathrm{mor}(\mathcal{M}_f)\to G$ be an induced transport function. 
    Let $x\in\mathrm{Crit}(f)$ be a critical point with index $|x| = k$.
    Assume that compatible sections exist for all critical points $y\in\mathrm{Crit}(f)$ with $|y|<|x|$.
    There is a homotopy class $\gamma_x\in \pi_{k-1}(G)$ with the following property: a compatible section $\phi_x\colon \overline{W}^u(x)\to E$ exists if and only if $\gamma_x = 0$.
\end{lemma}
\begin{proof}
    Let $h_x\colon \partial\overline{W}^u(x)\to  E$ be the map defined by equation \eqref{eq_compatible_section}.
    The existence of a compatible section is nothing else but a solution to the lifting problem
    $$   
    \begin{tikzcd}
               \partial \overline{W}^u(x) \arrow[]{r}{h_x} \arrow[hook]{d}{} & [2.5em] E \arrow[]{d}{q} \\
        \overline{W}^u(x) \arrow[]{r}{g_x} \arrow[dashed]{ru}{} & M   
    \end{tikzcd}
 $$
 Under the homeomorphism $(\overline{W}^u(x),\partial\overline{W}^u(x))\cong (\mathbb{D}^k,\mathbb{S}^{k-1})$ we can envoke the classical result in obstruction theory which says that there exists a homotopy class $\gamma\in \pi_{k-1}(G)$ which is the complete obstruction to the existence of the desired lift.
\end{proof}

\begin{example}
    Let $q\colon E\to M$ be a $T^r$-principal bundle with $T^r$ the $r$-torus for $r\geq 1$.
    Since $\pi_i(T^r) \cong \{0\}$ for $i\geq 2$ we see that the only non-trivial obstructions can exist for critical points of index $2$.
    In the next section, we shall see however that these obstructions vanish since compatible sections always exist for $T^r$-principal bundles.
    Indeed, we shall show that compatible sections always exist for $G$-principal bundles with $G$ a Lie group.
\end{example}

It is natural to ask whether the existence of compatible sections depends on the specific transport function or only on its equivalence class.
The next proposition answers this.

\begin{prop}\label{prop_isotop_transport_functions_both_admit_sections_if_one_does}
    Let $(M,f,X)$ be Morse data, let $G$ be a topological group and let $\Phi,\Psi\colon \mathrm{mor}(\mathcal{M}_f)\to G$ be equivalent transport functions.
    If $\Phi$ admits compatible sections, then so does $\Psi$.
\end{prop}
\begin{proof}
    We first show the claim for isotopic transport functions $\Phi,\Psi\colon \mathrm{mor}(\mathcal{M}_f)\to G$.
    Assume that compatible sections $\phi^{\Phi}_x\colon \overline{W}^u(x)\to E$, $x\in\mathrm{Crit}(f)$ exist for $\Phi$ and let $\mathcal{H}\colon \mathrm{mor}(\mathcal{M}_f)\times [0,1]\to G$ be an isotopy between $\Phi$ and $\Psi$.
    We want show the following claim by induction.
    There are maps $\rho_x\colon \overline{W}^u(x)\times [0,1]\to E$ such that the diagram
    \begin{equation}\label{eq_com_diagam_comp}
        \begin{tikzcd}
            & [2em] E \arrow[]{d}{q} \\
             \overline{W}^u(x)\times [0,1]\arrow[]{r}{g_x\circ \mathrm{pr}_1} \arrow[]{ru}{\rho_x} & M
        \end{tikzcd}
    \end{equation}    
    commutes and such that for $(\lambda,a)\in  \overline{\mathcal{L}}(x,y)\times \overline{W}^u(y)\subseteq \partial\overline{W}^u(x)$ and $s\in [0,1]$ we have
    $$    \rho_x ((u,a),s) = \mathcal{H}(u,s)\cdot \rho_y(a,s) .     $$ 
    In particular $\phi^{\Psi}_x := \rho_x(\cdot ,1)\colon \overline{W}^u(x)\to E$ is a compatible section for $\Psi$, so this implies the claim of the proposition for isotopic transport functions.    
    Let $|x| = 0$, and set $\rho_x\colon \overline{W}^u(x)\times [0,1]$ to be the map $\rho_x(x,s) = \phi_x^{\Phi}(x)$.
    Clearly, diagram \eqref{eq_com_diagam_comp} commutes and the boundary condition is empty in this case.

    Now, assume that we have shown the claim of the induction for critical points $x'\in\mathrm{Crit}(f)$ with $ |x'| \leq k$ and let $x\in \mathrm{Crit}(f)$ with $|x| = k+1$.
    Define a map $$\zeta\colon \overline{W}^u(x)\times \{0\}\cup \partial \overline{W}^u(x)\times [0,1]\to E$$ by $    \zeta( p,0) =  \phi^{\Phi}_x (p)     $ for $u\in\overline{W}^u(x)$ and 
    $$     \zeta((u,a),s) =   \mathcal{H}_{x,y}(u,s)\cdot \rho_y(a,s) .    \quad \text{for}\,\,\, (u,a)\in \overline{\mathcal{L}}(x,y)\times W^u(y),\,\, s\in [0,1]       $$
    Similarly to the proof of Lemma \ref{lemma_continuity_compatible_section_on_boundary} one checks that $\zeta$ is well-defined and continuous.
    Under the identification $(\overline{W}^u(x),\partial\overline{W}^u(x))\cong (\mathbb{D}^n,\mathbb{S}^{n-1})$ we see that we need to solve the relative lifting problem
    $$
        \begin{tikzcd}
            \mathbb{D}^n\times \{0\}\cup \mathbb{S}^{n-1}\times [0,1] \arrow[]{r}{\zeta} \arrow[hook]{d}{} & E \arrow[]{d}{q} \\
            \mathbb{D}^n\times [0,1] \arrow[dashed]{ru}{} \arrow[]{r}{g_x\circ \mathrm{pr}_1} & M
        \end{tikzcd}
    $$
    in order to obtain a map $\rho_x\colon \overline{W}^u(x)\times [0,1]\to E$ whose existence we wanted to prove.
    Since $\mathbb{D}^n\times \{0\}\cup \mathbb{S}^{n-1}\times [0,1]$ is a strong deformation retraction of $\mathbb{D}^n\times [0,1]$ it is well-known that a solution $Z\colon \mathbb{D}^n\times [0,1]\to E$ to this relative lifting problem exists, see e.g. \cite[Theorem VII.6.4]{bredon:2013}
    Under the homeomorphism $\mathbb{D}^n\cong \overline{W}^u(x)$ the map $Z$ yields a map $\rho_x\colon \overline{W}^u(x)\times [0,1]\to E$ which satisfies the claimed properties.
    This completes the induction.

    Next, assume that $\Phi,\Psi\colon \mathrm{mor}(\mathcal{M}_f)\to G$ are c-equivalent transport functions, i.e. there are elements $r_x\in G$ for $x\in \mathrm{Crit}(f)$ such that $\Phi_{x,y} = r_x\cdot \Psi_{x,y}\cdot r_y^{-1}$.
    Assume that there are maps $\phi_x^{\Phi}\colon \overline{W}^u(x)\to E $ which are compatible sections for $\Phi$.
    Define $\phi_x^{\Psi}\colon \overline{W}^u(x)\to E$ by $\phi_x^{\Psi}(a) = r_x\cdot \phi_x^{\Phi}(a)$.
    It is clear $q\circ \phi_x^{\Psi} = g_x$.
    For $(u,a)\in\partial\overline{W}^u(x)$ we have 
    $$   \phi_x^{\Psi}(u,a) = r_x  \phi_x^{\Phi}(u,a) = r_x  \Phi_{x,y}(u)r_y^{-1} r_y \phi_y^{\Phi}(a) = \Psi_{x,y}(u) \phi_y^{\Psi}(a)   .  $$
    Thus, $\phi_x^{\Psi}$ is a compatible section for $\Psi$.
    This completes the proof.
\end{proof}

In the next section we shall see transport functions for smooth principal bundles defined through parallel transport.
In this case compatible sections always exist and we construct them explicitly.
In general, we do not have an example where compatible sections do not exist.

We now want to show that if $\Phi$ is a transport function induced by the $G$-principal bundle $q\colon E\to M$, then the DG Morse complex $\mathrm{C}_{\bullet}(M;(\Phi,\mathrm{C}_{\bullet}(F)))$ computes the homology of the total space of the associated bundle $F\times_G E$ where $F$ is a right $G$-space.
Barraud, Damian, Humilière and Oancea define in \cite{barraud2025morse} a \emph{compatible representing chain system} as follows.
\begin{definition}
    Let $\{s_{x,y}\}$ be a representing chain system.
    A \emph{compatible representing chain system} is a collection $\{s_x\in \mathrm{C}_{|x|}(\overline{W}^u(x))\}_{x\in\mathrm{Crit}(f)}$ such that
    \begin{enumerate}
        \item the chain $s_x$ is a cycle relative to the boundary and represents the fundamental class $[\overline{W}^u(x)]\in \mathrm{H}_{|x|}(\overline{W}^u(x),\partial\overline{W}^u(x))$.
        \item The following equation holds
        $$    \partial s_x =   \sum_{\substack{y \in \mathrm{Crit}(f)\\ |y|<|x|}} s_{x,y}\times s_y .  $$
    \end{enumerate}
\end{definition}

In \cite[Lemma 7.6]{barraud2025morse} the authors show the following: given a representing chain system $\{s_{x,y}\}$there exists a compatible representing chain system $\{s_x\}$.

\begin{theorem}\label{theorem_assoc_bundle}
    Let $(M,f,X,o,\{s_{x,y}\})$ be enhanced Morse data and let $q\colon E\to M$ be a $G$-principal bundle.
    Let $\Phi\colon \mathrm{mor}(\mathcal{M}_f)\to G$ be a transport function induced by $q\colon E\to M$ and assume that compatible sections exist for all critical points $x\in \mathrm{Crit}(f)$.    
    Let $F$ be a right $G$-space and consider the associated bundle $F\times_G E\to M$.
    There is a chain map $t\colon \mathrm{C}_{\bullet}(M;(\Phi,\mathrm{C}_{\bullet}(F)))\to \mathrm{C}_{\bullet}(F\times_G E)$ which induces an isomorphism between the canonical spectral sequence associated to the DG Morse complex $\mathrm{C}_{\bullet}(M;(\Phi,\mathrm{C}_{\bullet}(F)))$ and the Leray-Serre spectral sequence associated to the bundle $F\times_G E\to M$.
    In particular, the homology of the chain complex $\mathrm{C}_{\bullet}(M; (\Phi,\mathrm{C}_{\bullet}(F)))$ is isomorphic to the singular homology of $F\times_G E$.
\end{theorem}
The following proof transfers the strategy of the proof of \cite[Theorem 7.2]{barraud2025morse} to our situation.
The role of the lifting function is taken by the canonical map $p\colon F\times E\to F\times_G E$ of the associated bundle.
\begin{proof}
    Write $X = F \times_G E$ for the total space of the associated bundle.
    We shall construct a chain map $t \colon \mathrm{C}_{\bullet}(M; (\Phi,\mathrm{C}_{\bullet}(F)))\to \mathrm{C}^{\text{cubical}}_{\bullet}(X)$ and we shall show that it respects filtrations and that it induces an isomorphism of the first page of the associated spectral sequences.

    By assumption there are compatible sections $\phi_x\colon \overline{W}^u(x)\to E$.
    Define a collection of chains $\{m_x \in \mathrm{C}_{\bullet}(E) \}_{x\in\mathrm{Crit}(f)}$ by setting $m_x := (\phi_x)_*(s_x)\in \mathrm{C}_{|x|}(E)$.
    By the boundary condition for the compatible sections, see \eqref{eq_compatible_section} and by the property of the compatible representing chain system we find that 
    $$  \partial m_x =   \sum_{\substack{y \in \mathrm{Crit}(f)\\ |y|<|x|}} m_{x,y} \cdot m_y   $$
    where we use the left $\mathrm{C}_{\bullet}(G)$-module structure of $\mathrm{C}_{\bullet}(E)$ through the left action of $G$ on $E$.
    We define a map $t \colon \mathrm{C}_{\bullet}(M; (\Phi,\mathrm{C}_{\bullet}(F)))\to \mathrm{C}^{\text{cubical}}_{\bullet}(X)$ by setting
    $$   t(\alpha \otimes x) = p_* (  \alpha \times m_x   )   , \quad \text{for}\,\,\, x\in \mathrm{Crit}(f),\,\, \alpha\in\mathrm{C}_{\bullet}(F)    $$
    where $p\colon F\times E\to X$ is the canonical projection.
    In order to show that $t$ is a chain map let $x\in\mathrm{Crit}(f)$ and $\alpha\in\mathrm{C}_{\bullet}(F)$.
    Then 
    \begin{eqnarray*}
        \partial t  (\alpha \otimes x ) &=& \partial p_*(\alpha \times m_x) 
        \\
        &=& p_* (  (\partial \alpha) \times m_x + (-1)^{|\alpha|} \alpha \times (\partial m_x)  ) \\
        &=&  p_* ((\partial \alpha)\times m_x)   \, + \,  (-1)^{|\alpha|} \sum_{\substack{y \in \mathrm{Crit}(f)\\ |y|<|x|}} p_* (\alpha \times (m_{x,y}\cdot m_y))  \\
        &=&
        p_* (( \partial \alpha) \times m_x) \, + \, (-1)^{|\alpha|}
        \sum_{\substack{y \in \mathrm{Crit}(f)\\ |y|<|x|}}  p_*( \alpha \cdot m_{x,y} \times  m_y  )
    \end{eqnarray*}
    where the last equality holds because of the fact that $p(z,g\cdot e) = p(z\cdot g,e)$ for $e\in E,g\in G$ and $z\in F$.
    On the other hand we see that the last term in the above equation is nothing else but $t (\partial (x\otimes \alpha))$ and consequently, $t$ is a chain map.

    The maps $g_x\colon \overline{W}^u(x)\to M$ induce a CW decomposition of $M$, see \cite[Section 4.9]{audin2014morse}.
    The skeletal filtration  of $M$ is given by
    $$   \mathrm{Sk}_k(M) = \bigcup_{\substack{x \in \mathrm{Crit}(f)\\ |x|\leq k}} g_x (\overline{W}^u(x)) .      $$
    Consequently, there is a filtration of the total space of the associated bundle $\pi\colon X\to M$ with $X_k  := \pi^{-1}(\mathrm{Sk}_k(M))$.
    We shall consider the Leray-Serre spectral sequence $\mathcal{E}_{p,q}$ associated to the filtration $\mathrm{C}_{\bullet}(X_k)\subseteq \mathrm{C}_{\bullet}(X)$ as well as the canonical spectral sequence $E_{p,q}$ associated to the DG Morse complex, see Section \ref{subsec_prelim_morse_DG}.
    We claim that $t$ preserves filtrations and that it induces an isomorphism on the first page of the spectral sequence.
    
    Write $\pi\colon X\to M$ for the bundle map of the associated bundle $X = F\times_G E\to M$.
   	By construction we have the identity $\pi\circ p =  q\circ \mathrm{pr}_2$ as maps $ F\times E\to M$.
    Thus we have
    $$   \pi_*(t(\alpha\otimes x)) = \pi_* (p_*(\alpha\times m_x)) = (q\circ \mathrm{pr}_2)_*(\alpha\times m_x) .   $$
    Hence, the chain $\alpha\otimes x\in \mathrm{C}_{\bullet}(F) \otimes \big(\bigoplus_{|y|\leq |x|} R\langle y\rangle\big)$ is mapped to $\mathrm{C}_{\bullet}(X_{|x|})$ and $t$ is therefore filtration-preserving.
    It is well-known that the spectral sequence $\mathcal{E}_{p,q}$ satisfies
    $$    \mathcal{E}^1_{p,q} = \mathrm{H}_{p+q}(X_p,X_{p-1}) \cong \mathrm{H}_q(F)\otimes \bigoplus_{\substack{x \in \mathrm{Crit}(f)\\ |x| = p}} \mathrm{H}_p(\overline{W}^u(x),\partial\overline{W}^u(x)) .  $$
    Explicitly, this isomorphism is induced by
    $$   \chi_x \colon F\times \overline{W}^u(x)  \to X_p   , \quad \chi_x(f,a) = [f,\phi_x(a)] = p(f,\phi_x(a))    $$
    for $f\in F$, $a\in \overline{W}^u(x)$.
    We obtain the following commutative diagram
    $$
        \begin{tikzcd}
            E^1_{p,q}= \mathrm{H}_q(F)\otimes R\langle \mathrm{Crit}_p(f)\rangle \arrow[]{r}{t_*} \arrow[]{d}{\zeta} & [3em] \mathrm{H}_{p+q}(X_p,X_{p-1}) = \mathcal{E}_{p,q}^1 \\
            \bigoplus_{\substack{x \in \mathrm{Crit}(f)\\ |x| = p}} \mathrm{H}_q(F)\otimes \mathrm{H}_p(\overline{W}^u(x),\partial\overline{W}^u(x)) \arrow[swap]{ru}{\oplus(\chi_x)_*,\,\,\cong} & 
        \end{tikzcd}
    $$
    where $\zeta$ is the map $\zeta([\alpha]\otimes x) \mapsto [\alpha]\otimes [s_x]$.
    The map $\zeta$ is an isomorphism and by the commutativity of the diagram we conclude that $t$ is an isomorphism as well.
    Hence, the chain map $t$ induces an isomorphism of the first page of the respective spectral sequences and thus an isomorphism of spectral sequences.
    Hence, $t$ induces an isomorphism in homology by a well-known statement in homological algebra, see \cite[Proposition VII.2.6]{brown1982cohomology}.
    This completes the proof.   
\end{proof}

\begin{example}
    Let $G$ be a Lie group and $H \subseteq G$ a closed subgroup
    The canonical projection $G\to H\backslash G = M$ is an $H$-principal fiber bundle and by Example \ref{example_free_loop_space} we can write the free loop space as the associated bundle $\Omega M\times_{H} G$.
    Let $\Phi\colon \mathrm{mor}(\mathcal{M}_f)\to H$ be a transport function.
    In case that all compatible sections exist then by Theorem \ref{theorem_assoc_bundle} we see that the homology of the DG Morse complex $\mathrm{C}_{\bullet}(M;(\Phi,\mathrm{C}_{\bullet}(\Omega M)))$ yields the homology of the free loop space $\Lambda M$.
    Since $H$ is a Lie group we shall see in the next section that compatible sections do indeed exist in this situation.
\end{example}

For later purposes we also want to establish a relative version of the above Theorem.
\begin{cor}
    Let $(M,f,X,o,\{s_{x,y}\})$ be enhanced Morse data and let $q\colon E\to M$ be a $G$-principal bundle.
    Let $\Phi\colon \mathrm{mor}(\mathcal{M}_f)\to G$ be an induced transport function and assume that all compatible sections exist.
    Let $F_1\hookrightarrow F_2$ be an embedding of $G$-spaces.
    Then the homology of the DG Morse complex $\mathrm{C}_{\bullet}(M;(\Phi,\mathrm{C}_{\bullet}(F_2,F_1)))$ is isomorphic to the relative homology of the pair $(F_2\times_G E, F_1\times_G E)$.
\end{cor}
\begin{proof}
    This follows from the proof of Theorem \ref{theorem_assoc_bundle} and the $5$-Lemma.
\end{proof}

\section{Transport functions for smooth principal bundles via parallel transport}\label{sec_comparison}

In this section we consider smooth $G$-principal bundles $E\to M$ and their associated bundles $F\times_G E\to M$.
Note that we do not require the fiber $F$ to be a manifold.
We shall see that the choice of a principal connection $\omega$ on $E$ and the induced holonomy yield a transport function  $\Phi^{\omega}\colon \mathrm{mor}(\mathcal{M}_f)\to G$.
The obvious question is how the transport function $\Phi^{\omega}$ and the transport functions that arise from local trivializations in Section \ref{sec_transport_function} are related.
After addressing this question we shall study the relation between the DG Morse complex with coefficients in the $\mathrm{C}_{\bullet}(G)$ module $\mathrm{C}_{\bullet}(F)$ induced by a $G$-action on $F$ and a DG Morse complex in the style of the DG Morse complexes in \cite{barraud2025morse}.
One technical complication which arises is that in \cite{barraud2025morse} the authors use a homotopy inverse to the map $M\to M/\mathcal{Y}$ which collapses a tree $\mathcal{Y}$.
In general it is not clear whether the collapsing map and its homotopy inverse preserve piecewise smooth paths.
Therefore we consider a construction using the DG category $\{\mathrm{C}_{\bullet}(P'_{x\to y}M)\}_{x,y\in\mathrm{Crit}(f)}$ with $P'_{x\to y}M$ the space of piecewise smooth Moore paths instead of the DGA $\mathrm{C}_{\bullet}(\Omega M)$.

\medskip
Let $q\colon E\to M$ be a smooth $G$-principal bundle and let $\omega$ be a principal connection on $E$, i.e. $\omega$ is a $G$-equivariant one-form on $E$ with values in the Lie algebra $\mathfrak{g}$ of $G$ or equivalently a $G$-invariant horizontal distribution in $TE$.
We refer to \cite[Section II]{kobayashi:1963} for an introduction to principal connections and the induced constructions like parallel transport and holonomy.
Recall that we are considering principal bundles with left actions in contrast to most of the literature which considers right actions.
Let ${P}'M$ be the space of piecewise smooth Moore paths in $M$, i.e.
$$  {P}'M = \{(\gamma,a)\in PM\times [0,\infty)\,|\, \gamma(t) = \gamma(a) \quad \text{for}\,\,t\geq a , \,\, \gamma \text{ piecewise smooth }\}.     $$
By abuse of notation we shall often just write $\gamma$ for an element $(\gamma,a)\in P'M $.
We further introduce the spaces
$$    {P}'_{x \to y}M = \{ (\gamma,a)\in {P}'M \,|\,   \gamma(0)= x, \,\gamma(a) = y\}     $$
for $x,y\in M$ as well as the based loop space $   \Omega_{x_0}'M = {P}_{x_0\to x_0}M     $ for a basepoint $x_0\in M$.
Note that the collection $\{P'_{x\to y}M\}_{x,y\in M}$ yields a topological category.
The composition is given by concatenation which is strictly associative since we are working with Moore paths.

We further recall that the principal connection $\omega$ induces a notion of \emph{horizontal lift}, i.e. there is a map
$$     {E}\tensor[{_{q}}]{{\times}}{_{\mathrm{ev}_0}}  {P}'M\to P' E     , \quad (u,(\gamma,a))\mapsto  (\widetilde{\gamma}_u,a) $$
such that $\widetilde{\gamma}_u(0) = u$ and $q\circ \widetilde{\gamma}_u = \gamma$.
Note that the horizontal lift depends on the connection.
As usual, a horizontal lift induces a \emph{parallel transport}
$$   \mathcal{P}^{\omega} \colon {E}\tensor[{_{q}}]{{\times}}{_{\mathrm{ev}_0}}  {P}'M\to E , \quad     (u,(\gamma,a))\mapsto  \mathcal{P}^{\omega}_{\gamma}(u)  :=  \widetilde{\gamma}_u(a)  . $$
Note that if $(\gamma,a)\in P'M$ then the map $\mathcal{P}_{\gamma}^{\omega}\colon E_{\gamma(0)}\to E_{\gamma(a)}$ is a $G$-equivariant diffeomorphism with inverse \begin{equation}\label{eq_parallel_transport_inverse}
    (\mathcal{P}_{\gamma}^{\omega})^{-1} = \mathcal{P}_{\overline{\gamma}}^{\omega}
\end{equation}  where $(\overline{\gamma},a) \in P'M$ is the Moore path $\overline{\gamma}(t) = \gamma(a-t)$ for $t\in [0,a]$.

We now recall the \emph{ratio map} of a principal bundle.
Let $q\colon E\to M$ be a $G$-principal fiber bundle and define a map $r\colon E\times_M E\to G$ as follows.
For $p_1,p_2\in E$ with $q(p_1) = q(p_2)$, let $g\in G$ be the unique element such that $g\cdot p_1 = p_2$.
We set $r(p_1,p_2) = g$.
The ratio map $r$ is a smooth map which can be seen by going to local trivializations.
One further checks that the ratio map is invariant under parallel transport, i.e. 
    \begin{equation}\label{eq_ratio_map_invariant_parallel}
              r(u,v) = r(\mathcal{P}^{\omega}_{\gamma}(u), \mathcal{P}^{\omega}_{\gamma}(v))    
        \end{equation} 
        for $(u,v)\in E\times_M E$ and a path $\gamma\in P'M$ with $\gamma(0) = q(u) = q(v)$.
    This follows from the $G$-equivariance of parallel transport.

The parallel transport induces the \emph{holonomy} of the connection $\omega$.
Fix a basepoint $x_0\in M$ as well as a point $e_0\in E_{x_0}$ in the fiber over $x_0$.
Then the map 
$$  \mathrm{hol}_{e_0}^{\omega}\colon \Omega' M\to G, \quad \gamma\mapsto r(e_0, \mathcal{P}_{\gamma}^{\omega}(e_0))    $$
is the \emph{holonomy} induced by $\omega$, i.e. $\mathrm{hol}_{e_0}^{\omega}(\gamma) = g$ for the unique group element $g\in G$ such that $\mathcal{P}_{\gamma}^{\omega}(e_0) = g\cdot e_0$.
We note that the holonomy satisfies 
\begin{equation}\label{equation_transformation_formula_holonomy}
        \mathrm{hol}_{k\cdot e_0}^{\omega}(\gamma) = k \cdot \mathrm{hol}_{e_0}^{\omega}(\gamma)\cdot  k^{-1}       
\end{equation}
for $\gamma\in \Omega'M$ and $k\in G$ as well as
\begin{equation}\label{eq_holonomy_is_multiplicative}
    \mathrm{hol}_{e_0}^{\omega} (\gamma\circ \sigma) = \mathrm{hol}_{e_0}^{\omega}(\gamma)\cdot \mathrm{hol}_{e_0}^{\omega}(\sigma)
\end{equation}
for $\gamma,\sigma\in \Omega'M$.

Fix enhanced Morse data $(M,f,X,o,\{s_{x,y}\})$ on $M$.
We choose and fix a tree $\mathcal{Y}$ in $M$ which connects the basepoint to each critical point with a piecewise smooth path.
We denote the paths given by the tree by $\gamma_{x_0\to x}$ for $x\in\mathrm{Crit}(f)$ and we shall write $\gamma_{x\to y_0}$ for the inverse paths, i.e. $\gamma_{y\to x_0} = \overline{\gamma}_{x_0\to y}$.
Define maps $Q_{x,y}\colon \overline{\mathcal{L}}(x,y)\to \Omega'M$ by 
$$   Q_{x,y} (u) =  \gamma_{x_0\to x}\circ \overline{\gamma_u} \circ \gamma_{y\to x_0}     $$
with $\gamma_u\colon [f(y),f(x)]\to M$ the parametrization of the flow line $u$ as in Section \ref{subsec_morse_flow}.
Note that these are not the same maps as the evaluation maps $q_{x,y}$ which are used in \cite{barraud2025morse}. The maps $Q_{x,y}$ do not strictly preserve compositions, i.e. we do not have that $Q_{x,y}(u_1\circ u_2) = Q(x,y)(u_1)\circ Q_{x,y}(u_2)$ in general.
However, we can use the maps $Q_{x,y}$ to obtain a transport function.
\begin{lemma}
    The maps $\Phi_{x,y}^{\omega}\colon \overline{\mathcal{L}}(x,y)\to G$ defined by $\Phi_{x,y}^{\omega}(u) = \mathrm{hol}_{e_0}^{\omega}(Q_{x,y}(u))$ yield a transport function.
\end{lemma}
\begin{proof}
    Let $u = u_1\circ u_2\in \overline{\mathcal{L}}(x,y)$ be a broken flow line with $u_1\in \overline{\mathcal{L}}(x,z)$ and $u_2\in \overline{\mathcal{L}}(z,y)$.
    We have that
    $$    Q_{x,y}(u_1 \circ u_2) = \gamma_{x_0\to x}\circ \overline{\gamma}_{u_1\circ u_2} \circ \gamma_{y\to x_0}  =  \gamma_{x_0\to x}\circ \overline{\gamma}_{u_2}\circ \overline{\gamma}_{u_1} \circ \gamma_{y\to x_0}.    $$
    Note that parallel along the path $\gamma_{z\to x_0}\circ \gamma_{x_0\to z}$ is the trivial map, see equation \eqref{eq_parallel_transport_inverse}.
    Therefore parallel transport along the path $Q_{x,y}(u_1\circ u_2)$ agrees with parallel transport along
    $$   \gamma_{x_0\to x}\circ \overline{\gamma}_{u_1}  \circ \gamma_{z\to x_0}\circ \gamma_{x_0\to z} \circ \overline{\gamma}_{u_2}\circ \gamma_{y\to x_0}   = Q_{x,z}(u_1)\circ Q_{z,y}(u_2)   .  $$
    By equation \eqref{eq_holonomy_is_multiplicative} we obtain
    $$  \Phi_{x,y}^{\omega}(u_1\circ u_2) =  \mathrm{hol}_{e_0}^{\omega}(Q_{x,y}(u_1\circ u_2)) = \mathrm{hol}_{e_0}^{\omega}(Q_{x,z}(u_1)) \cdot \mathrm{hol}_{e_0}^{\omega}(Q_{z,y}(u_2)) = \Phi_{x,z}^{\omega}(u_1) \cdot \Phi_{z,y}^{\omega}(u_2) .   $$
    This shows that the maps $\{\Phi_{x,y}^{\omega}\}$ are a transport function.
\end{proof}

We denote the transport function from the above Lemma by $\Phi^{\omega}\colon \mathrm{mor}(\mathcal{M}_f)\to G$.
We shall now see that the $G$-principal bundle induced by $\Phi^{\omega}$ is isomorphic to the original bundle $q\colon E\to M$ that we started with.

\begin{theorem}\label{theorem_parallel_transport_induces_og_bundle}
    Let $(M,f,X)$ be Morse data and let $q\colon E\to M$ be a smooth $G$-principal bundle with principal connection $\omega$.
    Let $\mathcal{Y}$ be a piecewise smooth tree rooted at the basepoint $x_0\in M$.
    The $G$-principal bundle induced by the transport function $\Phi^{\omega}\colon \mathrm{mor}(\mathcal{M}_f)\to G$ is isomorphic as a topological $G$-principal bundle to the bundle $q\colon E\to M$.
\end{theorem}
\begin{proof}
    Recall the open cover $\{V_z\}_{z\in \mathrm{Crit}(f)}$ of $M$ from Section \ref{sec_transport_function}.
    Using parallel transport with respect to $\omega$ we shall construct local trivializations $\psi_z\colon E|_{V_z}\to G\times V_z$ which induce a transport function $\Phi_{\psi}\colon \mathrm{mor}(\mathcal{M}_f)\to G$ by Lemma \ref{lemma_trivializations_induce_pre-transport} and Proposition \ref{prop_pre-transport_induces_transport}.
    We will show that the transport function $\Phi_{\psi}$ is isotopic as a transport function to the transport function $\Phi^{\omega}$.
    By Theorem \ref{theorem_isotopic_tr_functions_yield_isomorphic_bundles} this implies that the $G$-principal bundle induced by $\Phi^{\omega}$ is isomorphic to $E$.

    Let $H'\colon V_z\times I \to V_z$ be the strong deformation retraction of $V_z$ to $z$ which we constructed in Section \ref{sec_transport_function}, i.e. $H'(p,0) =p$ and $H'(p,1) =z$ for all $p\in V_z$.
    On $A= V_z \cap (W^u(z) \cup W^s(z))$ the homotopy $H'$ moves the points along the flow lines on which they lie until they reach $z$.
    Moreover, $H'|_{V_z\times \{0\}}$ and $H'|_{V_z\times \{1\}}$ are smooth maps since $H'(p,0) =p$ and $H'(p,1) = z$ for $p\in V_z$.
    Thus $H'$ is smooth on the subset $A\times [0,1]\cup V_z\times \{0,1\}$ which is a closed subset in $V_z\times [0,1]$.
    By the Whitney approximation theorem there is hence a smooth map $H\colon V_z\times I\to V_z$ which agrees with $H'$ on $A\times [0,1]\cup V_z\times \{0,1\}$ and which is homotopic to $H'$.
    Hence, $H$ yields a smooth strong deformation retraction of $V_z$.
    Let $\Gamma_z\colon V_z\to {P}'M$ be the map $\Gamma_z(p)(t) = H(p,1-t)$, i.e. the path $\Gamma_z(p)\in {P}'M$ is a smooth path from $z$ to $p\in V_z$.
    Choose a point $e_0\in E_{x_0}$ in the fiber over the basepoint $x_0\in M$ and define $\psi_z\colon E|_{V_z}\to G\times V_z$ by
    $$   \psi_z (u) =  \big(     r(\mathcal{P}^{\omega}_{\gamma_{x_0\to z}\circ \Gamma_z(q(u))} (e_0), u), q(u)  \big) \quad \text{for}\,\,\, u\in E|_{V_z}      $$
    with $r\colon E\times_M E\to G$ the ratio map.
    This is a local trivialization of $q\colon E\to M$.
    Let $z,a\in\mathrm{Crit}(f)$.
    We want to compute the change of local trivializations $\psi_z\circ \psi_a^{-1}\colon V_z\cap V_a\times G\to V_z\cap V_a\times G$.
    One checks that 
    $$   \psi_a^{-1}(e,p) = \mathcal{P}^{\omega}_{\gamma_{x_0\to a}\circ \Gamma_a(p)}(e_0)     $$
    for $p\in V_a$ and $e\in G$ the neutral element.
    We compute for $p\in V_a\cap V_z$ that
    \begin{eqnarray*}
       \psi_z \circ \psi_a^{-1} (e,p)
        &=& \psi_z\big( \mathcal{P}^{\omega}_{\gamma_{x_0\to a}\circ \Gamma_a(p)} (e_0)  \big) \\
        &=& \big(   r(  \mathcal{P}^{\omega}_{\gamma_{x_0\to z}\circ \Gamma_z(p)} (e_0) ,\mathcal{P}^{\omega}_{\gamma_{x_0\to a}\circ \Gamma_a(p)} (e_0)), p  \big)
 .    \end{eqnarray*}
Hence, the transition function is given by
\begin{eqnarray*}
    \psi_{z,a}(p) &=&  r(  \mathcal{P}^{\omega}_{\gamma_{x_0\to z}\circ \Gamma_z(p)} (e_0) ,\mathcal{P}^{\omega}_{\gamma_{x_0\to a}\circ \Gamma_a(p)} (e_0)) \\
    &=& r \big(     e_0, (\mathcal{P}^{\omega}_{\gamma_{x_0\to z}\circ \Gamma_z(p)})^{-1} (\mathcal{P}^{\omega}_{\gamma_{x_0\to a}\circ \Gamma_a(p)} (e_0))    \big) \\
    &=& r\big( e_0,   \mathcal{P}^{\omega}_{\gamma_{x_0\to a}\circ \Gamma_a(p)\circ \overline{\Gamma}_z(p)\circ \gamma_{z\to x_0}}(e_0)\big) \\
    &=& \mathrm{hol}^{\omega}_{e_0} (\gamma_{x_0\to a}\circ \Gamma_a(p)\circ \overline{\Gamma}_z(p)\circ \gamma_{z\to x_0}) .
\end{eqnarray*}
where we used the invariance of the ratio map with respect to parallel transport, see equation \eqref{eq_ratio_map_invariant_parallel}.
Let $\Phi_{\psi}\colon \mathrm{mor}(\mathcal{M}_f)\to G$ be the transport function induced by the local trivializations $\psi_z$.
The explicit construction is given in Lemma \ref{lemma_trivializations_induce_pre-transport} and Proposition \ref{prop_pre-transport_induces_transport}.
Let $u\in \overline{\mathcal{L}}(x,y)$ for critical points $x,y\in\mathrm{Crit}(f)$.
Assume that the maximal sequence of indices as in the proof of Proposition \ref{prop_pre-transport_induces_transport} is given by $|y| = i_0< i_1< \ldots < i_m = |x|$.
By construction we have
\begin{eqnarray*}
    \Phi_{\psi}(u) &=&    \psi_{c_1,x}(\pi(u,\rho_{x,y}^{i_{m-1}+1}(u))) \cdot \ldots \cdot \psi_{y,c_{m-1}}(\pi(u,\rho_{x,y}^{i_0+1}(u))) \\
    &=&    \mathrm{hol}_{e_0}^{\omega} \Big(  \gamma_{x_0\to x}\circ \Gamma_x(\pi(u,\rho_{x,y}^{i_{m-1}+1}(u))) \circ \overline{\Gamma}_{c_1}( \pi(u,\rho_{x,y}^{i_{m-1}+1}(u)))  \circ \\  &   & \gamma_{c_1\to x_0} \circ \gamma_{x_0\to c_1} \circ \Gamma_{c_1}(\pi(u,\rho_{x,y}^{i_{m-2}+1}(u))) \circ \overline{\Gamma}_{c_2}(\pi(u,\rho_{x,y}^{i_{m-2}+1}(u))) \circ \gamma_{c_2\to x_0} \circ \ldots \circ  \\   &  &  \gamma_{x_0\to c_{m-1}}\circ \Gamma_{c_{m-1}}(\pi(u,\rho_{x,y}^{i_0+1}(u))) \circ \overline{\Gamma}_{y} (\pi(u,\rho_{x,y}^{i_0+1}(u))) \circ \gamma_{y\to x_0}    \Big)  
\end{eqnarray*}
In the last expression the consecutive paths $\gamma_{c_i\to x_0}\circ \gamma_{x_0\to c_i}$ cancel out by equation \eqref{eq_parallel_transport_inverse}.
Recall the explicit parametrizations $\gamma_u\colon [f(y),f(x)]\to M$ of the closure of the flow line $u$ in Section \ref{subsec_morse_flow}.
Furthermore note that for $i\in \{0,\ldots, m-1\}$ the composition 
\begin{equation}\label{eq_gamma_which_we_want_to_homotope}
    \overline{\Gamma}_{c_i}( \pi(u,\rho_{x,y}^{i_{m-i}+1}(u)))  \circ  \Gamma_{c_i}(\pi(u,\rho_{x,y}^{i_{m-i-1}+1}(u)))
\end{equation} which appears in the expression for $\Phi_{\psi}(u)$ is a path from $\pi(u,\rho_{x,y}^{i_{m-i}+1}(u)) = \gamma_u(\rho_{x,y}^{i_{m-i}+1}(u))$ to $\pi(u,\rho_{x,y}^{i_{m-i-1}+1}(u)) = \gamma_u(\rho_{x,y}^{i_{m-i-1}+1}(u))$.
We want to homotope the path in equation \eqref{eq_gamma_which_we_want_to_homotope} to the path
\begin{eqnarray}\label{eq_composition_which_we_want}      \overline{\gamma_{u}|_{ [\tfrac{1}{2}(\rho_{x,y}^{i_{m-i-1}+1}(u) + \rho_{x,y}^{i_{m-i}+1}(u)),  \rho_{x,y}^{i_{m-i}+1}(u)]}} \circ \overline{\Gamma}_{c_i}(\pi (u, \tfrac{1}{2}(\rho_{x,y}^{i_{m-i-1}+1}(u) + \rho_{x,y}^{i_{m-i}+1}(u)))) \circ  \\ {\Gamma}_{c_i}(\pi (u, \tfrac{1}{2}(\rho_{x,y}^{i_{m-i-1}+1}(u) + \rho_{x,y}^{i_{m-i}+1}(u))))   \circ  \overline{\gamma_{u}|_{ [\rho_{x,y}^{i_{m-i-1}+1}(u),  \tfrac{1}{2}(\rho_{x,y}^{i_{m-i-1}+1}(u) + \rho_{x,y}^{i_{m-i}+1}(u))} ]} .\nonumber  \end{eqnarray}
See Figure \ref{fig_isotopy_transport} for a sketch.
Note that parallel transport along the middle two paths in the composition of paths in equation \eqref{eq_composition_which_we_want} cancels out so that we will be left with parallel transport along $\overline{\gamma_u|_{[\rho_{x,y}^{i_{m-i-1}+1}(u) , \rho_{x,y}^{i_{m-i}+1}(u)]}}$.
We write down an explicit homotopy.
We define the shorthand notation $\beta^i(u) = \tfrac{1}{2}(\rho_{x,y}^{i_{m-i-1}+1}(u) + \rho_{x,y}^{i_{m-i}+1}(u))\in \mathbb{R}$ and for $\tau\in [0,1]$ let $$\xi^i(u,\tau) = (1-\tau)\rho_{x,y}^{i_{m-i}+1}(u) + \tau \beta^i(u) \quad \text{and}\quad  \kappa^i(u,\tau) = (1-\tau)\rho_{x,y}^{i_{m-i-1}+1}(u) + \tau\beta^i(u) .$$
Note that for all $\tau\in [0,1]$ the points $\pi(u,\xi^i(u,\tau))$ and $\pi(u,\kappa^i(u,\tau))$ lie in $V_{c_i}$ by construction of the separation functions, see Section \ref{subsec_separation_functions}.
For $\tau\in [0,1]$ define the path
$$ h^i(u,\tau) = \overline{\gamma_u|_{ [\xi^i(u,\tau),\rho_{x,y}^{i_{m-i}+1}(u)]}    }  \circ \overline{\Gamma}_{c_i}\big(   \pi(u, \xi^i(u,\tau))  \big) \circ \Gamma_{c_i}(\zeta^i(u,\tau)) \circ \overline{\gamma_u|_{[\rho_{x,y}^{i_{m-i-1}+1}(u), \kappa^i(u,\tau)]}} . $$
One checks that this yields a homotopy between the paths in equations \eqref{eq_gamma_which_we_want_to_homotope} and \eqref{eq_composition_which_we_want}.
\begin{figure}[t]
\centering
\includegraphics[scale=0.8]{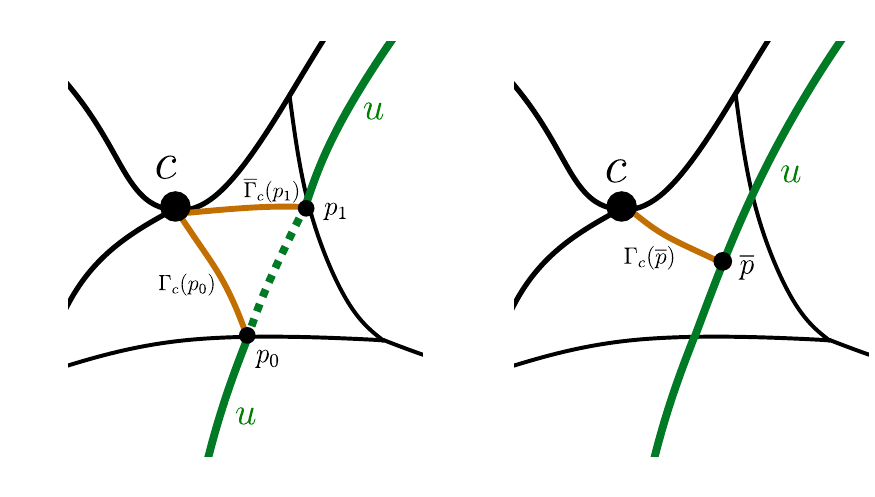}
\caption{Sketch of the construction of the isotopy between $\Phi^{\psi}(u)$ and $\Phi^{\omega}(u)$ for $u\in\overline{\mathcal{L}}(x,y)$ with $|x|-|y| = 2$. In the left figure we do parallel transport along the part of $u$ from $x$ to $p_1 = \pi(u,\rho_{x,y}^{i_1+1}(u))$ followed by parallel transport along $\overline{\Gamma}_c(p_1)\circ \Gamma_c(p_0)$ for $p_0 = \pi(u,\rho_{x,y}^{i_0+1}(u))$ and by parallel transport along $u$ from $p_0$ to $y$. By sliding the points $p_0$ and $p_1$ towards each other along $u$ we obtain the point $\overline{p} = \pi(u, \tfrac{1}{2}(\rho_{x,y}^{i_0+1}(u) + \rho_{x,y}^{i_1+1}(u)))$.
We homotope the parallel transport along the path in the left picture to parallel transport along $u$ from $x$ to $\overline{p}$, then going back and forth along $\Gamma_c(\overline{p})$ followed by parallel transport along the remainder of $u$. Since the parallel transport along $\overline{\Gamma}_c(\overline{p})\circ \Gamma_c(\overline{p})$ is trivial, this is just parallel transport along $u$.}
\label{fig_isotopy_transport}
\end{figure}
Using these homotopies we define 
\begin{eqnarray*}
    \mathcal{H}_{x,y}(u,\tau) &=& \mathrm{hol}_{e_0}^{\omega} \big(  \gamma_{x_0\to x}\circ \Gamma_x(\pi(u,\rho_{x,y}^{i_{m-1}+1(u)})) \circ  h^1(u,\tau) \circ \ldots \circ \\ &  &  \hphantom{ergroisrhsrthejrg}  h^{m-1}(u,\tau) \circ \overline{\Gamma}_y(\pi(u,\rho_{x,y}^{i_0+1}(u) \circ \gamma_{c_{y\to x_0}}  \big)  
\end{eqnarray*}
for $\tau\in [0,1]$.
By construction we have that $$\Gamma_x(u,\rho_{x,y}^{i_{m-1}+1}(u)) = \overline{\gamma_u|_{[\rho_{x,y}^{i_{m-1}+1}(u), f(x)]}} \quad \text{and}\quad \overline{\Gamma }_y(\pi(u,\rho_{x,y}^{i_0+1}(u))) = \overline{\gamma_u|_{[f(y),\rho_{x,y}^{i_0+1}(u)]} }  . $$
Hence, for $\tau = 0$ we have $\mathcal{H}_{x,y}(u,0)  =\Phi_{\psi}(u)$ and for $\tau= 1$ we obtain 
$$   \mathcal{H}_{x,y}(u,1) = \mathrm{hol}_{e_0}^{\omega}(\gamma_{x_0\to x} \circ \overline{\gamma}_u \circ \gamma_{y\to x_0}) = \Phi^{\omega}(u)  . $$
One further checks that for fixed $\tau\in [0,1]$ the map $u\mapsto \mathcal{H}_{x,y}(u,\tau)$ preserves composition of flow lines and therefore the map $\mathcal{H}\colon \mathrm{mor}(\mathcal{M}_f)\times [0,1]\to G$ is an isotopy between the transport function $\Phi_{\psi}$ and $\Phi^{\omega}$.
\end{proof}

Together with Corollary \ref{cor_same_bundle_equialent_tr_func} we obtain the following.

\begin{cor}\label{cor_equivalence_to_parallel_transport_function}
    Let $(M,f,X)$ be Morse data and let $q\colon E\to M$ be a smooth $G$-principal bundle with a principal connection $\omega$.
    If a transport function $\Phi\colon \mathrm{mor}(\mathcal{M}_f)\to G$ induces a $G$-principal bundle $q_{\Phi}\colon \mathcal{E}\to M$ which is isomorphic to $q\colon E\to M$ as a continuous $G$-principal bundle, then $\Phi$ is equivalent to $\Phi^{\omega}$.
\end{cor}

For connected Lie groups we obtain the following stronger statement.
\begin{cor}
    Let $G$ be a connected Lie group and let $(M,f,X)$ be Morse data.
    Every isotopy class of transport functions contains a function of the type $\Phi^{\omega}\colon \mathrm{mor}(\mathcal{M}_f)\to G$ for a principal connection $\omega$.
\end{cor}
\begin{proof}
    Let $[\Phi]$ be an isotopy class of transport functions.
    By Theorem \ref{theorem_isotopic_tr_functions_yield_isomorphic_bundles} an isotopy class of transport functions corresponds to a isomorphism class of $G$-principal bundles in the topological category.
    By \cite{muller2009equivalences} every continuous $G$-principal bundle is continuously isomorphic to a smooth $G$-principal bundle.
    Therefore the isomorphism class of $G$-principal bundles corresponding to $[\Phi]$ contains a smooth representative $q\colon E\to M$.
    If we choose a principal connection $\omega$ for $q$, then by Theorem \ref{theorem_parallel_transport_induces_og_bundle} the transport function $\Phi^{\omega}$ lies in the isotopy class $[\Phi]$.
\end{proof}

Next, we show that compatible sections exist for $\Phi^{\omega}$.

\begin{prop}\label{prop_transport_smooth_admits_comp_sections}
    Let $(M,f,X)$ be Morse data and $q\colon E\to M$ be a smooth $G$-principal bundle with principal connection $\omega$.
    The transport function $\Phi^{\omega}\colon \mathrm{mor}(\mathcal{M}_f)\to G$ admits compatible sections.
\end{prop}
\begin{proof}
    Let $x\in \mathrm{Crit}(f)$ be a critical point.
    We recursively define maps $\Gamma_x\colon \overline{W}_u(x) \to P'_{x\to M}M$ by associating to $a\in W^u(x)$ the unbroken flow line $u$ from $x\to a$ parametrized on $[0,f(x)-f(a)]$ and to an element $(u,a)\in\partial \overline{W}^u(x)$ the concatenation $\overline{\gamma}_u\circ \Gamma_y(a)\in P'_{x\to M}M$.
    Define $\phi_x\colon \overline{W}_u(x)\to E$ by
    $$   \phi_x(b) =  \mathcal{P}^{\omega}_{\Gamma_x(b)}(\mathcal{P}_{\gamma_{x_0\to x}}^{\omega}(e_0)) =\mathcal{P}^{\omega}_{\gamma_{x_0\to x}\circ \Gamma_x(b)}(e_0)  \quad \text{for}\,\,\, b\in \overline{W}_u(x)  .   $$
    Here, $\gamma_{x_0\to x}\in P'_{x_0\to x}M$ is the fixed path from $x_0$ to $x$ as before.
    It is clear that $\phi_x$ satisfies $q\circ \phi_x = g_x$ for $g_x\colon \overline{W}_u(x)\to M$ the map as in Section \ref{sec_assoc_bundles}.
    We need to show that for $(u,a)\in\overline{\mathcal{L}}(x,y)\times \overline{W}_u(y) \subseteq  \partial\overline{W}^u(x)$ we have $\phi_x(u,a) =\Phi_{x,y}(u)\cdot \phi_y(a) $.
    Let $(u,a)\in \overline{\mathcal{L}}(x,y)\times W^u(y)$, then we have
    \begin{eqnarray*}
        \phi_x(u,a) &=& \mathcal{P}^{\omega}_{\gamma_{x_0\to x}\circ \Gamma_x(u,a)}(e_0) \\
        &=& \mathcal{P}^{\omega}_{\gamma_{x_0\to x}\circ \overline{\gamma}_u\circ \gamma_{y\to x_0}\circ \gamma_{x_0\to y}\circ \Gamma_y(a)} (e_0) \\
        &=& \mathcal{P}^{\omega}_{\gamma_{x_0\to y}\circ \Gamma_y(a)} \big(   \mathcal{P}^{\omega}_{\gamma_{x_0\to x}\circ \overline{\gamma}_u\circ \gamma_{y\to x_0}}  (e_0) \big) \\
        &=& \mathcal{P}^{\omega}_{\gamma_{x_0\to y}\circ \Gamma_y(a)} ( \Phi_{x,y}^{\omega}(u)\cdot e_0) \\
        &=& \Phi_{x,y}^{\omega}(u) \cdot \mathcal{P}^{\omega}_{\gamma_{x_0\to y}\circ \Gamma_y(a)}(e_0) \\
        &=& \Phi_{x,y}^{\omega}(u)\cdot \phi_y(a) .
    \end{eqnarray*}
    Here, we used the left-invariance of parallel transport.
    This show that the $\{\phi_x\}$ are compatible sections for $\Phi^{\omega}$.
\end{proof}

\begin{cor}\label{cor_lie_group_all_comp_sections_exist}
    Let $(M,f,X)$ be Morse data and let $q\colon E\to M$ be a smooth $G$-principal bundle with $G$ a connected Lie group.
    Then any transport function $\Phi\colon \mathrm{mor}(\mathcal{M}_f)\to G$ which induces $q\colon E\to M$ admits compatible sections.
\end{cor}
\begin{proof}
    Let $\omega$ be a principal connection for the smooth bundle $q$.
    The transport function $\Phi^{\omega}$ induces a $G$-principal bundle isomorphic to $q\colon E\to M$ by Theorem \ref{theorem_parallel_transport_induces_og_bundle}.
    Moreover, $\Phi^{\omega}$ admits compatible sections by Proposition \ref{prop_transport_smooth_admits_comp_sections}.
    Every transport function $\Phi\colon \mathrm{mor}(\mathcal{M}_f)\to G$ which induces the isomorphism class of $q$ is equivalent to $\Phi^{\omega}$ by Corollary \ref{cor_same_bundle_equialent_tr_func}.
    By Proposition \ref{prop_isotop_transport_functions_both_admit_sections_if_one_does} the transport function $\Phi$ therefore also admits compatible sections.
\end{proof}

Together with Theorem \ref{theorem_assoc_bundle} we obtain the following.

\begin{cor}
    Let $(M,f,X,o,\{s_{x,y}\})$ be enhanced Morse data and let $q\colon E\to M$ be a smooth $G$-principal bundle with $G$ connected.
    Let $\Phi\colon \mathrm{mor}(\mathcal{M}_f)\to G$ be a transport function which induces the bundle $q\colon E\to M$.
    Let $\psi\colon F\times G\to F$ be a continuous $G$-action.
    Then the homology of the DG Morse complex $\mathrm{C}_{\bullet}(M;(\Phi,\mathrm{C}_{\bullet}(F)))$ is isomorphic to the homology of the total space of the associated bundle $F\times_G E\to M$.
\end{cor}

We now turn to the comparison of the DG Morse complex $\mathrm{C}_{\bullet}(M;(\Phi^{\omega},\mathrm{C}_{\bullet}(F)))$ to the DG Morse complex considered in \cite{barraud2025morse}.
Let $q\colon E\to M$ be a smooth $G$-principal fiber bundle with principal connection $\omega$
and let $\psi\colon F\times G\to F$ be a $G$-right action.
As before, we fix a point $e_0\in E$ over the basepoint $x_0\in M$.
Consider the induced bundle $\mathcal{E}= F\times_G E\to M$.
We note that the horizontal lifts for $E$ also induce lifts of paths from $M$ to $\mathcal{E}$.
Indeed, if $\gamma\colon [0,1]\to M$ is a piecewise smooth curve, then as before we denote by $\widetilde{\gamma}_u\colon [0,1]\to E$ the horizontal lift with respect to $\omega$ such that $\widetilde{\gamma}_u( 0) = u\in E_{\gamma(0)}$.
If $v = [f,u]\in \mathcal{E}$, i.e. $f\in F$ and $u\in E$, then the map $$\widehat{\gamma}_v \colon [0,1]\to \mathcal{E},\quad  t\mapsto [f,\gamma_{u}(t)]$$ is a lift of $\gamma$ to $\mathcal{E}$ with $\widehat{\gamma}_v(0) = v$.
Hence we obtain a lifting function 
$$    \mathring{{\xi}}\colon   \mathcal{E} \tensor[{_{q}}]{{\times}}{_{\mathrm{ev}_0}} {P}'M \to {P}'\mathcal{E} , \quad ([f,u],\gamma)\mapsto  \widehat{\gamma}_{[f,u]}  .   $$
By evaluating this at the endpoint we obtain a map
$$   \mathring{\Xi}\colon \mathcal{E}\tensor[{_{q}}]{{\times}}{_{\mathrm{ev}_0}} {P}'M\to \mathcal{E}    . $$
Recall the notion of a transitive lifting function from Section \ref{subsec_prelim_morse_DG}.
\begin{lemma}
    The map $\mathring{\Xi}\colon \mathcal{E}\tensor[{_{q}}]{{\times}}{_{\mathrm{ev}_0}} {P}'M\to \mathcal{E}$, $([(f,u)],(\gamma,a))\mapsto \widehat{\gamma}_{[(f,u)]}(a)$ is a transitive lifting function.
\end{lemma}
\begin{proof}
    The fact that $q\circ \mathring{\Xi} = \mathrm{ev}_1\circ \mathrm{pr}_2$ is clear by construction.
    Let $[f,u]\in \mathcal{E}$ and let $x = q(u)\in M$.
    Let $c_x$ be the constant path at $x$.
    The horizontal lift $\widetilde{c}_u\in {P}'E$ is the constant path at $u$ as well and thus we obtain
    $$    \mathring{\Xi}([(f,u)], c_x) = [(f, \widetilde{c}_u(1))] = [(f,u)]    $$
    as required.
    Finally, let $(\gamma,a),(\sigma,b)\in {P}'M$ be concatenable paths and let $[(f,u)]\in\mathcal{E}$.
    We have
    $$   \mathring{\Xi}([(f,u)],\gamma\circ \sigma) = [f, \widetilde{\gamma\circ \sigma}_u(a+b)] = [f, \widetilde{\sigma}_{\widetilde{\gamma_u}(a)} (b)]      $$
    since horizontal lifts commute with concatenation in the sense that
    $   \widetilde{\gamma\circ \sigma}_u = \widetilde{\gamma}_u \circ \widetilde{\sigma}_{\widetilde{\gamma}_u(a)}     $.
    On the other hand we have
    $$   \mathring{\Xi}(\mathring{\Xi}([(f,u)],\gamma),\sigma) = \mathring{\Xi}( [f, \widetilde{\gamma}_u(a)],\sigma) = [f,\widetilde{\sigma}_{\widetilde{\gamma}_u(a)}(b) ] .    $$
    Hence, $\mathring{\Xi}$ is a transitive lifting function.
\end{proof}
We now fix the identification $F\cong \mathcal{E}_{x_0}$ through the homeomorphism
\begin{equation}\label{eq_identification_tau}  \tau\colon  F\to \mathcal{E}_{x_0}    ,\quad  f\mapsto [f,e_0] . 
\end{equation}
In principle would like to compare the DG Morse complex as in \cite{barraud2025morse} with $\mathrm{C}_{\bullet}(F)$ as a $\mathrm{C}_{\bullet}(\Omega M)$-module with the DG Morse complex using the transport function.
The main technical issue here is that the \emph{evaluation maps} $q_{x,y}\colon \overline{\mathcal{L}}(x,y)\to \Omega M$ are not guaranteed to take values in the piecewise smooth loops since they are defined by collapsing a tree.
As we want to work with piecewise smooth paths we introduce a DG Morse complex very much in the style of the complexes used in \cite{barraud2025morse} but which avoids the issue of the collapse maps.

There is a topological category $P'_{f}$ with objects the critical points of $f$ and morphism the piecewise smooth Moore path spaces $P'_{x\to y}M$ for $x,y\in\mathrm{Crit}(f)$.
Consequently, the cubical chains on these morphism spaces yield a DG category $\mathcal{CP}_f$.
Define maps $\zeta_{x,y}\colon \overline{\mathcal{L}}(x,y)\to P'_{x\to y}M$ by $\zeta_{x,y}(u) = (\sigma_u, f(x)-f(y))$ with $\sigma_u\colon [0,f(x)-f(y)]\to M$ defined by $\sigma_u(t) = u\cap f^{-1}(f(x)-t)$.
We define a collection of chains $m_{x,y} = (\zeta_{x,y})_*(s_{x,y})\in \mathrm{C}_{|x|-|y|-1}(P'_{x\to y}M) $ with $s_{x,y}\in \mathrm{C}_{|x|-|y|-1}(\overline{\mathcal{L}}(x,y))$ a representing chain system, see Section \ref{subsec_prelim_morse_DG}.
This collection of chains $\{m_{x,y}\}$ satisfies
$$   \partial m_{x,y} =   \sum_{\substack{z \in \mathrm{Crit}(f)\\ |y|  < |z| <|x|}}   (-1)^{|x|-|z|} m_{x,z} \cdot m_{z,y}      $$
where the product is the composition of morphisms in the DG category $\mathcal{CP}_f$.
Let $q\colon E\to M$ be a fibration with a transitive lifting function $\Xi\colon {E}\tensor[{_{q}}]{{\times}}{_{\mathrm{ev}_0}} {P}'M\to {E}$.
We define a chain complex by considering the graded $R$-module 
$$       \mathrm{C}_{\bullet}'(M; E) := \bigoplus_{x\in\mathrm{Crit}(f)} \mathrm{C}_{\bullet}(E_x)\otimes R\langle x\rangle .  $$
For the definition of the differential define the restrictions $ {\Xi}_{x,y} := {\Xi}|_{{E}_x\times P'_{x\to y}M} \colon {E}_x\times P'_{x\to y}M \to {E}_y$. 
The differential is defined by 
$$     \partial(\alpha\otimes x) = (\partial \alpha)\otimes x + (-1)^{|\alpha|} \sum_{\substack{y\in \mathrm{Crit}(f)} \\ |y| < |x|} ({\Xi}_{x,y})_*(\alpha\times m_{x,y}) \otimes y .          $$
Similarly to what we have seen before, this yields a differential.
The author notes that many of the constructions in \cite{barraud2025morse} can be transferred from the situation of a DG module over $\mathrm{C}_{\bullet}(\Omega M)$ to the present situation of considering the DG category $\mathcal{CP}_f$.
Of course, it is often more appealing to work with a DG algebra and the corresponding DG module compared to a DG category.
However, by working with the complex $\mathrm{C}_{\bullet}'(M;E)$ we can work with piecewise smooth paths and thus compare the chain complex $\mathrm{C}_{\bullet}'(M;\mathcal{E})$ with $\mathcal{E} = F\times_G E\to M$ to the DG Morse chain complex induced by the transport function $\Phi^{\omega}$.

\begin{theorem}\label{theorem_isomorphic_complexes_connection}
    Let $(M,f,X,o,\{s_{x,y}\})$ be enhanced Morse data and let $q\colon E\to M$ be a smooth $G$-principal bundle with principal connection $\omega$.
    Let $\psi\colon F\times G\to F$ be a continuous right action of $G$ and let $\mathcal{E} = F\times_G E\to M$ be the associated bundle.
    Then the chain complexes $\mathrm{C}_{\bullet}(M;(\Phi^{\omega},\mathrm{C}_{\bullet}(F)))$ and $\mathrm{C}_{\bullet}'(M;\mathcal{E})$ are isomorphic.
\end{theorem}
\begin{proof}
    We write $n_{x,y} = (\Phi^{\omega}_{x,y})_*(s_{x,y})\in \mathrm{C}_{\bullet}(G)$ for the twisting cocycles of the DG Morse complex involving the transport function and $m_{x,y}\in \mathrm{C}_{\bullet}(P'_{x\to y}M)$ for the chains in the DG category $\mathcal{CP}_f$ as above.
    For $x\in \mathrm{Crit}(f)$ let $\rho_x\colon \mathcal{E}_x\to F$ be the homeomorphism given by
    $$     \rho_x( [f,e])  =  \tau^{-1} (  \mathring{\Xi}([f,e], \gamma_{x\to x_0})      )     $$
    with $\tau$ being the homeomorphism $F\cong \mathcal{E}_{x_0}$ as in \eqref{eq_identification_tau}.
    We define a map $\eta \colon \mathrm{C}_{\bullet}'(M; \mathcal{E}) \to  \mathrm{C}_{\bullet}(M;(\Phi^{\omega}, \mathrm{C}_{\bullet}(F))) $ by setting
    $$  \eta(\alpha\otimes x) = (\rho_x)_*(\alpha) \otimes x \quad \text{for}\,\,\,x\in\mathrm{Crit}(f),\,\, \alpha\in \mathrm{C}_{\bullet}(\mathcal{E}_x)    .     $$
    It is clear that this is an isomorphism of graded modules.
    We just need to verify that $\eta$ is indeed a chain map.
    Let $x\in \mathrm{Crit}(f)$ and $\alpha\in\mathrm{C}_{\bullet}(\mathcal{E}_x)$.
    We compute that on the one hand
    \begin{eqnarray*}
        \partial (\eta(\alpha\otimes x) ) &=& \partial ((\rho_x)_*(\alpha)\otimes x) \\ &=&
        (\rho_x)_*(\partial \alpha)\otimes x + (-1)^{|\alpha|} \sum_{|y|<|x|} \psi_* ((\rho_x)_* \alpha\times n_{x,y})\otimes y  \\
        &=&   (\rho_x)_*(\partial \alpha)\otimes x + (-1)^{|\alpha|} \sum_{|y|<|x|} \psi_* ((\rho_x)_*\alpha \times (\Phi^{\omega}_{x,y})_* (s_{x,y})) \otimes y   
    \end{eqnarray*}
    while on the other hand
    \begin{eqnarray*}
           \eta(\partial (\alpha\otimes x)) & =  & ((\rho_x)_*(\partial \alpha))\otimes x + (-1)^{|\alpha|} \sum_{|y|<|x|} (\rho_y)_* \big( (\mathring{\Xi}_{x,y})_*(\alpha\times m_{x,y})\big) \otimes y    \\
           &=& 
            ((\rho_x)_*(\partial \alpha))\otimes x + (-1)^{|\alpha|} \sum_{|y|<|x|} (\rho_y)_* \big( (\mathring{\Xi}_{x,y})_*(\alpha\times (\zeta_{x,y})_*(s_{x,y}) )\big) \otimes y  .
    \end{eqnarray*}
    Hence, it is sufficient to show that the maps 
    \begin{equation}\label{eq_two_maps}
          \rho_y  \circ  \mathring{\Xi}_{x,y} \circ (\mathrm{id}\times \zeta_{x,y})\colon \mathcal{E}_x\times \overline{\mathcal{L}}(x,y)\to  F  \quad \text{and}\quad \psi \circ (\rho_x \times \Phi^{\omega}_{x,y})      \colon \mathcal{E}_x\times \overline{\mathcal{L}}(x,y) \to F    
    \end{equation}
    agree.
    Let $[f,e]\in \mathcal{E}_x$ and $u\in\overline{\mathcal{L}}(x,y)$.
    There is a unique $g\in G$ such that $\mathcal{P}^{\omega}_{\gamma_{x\to x_0}}(e) = g\cdot e_0$.
    We have
    \begin{eqnarray*}
        \tau \circ \rho_y \circ \mathring{\Xi}_{x,y}([f,e], \zeta_{x,y}(u)) &=& 
        \mathring{\Xi}(\mathring{\Xi}_{x,y}([f,e], \sigma_u), \gamma_{y\to x_0}) \\
        &=& [f, \mathcal{P}^{\omega}_{\sigma_u\circ \gamma_{y\to x_0}}(e)] \\
        &=& [ f, \mathcal{P}^{\omega}_{\gamma_{x_0\to x}\circ \sigma_u\circ \gamma_{y\to x_0}} (\mathcal{P}^{\omega}_{\gamma_x\to x_0}(e) ]  \\
        &=& [f, \mathcal{P}^{\omega}_{\gamma_{x_0\to x}\circ \sigma_u \circ \gamma_{y\to x_0}} (g\cdot e_0)  ] \\
        &=& 
        [f, \mathrm{hol}_{g\cdot e_0} ( {\gamma_{x_0\to x}\circ \sigma_u\circ \gamma_{y\to x_0}}) \cdot g \cdot e_0  ] \\ 
        &=& [\psi (f,  \mathrm{hol}_{g\cdot e_0} ( {\gamma_{x_0\to x}\circ \sigma_u \circ \gamma_{y\to x_0}})) , g\cdot e_0] \\
        &=& [ \psi( f, g \cdot \mathrm{hol_{e_0}}(  {\gamma_{x_0\to x}\circ \sigma_u\circ \gamma_{y\to x_0}}) g^{-1} ), g\cdot e_0] \\
        &=& [\psi( f , g \cdot \mathrm{hol}_{e_0} (  {\gamma_{x_0\to x}\circ \sigma_u\circ \gamma_{y\to x_0}})) , e_0] \\
        &=& 
        [\psi( f, g \cdot \Phi^{\omega}_{x,y}(u) ), e_0] 
    \end{eqnarray*}
    where we used equation \eqref{equation_transformation_formula_holonomy}.
    On the other hand we have 
    \begin{eqnarray*}
        \tau \circ \psi \circ (\rho_x\times \Phi^{\omega}_{x,y})( [f,e] , u) 
        &=& 
        \tau \circ \psi (  \tau^{-1}([ f, \mathcal{P}_{\gamma_{x\to x_0}} e] , \Phi_{x,y}^{\omega}(u)) \\
        &=& \tau \circ \psi( \tau^{-1} ( [ f, g\cdot e_0], \Phi^{\omega}_{x,y}(u)) \\
        &=& \tau \circ \psi( f, g \cdot \Phi^{\omega}_{x,y}(u))) \\
        &=& [\psi(f,g\cdot \Phi^{\omega}_{x,y}(u)), e_0 ] .
    \end{eqnarray*}
    Since $\tau$ is a homeomorphism this show that the two maps in \eqref{eq_two_maps} agree and it follows that $\eta$ is a chain map.
\end{proof}
\begin{remark}\label{remark_discrepancy_transport_bdho}
    Let $p\colon M\to M/\mathcal{Y}$ be the map collapsing the tree $\mathcal{Y}$ and let $\theta\colon M/\mathcal{Y}\to M$ be a homotopy inverse.
    The evaluation maps $q_{x,y}\colon \overline{\mathcal{L}}(x,y)\to \Omega M$ in \cite{barraud2025morse} are given by composing the path $\sigma_u\in P'_{x\to y}M$ with the map $\theta\circ p$. 
    If we assume that $q_{x,y}\colon \overline{\mathcal{L}}(x,y)\to \Omega M$ take values in the piecewise smooth loops then we could have used the loops $q_{x,y}(u)\in \Omega'M$ for $u\in\overline{\mathcal{L}}(x,y)$ in order to obtain a transport function $\widetilde{\Phi}^{\omega}$ instead of the transport function $\Phi^{\omega}$.
    Under this assumption one can compare the DG Morse complex as in \cite{barraud2025morse} to the DG Morse complex using the transport function $\widetilde{\Phi}^{\omega}$ with the same ideas as in the proof of Theorem \ref{theorem_isomorphic_complexes_connection}.    
    One finds again that the chain complexes are isomorphic.
    In general though the paths $(\theta \circ p)\circ \sigma_u \in \Omega M$ might not be smooth.
    
    By the Whitney approximation theorem, see \cite[Theorem 6.26]{lee:2013}, we can find a smooth map $\zeta\colon M\to M$ which is homotopic to $\theta \circ p$ and which agrees with $\theta \circ p$ on the given tree, i.e. $\zeta$ collapses the tree $\mathcal{Y}$ to the basepoint.
    If we define the evaluation maps $q_{x,y}\colon \overline{\mathcal{L}}(x,y)\to \Omega M$ using the map $\zeta$ we obtain piecewise smooth loops $\zeta\circ \sigma_u\in \Omega' M$.
    Hence, up to replacing $\theta \circ p\colon M\to M$ with a homotopic map $\zeta$ one can indeed achieve that the evaluation maps $q_{x,y}$ take value in the piecewise smooth loops.
    Because of the constructions in Theorem \ref{theorem_parallel_transport_induces_og_bundle} and its proof we however preferred to consider the paths $Q_{x,y}(u)\in \Omega'M$ for $u\in \overline{\mathcal{L}}(x,y)$.
\end{remark}

\section{Towards functoriality}\label{sec_functoriality}

We conclude by giving some remarks on the functoriality of the construction of DG Morse homology with transport functions.
Many aspects of the functoriality of DG Morse homology with coefficients in a $\mathrm{C}_{\bullet}(\Omega M)$-module are treated in \cite[Sections 9 and 10]{barraud2025morse}.
Furthermore, Riegel \cite{riegel2024chain}, \cite{riegel2025path} and Clivio \cite{clivio2025goresky} have advanced this investigation by studying specific situations arising in string topology.
In particular, Riegel shows that if $\mathcal{F}_{\bullet}$ and $\mathcal{G}_{\bullet}$ are $\mathrm{C}_{\bullet}(\Omega M)$ are DG-modules, then not only equivariant chain maps $\mathcal{F}_{\bullet}\to \mathcal{G}_{\bullet}$, but also morphisms $\mathcal{F}_{\bullet}\to \mathcal{G}_{\bullet}$ of $A_{\infty}$-modules over the strict DGA $\mathrm{C}_{\bullet}(\Omega M)$ induce chain maps in the DG Morse complexes.
Riegel and Clivio further show that these morphisms yield the \emph{correct} maps under the fibration theorem.
In this section we consider the situation of an associated bundle $F\times_G E\to M$ and we begin by treating direct maps and shriek maps in the fiber $F$.
We are particularly interested in shriek maps in the fiber for the following reason.
If $G$ is a compact Lie group, averaging over $G$ can yield a way to produce a $\mathrm{C}_{\bullet}(G)$-equivariant chain model of the shriek map $i_! \colon \mathrm{H}_{\bullet}(F)\to \mathrm{H}_{\bullet}(\mathring{F})$ for a $G$-equivariant embedding $\mathring{F}\hookrightarrow F$.
In the situation of $\mathrm{C}_{\bullet}(\Omega M)$-modules, it is not clear how to perform such an averaging process in general.
Clivio manages to describe shriek maps in the fiber in \cite{clivio2025goresky} in the specific situation of the string topology coproduct.
The treatment of shriek maps in this article follows Clivio's ideas.
Finally, we sketch some ideas about how direct maps in the base can be dealt with in the DG Morse complexes with transport functions.

\medskip
We begin by considering direct maps in the fiber.
In the situation of associated bundles, a $G$-equivariant map between the fibers induces a fiberwise map of the total spaces of the associated bundles.
The induced chain map in the DG Morse complex can be constructed straight-forwardly.
More precisely, let $q\colon E\to M$ be a $G$-principal bundle.
Let $F_1$ and $F_2$ be right $G$-spaces and assume that there is a $G$-equivariant map $\rho\colon F_1\to F_2$.
This induces a well-defined map of associated bundles
$$     \overline{\rho}\colon F_1\times_G E \to F_2\times_G E , \quad [a,u]\mapsto [\rho(a),u]     \quad \text{for}\,\,\, [a,u]\in F_1\times_G E  $$
which is moreover a fiberwise map.
Assume that $(M,f,X,o,\{s_{x,y}\})$ is enhanced Morse data and let $\Phi\colon \mathrm{mor}(\mathcal{M}_f)\to G$ be a transport function for $q$.
The map $\rho$ induces a map
$$    \widetilde{\rho}\colon \mathrm{C}_{\bullet}(M;(\Phi, \mathrm{C}_{\bullet}(F_1))) \to \mathrm{C}_{\bullet}(M; (\Phi,\mathrm{C}_{\bullet}(F_2)))  , \quad  \widetilde{\rho}(\alpha\otimes x) = (\rho_*(\alpha))\otimes x       $$
for $\alpha\in \mathrm{C}_{\bullet}(F_1)$, $x\in\mathrm{Crit}(f)$.
By \cite[Proposition 4.4]{barraud2025morse} this is a chain map which preserves the canonical filtrations.
\begin{prop}\label{prop_direct_maps_fiber}
    Let $q\colon E\to M$ be a $G$-principal bundle and $\rho\colon F_1\to F_2$ a $G$-equivariant map between $G$-spaces $F_1,F_2$.
    Let $(M,f,X,o,\{s_{x,y}\})$ be enhanced Morse data on $M$ and let $\Phi\colon \mathrm{mor}(\mathcal{M}_f)\to G$ be a transport function for $q$.
    Assume that all compatible sections for $\Phi$ exist.
    Under the isomorphisms $t_i\colon \mathrm{H}_{\bullet}(M;(\Phi,\mathrm{C}_{\bullet}(F_i)))\xrightarrow[]{\cong} \mathrm{H}_{\bullet}(F_i\times_G E)$, $i= 1,2$, the map induced by $\widetilde{\rho}$ in homology agrees with the map $\overline{\rho}_*\colon \mathrm{H}_{\bullet}(F_1\times_G E)\to \mathrm{H}_{\bullet}(F_2\times_G E)$.
\end{prop}
\begin{proof}
Let $p_1\colon F_1\times E\to F_1\times_G E$ and $p_2\colon F_2\times E\to F_2\times_G E$ be the canonical projections for the associated bundles.
Note that the diagram
$$
    \begin{tikzcd}
        F_1\times E\arrow[]{r}{p_1} \arrow[swap]{d}{\rho\times \mathrm{id}} & F_1\times_G E \arrow[]{d}{\overline{\rho}} \\
        F_2\times E\arrow[]{r}{p_2} & F_2\times_G E 
    \end{tikzcd}
$$
commutes.
The assertion follows from a direct computation using the definition of the maps $t_i$, see Theorem \ref{theorem_assoc_bundle}. 
\end{proof}
\begin{remark}
For the treatment of shriek maps we also need a relative version of the above Proposition.
Assume that $F_1,F_2$ are $G$-spaces and $F_1',F_2'$ are $G$-invariant subspace of $F_1,F_2$, respectively.
If $\rho\colon F_1\to F_2$ is a $G$-equivariant map which maps $F_1'\to F_2'$, then we obtain an induced map of pairs
$$     \overline{\rho}\colon  (F_1\times_G E, F_1'\times_G E)\to (F_2\times_G E, F_2'\times_G E) .      $$
Proposition \ref{prop_direct_maps_fiber} carries over directly to this situation, i.e. the chain map $$\widetilde{\rho}\colon \mathrm{C}_{\bullet}(M;(\Phi,\mathrm{C}_{\bullet}(F_1,F_1'))) \to \mathrm{C}_{\bullet}(M;(\Phi,\mathrm{C}_{\bullet}(F_2,F_2'))) , \quad \widetilde{\rho}(\alpha\otimes x) = (\rho_*(\alpha))\otimes x  $$
models the map induced by $\overline{\rho}$ in relative homology.
\end{remark}

\begin{example}
    Let $G$ be a Lie group and let $H\subseteq G$ be a closed subgroup such that the quotient $M = H\backslash G$ is compact.
    As seen in Example \ref{example_free_loop_space} this yields an $H$-principal bundle $q\colon G\to M$ and $H$ acts on the based loop space $\Omega M$.
    By Example \ref{example_free_loop_space} we have a fiberwise homeomorphism $f\colon \Omega M\times_H G\to \Lambda M$.
    Similarly, one can check that the fiber product $\Lambda M\times_M \Lambda M$ is the associated bundle $(\Omega M\times \Omega M)\times_H G$ where $H$ acts diagonally on $\Omega M\times \Omega M$.
    Note that there is a map $\overline{\phi}\colon \Lambda M\times_M \Lambda M\to \Lambda M $ which concatenates a pair $(\gamma,\sigma)\in \Lambda M\times_M \Lambda M$.
    This map is induced by the $H$-equivariant concatenation map $\phi\colon \Omega M\times \Omega M\to \Omega M$.
    Proposition \ref{prop_direct_maps_fiber} says that the map $\overline{\phi}_*\colon \mathrm{H}_{\bullet}(\Lambda M\times_M \Lambda M)\to \mathrm{H}_{\bullet}(\Lambda M)$ can be modeled in DG Morse homology by the chain map $\widetilde{\phi}\colon \mathrm{C}_{\bullet}(M;(\Phi,\mathrm{C}_{\bullet}(\Omega M\times \Omega M))) \to \mathrm{C}_{\bullet}(M;(\Phi,\mathrm{C}_{\bullet}(\Omega M))$.
    We note that the concatenation map $\overline{\phi}_*$ is one ingredient of the Chas-Sullivan product in string topology, see e.g. \cite{cohen2006string}.
    If $N$ is an arbitrary closed oriented manifold and one considers DG Morse homology with $\mathrm{C}_{\bullet}(\Omega N)$-modules as in \cite{barraud2025morse}, then the concatenation map $\phi\colon \Omega N\times \Omega N\to \Omega N$ does not strictly intertwine the $\Omega M$-actions and considerable technical effort has to be undertaken in order to remedy this problem, see \cite{riegel2024chain}.
\end{example}

We turn to shriek maps in the fiber, resp. shriek maps for the associated bundles.
Let $q\colon E\to M$ be a $G$-principal bundle.
Assume that $F_1\subseteq F_2$ is an inclusion of $G$-spaces.
We do not assume that $F_1$ or $F_2$ are finite-dimensional manifolds.
We say that $F_1\subseteq F_2$ admits a \emph{tubular neighborhood} $U\subseteq F_2$ of rank $k\in\mathbb{N}$ if there is a real vector bundle $p\colon P\to F_1$ of rank $k$ and a homeomorphism $\Gamma\colon P\to U$ onto the open set $U$ such that $\Gamma\circ s_0 = i \colon F_1\to F_2$ where $i\colon F_1\hookrightarrow F_2$ is the inclusion and $s_0\colon F_1\to P$ is the zero-section.
We say that $U$ is $G$-invariant, if it is a $G$-invariant subset of $F_2$ and if $G$ acts on $P$ by vector bundle isomorphisms such that $\Gamma$ is $G$-equivariant.
Note that we obtain three associated bundles, namely $F_2\times_G E, F_1\times_G E$ as well as $U\times_G E$.
One checks that $U\times_G E$ is a tubular neighborhood of rank $k$ of $F_1\times_G E$ inside $F_2\times_G E$.

The main example for this situation is the following.
Let $G$ be a compact Lie group and let $F_2$ be a smooth manifold with $F_1\subseteq F_2$ a $G$-invariant closed submanifold.
Then $F_1$ admits a $G$-invariant tubular neighborhood, see \cite[Theorem VI.2.2]{bredon1972introduction}.

Coming back to the general situation denote the inclusions by $i\colon F_1\hookrightarrow F_2$ and $j\colon F_1\times_G E\hookrightarrow F_2\times_G E$.
The goal is to find a chain model in the DG Morse complex for the Gysin map $j_!\colon \mathrm{H}_{\bullet}(F_2\times_G E)\to \mathrm{H}_{\bullet-k}(F_1\times_G E)$.
Recall that in singular homology the Gysin map $j_!$ is defined as the composition 
\begin{eqnarray}
    \label{eq_gysin_map_via_thom2} 
       &  & \mathrm{H}_{\bullet}(F_2\times_G E)\to \mathrm{H}_{\bullet}(F_2\times_G E, \sim F_1\times_G E) \xrightarrow[]{\cong} \mathrm{H}_{\bullet}(U\times_G E,\sim  F_1\times_G E) \xrightarrow[]{\tau\cap} \\ &   &  \nonumber \hphantom{aeorgnaeorgoaeorrgaoeirsrthsrthasetgaerggoaierhg} \mathrm{H}_{\bullet}(U\times_G E)\xrightarrow[\cong]{r_*} \mathrm{H}_{\bullet}(F_1\times_G E)     
\end{eqnarray}
where we write $(X,\sim Y)$ for a topological pair of the form $(X,X\setminus Y)$.
The second arrow in \eqref{eq_gysin_map_via_thom2} is given by excision, the third one is capping with the Thom class $\tau\in \mathrm{H}^k(U\times_G E,\sim F_1\times_G E)$ and the last arrow is induced by the retraction $r\colon U\times_G E\to F_1\times_G E$.
Heuristically speaking the chain map in the DG Morse complex should be defined by applying a chain model of the Gysin map $i_! \colon \mathrm{H}_{\bullet}(F_2)\to \mathrm{H}_{\bullet}(F_1)$ on the tensor factor $\mathrm{C}_{\bullet}(F_2)$ in the DG Morse complex $\mathrm{C}_{\bullet}(M;(\Phi,\mathrm{C}_{\bullet}(F_2)))  = \mathrm{C}_{\bullet}(F_2)\otimes R\langle \mathrm{Crit}(f)\rangle$.
Note that the difficulty here lies in the fact that we do not only need a chain map representing $i_!$, but we need every chain map to be $\mathrm{C}_{\bullet}(G)$-equivariant.

\begin{theorem}\label{theorem_gysin_map_chain_map}
    Let $q\colon E\to M$ be a $G$-principal bundle.
    Let $F_1\hookrightarrow F_2$ be an equivariant embedding of $G$-spaces admitting a $G$-invariant orientable tubular neighborhood $U$ of rank $k\in \mathbb{N}$.
    Further let $(M,f,X,o,\{s_{x,y}\})$ be enhanced Morse data on $M$ and let $\Phi\colon \mathrm{mor}(\mathcal{M}_f)\to G$ be a transport function for $E$.
    Assume that all compatible sections for $\Phi$ exist.
    Let $\iota\colon U\hookrightarrow U\times_G E$ be the embedding $x\mapsto [x,e_0]$ for a fixed point $e_0\in E$ and consider chains with field coefficients.
    
    If there is a cocycle $\tau\in \mathrm{C}^k(U\times_G E,\sim F_1\times_G E)$ representing the Thom class such that the chain map $\iota^* \tau\cap \cdot \colon \mathrm{C}_{\bullet}(U,U\setminus F_1)\to \mathrm{C}_{\bullet-k}(U)$ is $\mathrm{C}_{\bullet}(G)$-equivariant, then there exists a DG Morse chain model for the Gysin map $\mathrm{H}_i(F_2\times_G E)\to \mathrm{H}_{i-k}(F_1\times_G E)$.
\end{theorem}
\begin{proof}
    Recall from equation \eqref{eq_gysin_map_via_thom2} that the Gysin map in homology is given by the composition of four maps.
    We will find chain level versions of all four maps.

    The first map is induced by a map of pairs $(F_2\times_G E,\emptyset)\to (F_2\times_G E,\sim F_1\times_G E)$ which itself is induced by a $G$-equivariant map on the fiber pairs $(F_2,\emptyset)\to (F_2,\sim F_1)$.
    We thus get a chain map $\rho_1\colon \mathrm{C}_{\bullet}(M;(\Phi,\mathrm{C}_{\bullet}(F_2)))\to \mathrm{C}_{\bullet}(M;(\Phi,\mathrm{C}_{\bullet}(F_2,\sim F_1)))$ which models this map by Proposition \ref{prop_direct_maps_fiber}, i.e. the square 
    $$
        \begin{tikzcd}
            \mathrm{H}_{\bullet}(M;(\Phi,\mathrm{C}_{\bullet}(F_2))) \arrow[]{r}{t_1} \arrow[]{d}{\rho_1} &  \mathrm{H}_{\bullet}(F_2\times_G E)    \arrow[]{d}{}  \\ 
            \mathrm{H}_{\bullet}(M;(\Phi,\mathrm{C}_{\bullet}(F_2,\sim F_1))) \arrow[]{r}{t_2} & \mathrm{H}_{\bullet}(F_2\times_G E,\sim F_1\times_G E) 
        \end{tikzcd}
    $$
    commutes.
    Similarly, the fourth map is induced by the equivariant retraction on the fibers $U\to F_1$ and Proposition \ref{prop_direct_maps_fiber} thus yields a chain model $\rho_4\colon \mathrm{C}_{\bullet}(M;(\Phi,\mathrm{C}_{\bullet}(U)))\to \mathrm{C}_{\bullet}(M;(\Phi,\mathrm{C}_{\bullet}(F_1)))$ for this map.

    For the second map, note that the map of pairs $\zeta\colon (U\times_G E,\sim F_1\times_G E)\to (F_2\times_G E,\sim F_1\times_G E)$ is induced by the $G$-equivariant inclusion $U\hookrightarrow F_2$.
    Therefore there is an induced chain map between DG Morse complexes $\mathrm{C}_{\bullet}(M;(\Phi,\mathrm{C}_{\bullet}(U,\sim F_1)))\to \mathrm{C}_{\bullet}(M;(\Phi,\mathrm{C}_{\bullet}(F_2,\sim F_1)))$.
    The induced map in homology is an isomorphism by excision.
    Clivio shows in \cite{clivio2025goresky} that in this situation the map of $\mathrm{C}_{\bullet}(G)$-modules $\zeta_*\colon \mathrm{C}_{\bullet}(U,\sim F_1)\to  \mathrm{C}_{\bullet}(F_2,\sim F_1)$ can be inverted up to homotopy as a morphism of $A_{\infty}$-modules.
    We summarize Clivio's construction in Appendix \ref{sec_coherent_chain_homotopies_appenix} and show that it can be applied to our situation.
    Note that we need field coefficients here.
    By Example \ref{example_excision_path_module-stuff} we obtain a chain map $\rho_2 := \widetilde{\xi}\colon \mathrm{C}_{\bullet}(M;(\Phi,\mathrm{C}_{\bullet}(F_2,\sim F_1)))\to \mathrm{C}_{\bullet}(M;(\Phi,\mathrm{C}_{\bullet}(U,\sim F_1)))$  such that the diagram 
    $$
        \begin{tikzcd}
            \mathrm{H}_{\bullet}(M;(\Phi,\mathrm{C}_{\bullet}(F_2,\sim F_1))) \arrow[]{r}{t_2} \arrow[]{d}{\widetilde{\xi}} & \mathrm{H}_{\bullet}(F_2\times_G E,  \sim F_1\times_G E) \arrow[]{d}{\text{excision}} \\
            \mathrm{H}_{\bullet}(M;(\Phi,\mathrm{C}_{\bullet}(U,\sim F_1))) \arrow[]{r}{t_3} & \mathrm{H}_{\bullet}(U\times_G E,\sim F_1\times_G E) 
        \end{tikzcd}
    $$
    commutes.
    For the third map define
    \begin{eqnarray*}
          \rho_3\colon \mathrm{C}_{\bullet}(M;(\Phi,\mathrm{C}_{\bullet}(U,\sim F_1))) & \to & \mathrm{C}_{\bullet} (M;(\Phi,\mathrm{C}_{\bullet}(U))), \\ \alpha\otimes x  & \mapsto  & (\iota^*\tau \cap \alpha) \otimes x .  
    \end{eqnarray*} 
    By assumption, capping with $\iota^*\tau$ is $\mathrm{C}_{\bullet}(G)$-equivariant and therefore $\rho_3$ is a chain map.
    We want to show that the square 
    \begin{equation}\label{square_thom}
        \begin{tikzcd}
            \mathrm{H}_{\bullet}(M;(\Phi,\mathrm{C}_{\bullet}(U,\sim F_1))) \arrow[]{r}{t_3} \arrow[]{d}{\rho_3} & \mathrm{H}_{\bullet}(U\times_G E,\sim F_1\times_G E) \arrow[]{d}{\tau\cap} \\
            \mathrm{H}_{\bullet}(M;(\Phi,\mathrm{C}_{\bullet}(U))) \arrow[]{r}{t_4} & \mathrm{H}_{\bullet}(U\times_G E) 
        \end{tikzcd}
    \end{equation}
    commutes, where $t_3,t_4$ are the maps as in Theorem \ref{theorem_assoc_bundle}.
    Denote by $p\colon U\times E\to U\times_G E$ the canonical projection and note that we have $t_3(\alpha\otimes x) = p_*(\alpha\times m_x)$ for $\alpha\otimes x\in \mathrm{C}_{\bullet}(U,\sim F_1)$ and $x\in\mathrm{Crit}(f)$.
    Here, $m_x = (\phi_x)_*(s_x)$ with $\phi_x\colon \overline{W}_u(x)\to M$ the compatible section, see the proof of Theorem \ref{theorem_assoc_bundle}.
    We have
    $$
        \tau\cap (t_3(\alpha\otimes x)) = \tau\cap (p_*(\alpha\times m_x)) 
        = p_* \big( (p^* \tau)\cap (\alpha\times m_x)\big) .
        $$      
        On the other hand we have
    $$    t_4 (\rho_3(\alpha\otimes x)) = p_* ((\iota^*\tau \cap \alpha)  \times m_x) .    $$
    Now in cohomology we have $p^*\tau  = \iota^*\tau \times 1 \in \mathrm{H}^k(U\times E, \sim F_1\times E)$ and therefore, square \eqref{square_thom} commutes in homology.
    Taking the composition of all four maps yields a chain map $$\rho_4\circ \rho_3\circ \rho_2\circ \rho_1 \colon \mathrm{C}_{\bullet}(M;(\Phi,\mathrm{C}_{\bullet}(F_2)))\to \mathrm{C}_{\bullet}(M;(\Phi,\mathrm{C}_{\bullet}(F_1)))$$ and we have shown that it induces the Gysin map $j_!\colon \mathrm{H}_{\bullet}(F_2\times_G E)\to \mathrm{H}_{\bullet}(F_1\times_G E)$.    
    This completes the proof.
\end{proof}
We obtain an induced map of spectral sequences.
Using our knowledge about the first pages of the spectral sequence, see Section \ref{subsec_prelim_morse_DG}, we can understand this map explicitly on the $E^2$-page.
\begin{cor}\label{cor_spectral_sequence}
    In the situation of Theorem \ref{theorem_gysin_map_chain_map} let $E^r_{p,q}$ be the canonical spectral sequence associated to the DG Morse complex $\mathrm{C}_{\bullet}(M;(\Phi,\mathrm{C}_{\bullet}(F_2)))$ and let $\mathcal{E}^r_{p,q}$ be the canonical spectral sequence associated to the DG Morse complex $\mathrm{C}_{\bullet}(M;(\Phi,\mathrm{C}_{\bullet}(F_1)))$. 
    The chain map $\rho \colon \mathrm{C}_{\bullet}(M;(\Phi,\mathrm{C}_{\bullet}(F_2)))\to \mathrm{C}_{\bullet}(M;(\Phi,\mathrm{C}_{\bullet}(F_1)))$ induces a map of spectral sequences $ E_{p,q}^r \to \mathcal{E}_{p,q-k}^r$ which converges to the Gysin morphism $j_!\colon \mathrm{H}_{\bullet}(F_2\times_G E)\to \mathrm{H}_{\bullet}(F_1\times_G E)$.
    On the $E^2$-page the map of spectral sequences is given by
    $$    \mathrm{H}_p(\widetilde{C}_{\bullet};\mathrm{H}_q(F_2)) \to \mathrm{H}_p(\widetilde{C}_{\bullet}; \mathrm{H}_{q-k}(F_1))        $$
    which is induced by the Gysin map $i_!\colon \mathrm{H}_q(F_2)\to \mathrm{H}_{q-k}(F_1)$ in the fiber.
\end{cor}

It remains to show that there are interesting cases where a cocycle satisfying the condition of Theorem \ref{theorem_gysin_map_chain_map} exists.

\begin{prop}\label{prop_invariant_thom_cc_exists_for_Gamma}
    Let $G= \Gamma$ be a finite abelian group and let $q\colon E\to M$ be a $\Gamma$-principal bundle.
    Let $F_1\subseteq U\subseteq F_2$ be $\Gamma$-spaces as in Theorem \ref{theorem_gysin_map_chain_map} and take rational coefficients.
    Assume that $\Gamma$ acts in an orientation-preserving manner on $U$.
    Then there exists a cocycle $\tau\in \mathrm{C}^k(U\times_{\Gamma} E,\sim F_1\times_{\Gamma} E)$ representing the Thom class such that the chain map $\iota^*\tau \cap \colon \mathrm{C}_{\bullet}(U,\sim F_1)\to \mathrm{C}_{\bullet-k}(F_1)$ is $\mathrm{C}_{\bullet}(\Gamma)$-equivariant.
\end{prop}
\begin{proof}
    Recall that the associated bundle $U\times_{\Gamma}E$ is defined by quotienting out the action $$U\times E\times \Gamma \to U\times E , \quad    ((x,e),h) \mapsto (x\cdot h, h^{-1}\cdot  e)   . $$
    We define a map
    $$  \chi\colon   U\times_{\Gamma} E \times \Gamma \to U\times_{\Gamma} E, \quad ([x,e],g)\mapsto [x\cdot  g,e] .        $$
    Since $\Gamma$ is abelian this is well-defined and again a group action.
    Moreover, the embedding $\iota\colon U\hookrightarrow U\times_{\Gamma} E, x\mapsto [x,e_0]$ is an equivariant map.
    Let $\tau\in \mathrm{C}^k(U\times_{\Gamma} E, \sim F_1\times_{\Gamma} E;\mathbb{Q})$ be a cocycle representing the Thom class.
    Define 
    \begin{equation} \label{eq_averaged_cocycle} \tau_{\Gamma} = \frac{1}{|\Gamma|} \sum_{g\in \Gamma} (\chi_g)^* \tau \in \mathrm{C}^k(U\times_{\Gamma} E,\sim F_1\times_{\Gamma} E) .    \end{equation}
    One checks that $\tau_{\Gamma}$ is $\Gamma$-invariant in the sense that 
    $
           \chi_g^* \tau_{\Gamma} = \tau_{\Gamma} $ for all $g\in\Gamma$.
    Since $\Gamma$ acts in an orientation-preserving manner the cocycle $\tau_{\Gamma}$ represents the Thom class.
    We claim that $\iota^*\tau_{\Gamma}\in \mathrm{C}^k(U,\sim F_1)$ is such that capping with $\iota^*\tau$ is $\mathrm{C}_{\bullet}(\Gamma)$-equivariant.
    Recall that $\mathrm{C}_{\bullet}(\Gamma) = \mathrm{C}_0(\Gamma)$.
    Let $g\in \Gamma$ and consider the $0$-chain $g\in\mathrm{C}_0(\Gamma)$.
    Recall that the $\mathrm{C}_{\bullet}(\Gamma)$-action on $U$ is induced by the group action $\psi\colon U\times G\to U$ and denote the map $p\mapsto \psi(p,g)$ for fixed $g\in G$ by $\psi_g\colon U\to U$.
    For $\alpha\in \mathrm{C}_{\bullet}(U,\sim F_1)$ we have
    \begin{eqnarray*}
       &   &     \iota^* \tau_{\Gamma} \cap (\psi_* (\alpha \times g )) = \iota^*\tau_{\Gamma} \cap (\psi_g)_* (\alpha)  =   (\psi_g)_* \big(  (\psi_g)^*\iota^*\tau_{\Gamma}  \cap \alpha    \big)  \\ &  & = (\psi_g)_* \big( \iota^* (\chi_g^*\tau_{\Gamma}) \cap \alpha\big)      =  \psi_* \big( (\iota^* \tau_{\Gamma}\cap \alpha) \times g\big) .
    \end{eqnarray*}
    This shows that the chain map $(\iota^*\tau )\cap \cdot \colon \mathrm{C}_{\bullet}(U,\sim F_1)\to \mathrm{C}_{\bullet}(U)$ is $\mathrm{C}_{\bullet}(\Gamma)$-equivariant as claimed. 
\end{proof}
\begin{remark}
    It is desirable to find larger classes of examples where cocycles satisfying the condition of Theorem \ref{theorem_gysin_map_chain_map} exist.
    The averaging in equation \eqref{eq_averaged_cocycle} with respect to the counting measure on the finite group $\Gamma$ could be done more generally on de Rham-cochains by averaging with respect to the Haar measure on a compact abelian Lie group $G$.
    However, it requires more work to show that capping with $i^*\tau_{\Gamma}$ is $\mathrm{C}_{\bullet}(\Gamma)$-equivariant.
    Moreover, de Rham-cochains do not act directly on cubical chains but only on smooth ones, so some more care is required.
    Clivio \cite{clivio2025goresky} shows how one can use smoothening of chains in order to reconcile these two situations.
\end{remark}

We conclude by sketching how one can treat maps in the base manifold.
The main complication here arises from the challenge of pulling back transport functions along smooth maps $f\colon N\to M$.
In general we do not have a good notion of pullback of transport functions.
Following ideas similar to \cite[Section 9]{barraud2025morse} we shall now sketch how this problem can be circumvented by replacing a smooth map $f\colon N\to M$ with an embedding $f'\colon N\hookrightarrow M\times \mathbb{D}^k$ where $\mathbb{D}^k$ is a Euclidean disc.

Let $U\subseteq M$ be a codimension $0$ submanifold with $\partial U\subseteq M$.
Choose Morse data $(f_U,X_U)$ on the manifold with boundary $U$ by which we mean the following.
The function $f_U\colon U\to \mathbb{R}$ has critical points only in the interior of $U$.
The critical points are non-degenerate and $X_U$ is a pseudo-gradient which is inward-pointing on $\partial U$.
We extend this to Morse data $(f_M,X_M)$ on $M$.
For the function $f_U$ one defines moduli spaces as before and we note that we have $\overline{\mathcal{L}}^U(x,y) = \overline{\mathcal{L}}^M(x,y)$ where $\overline{\mathcal{L}}^U(x,y)$ refers to the moduli spaces defined by $f_U$ and $\overline{\mathcal{L}}^M(x,y)$ means the moduli spaces for $(f_M,X_M)$ as before.
We obtain a flow category $\mathcal{M}_{f_U}$
and we can again construct a space $K$ out of the moduli spaces as in Section \ref{subsec_morse_flow}.
One can show that the space $ K $ admits a topological embedding $\sigma\colon K\hookrightarrow U$ such that $\sigma(K)$ is a strong deformation retract of $U$.
We do note provide a proof but we note that this is done by retracting the flow lines which run inwards from the boundary of $U$ to their endpoint in the interior of $U$.
Let $G$ be a topological group and let $\Phi\colon \mathrm{mor}(\mathcal{M}_{f_M})\to G$ be a transport function.
One can define a map $\Phi'\colon \mathrm{mor}(\mathcal{M}_{f_U}) \to G$ by $$   \Phi'|_{\overline{\mathcal{L}}^U(x,y)} =  \Phi|_{\overline{\mathcal{L}}^M(x,y) }  $$
under the identification $\overline{\mathcal{L}}^M(x,y) = \overline{\mathcal{L}}^U(x,y)$.
If one introduces the obvious notion of a transport function for manifolds with boundary then $\Phi'$ is a transport function.
For a $\mathrm{C}_{\bullet}(G)$-module $\mathcal{F}_{\bullet}$ there is thus also a DG Morse complex and we can define a map $$i_*\colon \mathrm{C}_{\bullet}(U;(\Phi',\mathcal{F}_{\bullet}))\to  \mathrm{C}_{\bullet}(M; (\Phi,\mathcal{F}_{\bullet})) ,\quad   i_*(\alpha\otimes x) = \alpha\otimes x  .  $$
It is easy to see that this is a chain map.
If $N\subseteq M$ is a closed submanifold we can define morphisms through direct maps in the base as follows.
Let $U \subseteq M$ be a tubular neighborhood with boundary $\partial U \subseteq M$.
Choose a Morse function $f\colon N\to\mathbb{R}$ and extend this to a Morse function $F_U \colon U\to \mathbb{R}$ as in \cite[Section 9.2]{barraud2025morse}.
We have $\overline{\mathcal{L}}_{F_U}(x,y) = \overline{\mathcal{L}}_f(x,y)$ for $x,y\in\mathrm{Crit}(f)$.
If $\Phi'\colon \mathrm{mor}(\mathcal{M}_{F_U})\to G$ is a transport function in $U$ as above then we get an induced transport function $\Phi''\colon \mathrm{mor}(\mathcal{M}_f)\to G$ on $N$.
Again one gets an obvious map between the DG Morse complexes $\mathrm{C}_{\bullet}(N;(\Phi'',\mathcal{F}_{\bullet}))\to \mathrm{C}_{\bullet}(U;(\Phi',\mathcal{F}_{\bullet}))$.
Combining the two maps we obtain a map $\mathrm{C}_{\bullet}(N,(\Phi'',\mathcal{F}_{\bullet}))\to \mathrm{C}_{\bullet}(M;(\Phi,\mathcal{F}_{\bullet}))$.
Finally, if $f\colon N\to M$ is an arbitrary map, we can do the following.
Let $\chi\colon N\hookrightarrow \mathbb{D}^k$ be an embedding for some $k\in\mathbb{N}$.
Then the map $(f,\chi)\colon N\to M\times \mathbb{D}^k$ is an embedding as well.
We can thus run our sketched construction for the map $(f,\chi)$ and since $M\times \mathbb{D}^k\simeq M$ we expect this to yield a chain model for the map $f$ under some suitable isomorphisms.

Of course, this is only a rough sketch of the construction of direct maps in the base and many details are left open.
In particular, it would be interesting to treat the connection between transport functions and principal bundles on manifolds with boundary in style of Section \ref{sec_transport_function}.

\appendix

\section{Proof of Theorem \ref{theorem_isotopic_tr_induces_quasi-iso}}\label{sec_proof_isotopy_quasi}

In this section we prove the second part of Theorem \ref{theorem_isotopic_tr_induces_quasi-iso}.
\begin{claim}
     Let $(M,f,X,o,\{s_{x,y}\})$ be enhanced Morse data and let $G$ be a topological group.
    Let $\mathcal{F}_{\bullet}$ be a DG $\mathrm{C}_{\bullet}(G)$-module and let $\Phi,\Psi\colon \mathrm{mor}(\mathcal{M}_f)\to G$ be transport functions.
    Take chain complexes with $\mathbb{Z}_2$-coefficients.
    If $\Phi$ and $\Psi$ are isotopic then the DG Morse complexes $\mathrm{C}_{\bullet}(M;(\Phi,\mathcal{F}_{\bullet}))$ and $\mathrm{C}_{\bullet}(M;(\Psi,\mathcal{F}_{\bullet}))$ are quasi-isomorphic.
\end{claim}

Before we give the details of the proof we outline the general strategy.
We want to come up with a chain map $$\mathrm{C}_{\bullet} (M;(\Phi,\mathcal{F}_{\bullet}))  =  \mathcal{F}_{\bullet}\otimes \mathbb{Z}_2\langle \mathrm{Crit}(f) \rangle \, \xrightarrow[]{\hphantom{ebv}} \, \mathcal{F}_{\bullet}\otimes \mathbb{Z}_2\langle \mathrm{Crit}(f)\rangle  = \mathrm{C}_{\bullet}(M,(\Psi,\mathcal{F}_{\bullet})) .   $$
In \cite[Sections 2 and 4]{barraud2025morse} Barraud, Damian, Humilière and Oancea study certain maps between DG Morse complexes which are suitable for our situation.
Let $m_{x,y}^{\Phi}\in \mathrm{C}_{\bullet}(G)$ denote the twisting cocycle induced by the transport function $\Phi$ and let $m_{x,y}^{\Psi}\in \mathrm{C}_{\bullet}(G)$ denote the twisting cocycle induced by $\Psi$.
Assume that $n_{x,y}\in \mathrm{C}_{|x|-|y|}(G)$ is a collection of chains such that
    \begin{equation}\label{eq_chain_nab}
        \partial n_{x,y} = m^{\Phi}_{x,y}  + m^{\Psi}_{x,y} + \sum_{|y|<|c|<|x|} (n_{x,c}\cdot m_{c,y}^{\Phi} + m^{\Psi}_{x,c}\cdot n_{c,y} ).
    \end{equation}
        
Then the map $\Xi\colon \mathrm{C}_{\bullet}(M;(\Phi,\mathcal{F}_{\bullet})) \to \mathrm{C}_{\bullet}(M;(\Psi,\mathcal{F}_{\bullet}))$ given by
$$    \Xi(\alpha\otimes x) = \alpha\otimes y + \sum_{|y|<|x|} \alpha\cdot n_{x,y}\otimes y    $$
is a chain map and a quasi-isomorphism by \cite[Proposition 4.10]{barraud2025morse}.
In order to come up with the collection of chains $\{n_{x,y}\}$ we want to use the homotopies $\mathcal{H}_{x,y}\colon\overline{\mathcal{L}}(x,y)\times I\to G $ given by the isotopy between $\Phi$ and $\Psi$.
A first approach would have been to define the chains $n_{x,y}$ by pushing forward chains representing the fundamental class of $\overline{\mathcal{L}}(x,y)\times [0,1]$ through $\mathcal{H}_{x,y}$.
However, this does not lead to the right behaviour of the boundary of $n_{x,y}$.
Instead, we will enlarge the space $\overline{\mathcal{L}}(x,y)\times [0,1]$ in a specified way.
Heuristically speaking we attach simplexes to the boundary strata such that the dimension of the simplex increases with the codimension of the stratum.
See Figure \ref{fig_space_wab} for a sketch of the case $|x|-|y| = 2$.

Let $(M,f,X,o,\{s_{x,y}\})$ be enhanced Morse data and let $\mathcal{H}\colon \mathrm{mor}(\mathcal{M}_f)\times I \to G$ be an isotopy between the transport functions $\Phi$ and $\Psi$. 
We shall define spaces $W_{x,y}$ for critical points $x,y\in\mathrm{Crit}(f)$ with $|x| > |y|$ as well as maps $k^{\mathcal{H}}_{x,y}\colon W_{x,y}\to G$.
We want to have an inclusion $\overline{\mathcal{L}}(x,y)\times [0,1]\hookrightarrow W_{x,y}$ and we want $W_{x,y}$ to be topological manifold with boundary
$$    \partial W_{x,y} \cong \overline{\mathcal{L}}(x,y) \times \partial I \cup   \bigcup_{|y|<|z| < |x|} (W_{x,z}\times \overline{\mathcal{L}}(z,y) \cup \overline{\mathcal{L}}(x,z)\times W_{z,y})   .  $$
The restriction of $k^{\mathcal{H}}_{x,y}$ to $\overline{\mathcal{L}}(x,y)\times [0,1]$ should simply by $\mathcal{H}$.
The restriction of the map $k^{\mathcal{H}}_{x,y}$ to the boundary $\partial W_{x,y}$ should satisfy
\begin{eqnarray*}
    k^{\mathcal{H}}_{x,y}|_{\overline{\mathcal{L}}(x,y)\times \{0\}} = \Phi_{x,y} , \,\, & & \quad k^{\mathcal{H}}_{x,y}|_{\overline{\mathcal{L}}(x,y)\times \{1\}} = \Psi_{x,y}, \\
    k^{\mathcal{H}}_{x,y}|_{W_{x,z}\times \overline{\mathcal{L}}(z,y)} = \mu\circ (k^{\mathcal{H}}_{x,z}\times \Phi_{z,y}), \,\, &  & \quad k^{\mathcal{H}}_{x,y}|_{\overline{\mathcal{L}}(x,z)\times W_{z,y}} =\mu \circ ( \Psi_{x,z}\times k^{\mathcal{H}}_{z,y}) 
 \end{eqnarray*}
 where $\mu\colon G\times G\to G$ is the multiplication in $G$.
 The desired chain $n_{x,y}\in \mathrm{C}_{|x|-|y|}(G)$  as in equation \eqref{eq_chain_nab} is then obtained by pushing forward a certain representative of the fundamental class of $W_{x,y}$ via the map $k^{\mathcal{H}}_{x,y}$.

\medskip

We shall now construct the spaces $W_{x,y}$.
For $|x|-|y| = 1$ recall that $\mathcal{L}(x,y) = \overline{\mathcal{L}}(x,y)$ is discrete and we set $W_{x,y} = \overline{\mathcal{L}}(x,y)\times I$.
This is a disjoint union of copies of the interval and we therefore get $\partial W_{x,y} = \overline{\mathcal{L}}(x,y)\times \partial I$.
We set $k^{\mathcal{H}}_{x,y}\colon W_{x,y}\to G$ to be $k^{\mathcal{H}}_{x,y} = \mathcal{H}_{x,y}\colon \mathcal{L}(x,y)\times I\to G$.
Since $\mathcal{H}$ is an isotopy of transport functions we have $k^{\mathcal{H}}_{x,y}|_{\overline{\mathcal{L}}(x,y)\times \{0\}} = \Phi_{x,y}$ and $k^{\mathcal{H}}_{x,y}|_{\overline{\mathcal{L}}(x,y)\times \{1\}} = \Psi_{x,y}$ as desired.

Assume that $W_{x,y}$ has been defined for all pairs of critical points $x,y\in\mathrm{Crit}(f)$ with $ |x|-|y| \leq k$.
Let $a,b\in\mathrm{Crit}(f)$ be critical points with $|a| - |b| = k+1$.
Recall that we have
$$    \overline{\mathcal{L}}(a,b) = \mathcal{L}(a,b) \,\cup \,   \bigcup_{\substack{c_1,\ldots, c_{\ell} \in \mathrm{Crit}(f)\\ |b| < |c_{\ell}|<\ldots < |c_1|  <|a|}}   \mathcal{L}(a,c_1)\times \mathcal{L}(c_1,c_2)\times \ldots \times \mathcal{L}(c_{\ell},b) .   $$
We want to define a space $W_{a,b}\subseteq \overline{\mathcal{L}}(a,b)\times \mathbb{R}^{n}$ where $n = \mathrm{dim}(M)$.
Let $\Delta^n_{\mathrm{st}}\subseteq \mathbb{R}^{n}$ be the standard simplex
$$    \Delta^n_{\mathrm{st}} = \{ (t_1,t_2 ,\ldots ,t_n)\in \mathbb{R}^{n}  \,|\, 1\geq t_1 \geq \ldots \geq t_n \geq 0 \}.   $$
We want to introduce some bookkeeping for the faces of $\Delta^n_{\mathrm{st}}$.
Let $k\in \{1,\ldots, n-1\}$ and let $n\geq i_1 > i_2 > \ldots > i_{k+1} \geq 0$ be integers.
We define the $k$-simplex $\Delta^k_{i_1,\ldots, i_{k+1}}\subseteq \Delta^n$ by
$$   \Delta^k_{i_1,\ldots, i_{k+1}} = \{ (1,\ldots, 1, \tau_1,\ldots, \tau_1,\tau_2,\ldots, \tau_2,\ldots, \tau_k,\ldots ,\tau_k,0,\ldots, 0 ) \in \Delta^n_{\mathrm{st}} \}      $$
where the first $\tau_1$ appears at position $(n+1)-i_1$, the first $\tau_2$ at position $(n+1)-i_2$ and so on.
The first $0$ appears at position $(n+1)-i_{k+1}$.
There is an obvious identification $\Delta^k_{i_1,\ldots, i_{k+1}}\cong \Delta_{\mathrm{st}}^k$ given by
\begin{equation}\label{eq_identification_simplices}
        (1,\ldots, 1,   \tau_1,\ldots, \tau_1,\tau_2,\ldots, \tau_2,\ldots, \tau_k,\ldots ,\tau_k,0,\ldots, 0 ) \mapsto (\tau_1,\tau_2,\ldots, \tau_k) .     
\end{equation}
As an example if $n = 3$, we have
$   \Delta^3_{3,2,1,0 } = \Delta^3_{\mathrm{st}}  \quad \text{and}\quad \Delta^2_{3,2,0}= \{ (t,s,s)\in \Delta^3_{\mathrm{st}} \} \cong \Delta^2_{\mathrm{st}}  $.
Moreover, we have
$    \Delta^n_{n,0 } = \{(\tau,\tau,\ldots, \tau) \in \mathbb{R}^n \,|\, 1\geq \tau \geq 0 \} \cong [0,1] .      $
Now define $W_{a,b}\subseteq \overline{\mathcal{L}}(a,b)\times \Delta^n_{\mathrm{st}}$ by
$$     W_{a,b} = {\mathcal{L}}(a,b) \times \Delta^1_{|a|,|b|} \cup    \bigcup_{\substack{c_1,\ldots, c_{\ell} \in \mathrm{Crit}(f)\\ |b| < |c_{\ell}|<\ldots < |c_1|  <|a|}}   \mathcal{L}(a,c_1)\times \mathcal{L}(c_1,c_2)\times \ldots \times \mathcal{L}(c_{\ell},b) \times   \Delta^{l+1}_{|a|,|c_1|,\ldots, |c_{\ell}|,|b|} . $$
We equip $W_{a,b}$ with the subspace topology of $\overline{\mathcal{L}}(a,b)\times \Delta^n_{\mathrm{st}}$.
For a sketch of the space $W_{a,b}$ in case that $|a|-|b| = 2$ we refer to Figure \ref{fig_space_wab}.
\begin{figure}[t]
\centering
\includegraphics[scale=0.6]{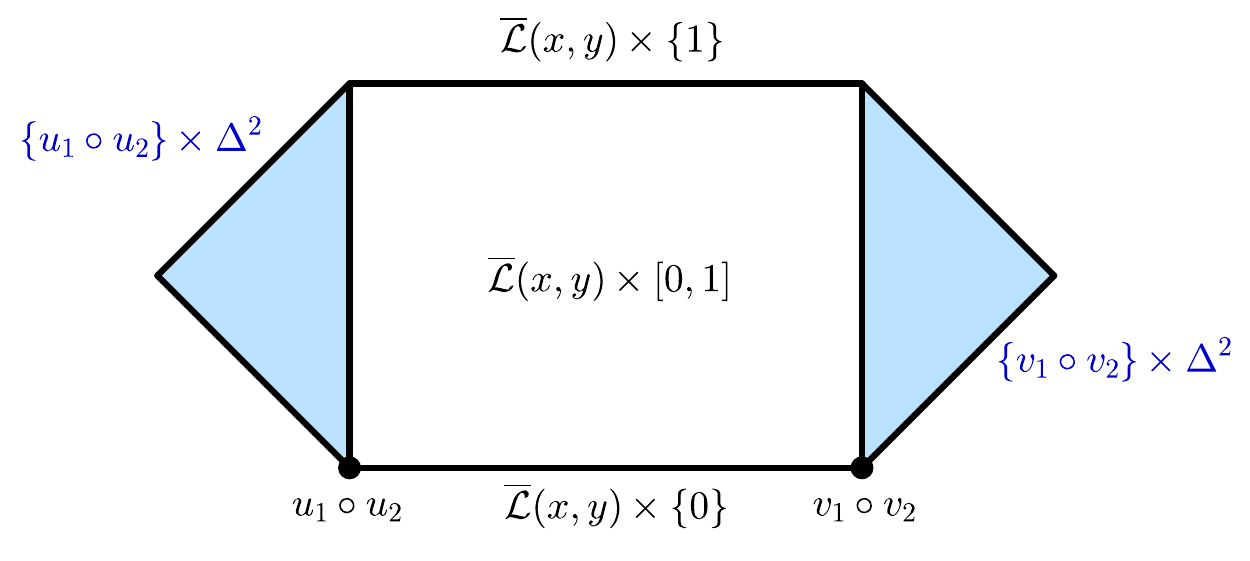}
\caption{The space $W_{x,y}$ for $x,y\in\mathrm{Crit}(f)$ with $|x|-|y| = 2$ and with $\overline{\mathcal{L}}(x,y)\cong [0,1]$. The $2$-simplexes are attached to $\overline{\mathcal{L}}(x,y)\times [0,1]$ such that one of its boundary $1$-simplexes glues to $\{u_1\circ u_2\}\times [0,1]$, resp. to $\{v_1\circ v_2\}\times [0,1]$.}
\label{fig_space_wab}
\end{figure}
We define $k^{\mathcal{H}}_{a,b}\colon W_{a,b}\to G$ by 
\begin{eqnarray}\label{eq_k_ab_h}
            \mathcal{L}(a,c_1)\times \mathcal{L}(c_1,c_2)\times \ldots \times \mathcal{L}(c_{\ell},b) \times   \Delta^{l+1}_{|a|,|c_1|,\ldots, |c_{\ell}|,|b|} & \to &  G, \\
             ((u_1,\ldots, u_{\ell+1}), (\tau_1,\ldots, \tau_{\ell+1}) )& \mapsto & \mathcal{H}_{a,c_1}(u_1,\tau_1)\cdot \ldots \cdot \mathcal{H}_{c_{\ell},b}(u_{\ell+1},\tau_{\ell+1})  \nonumber .
\end{eqnarray}
In particular for $\ell = 0$ we therefore have $k^{\mathcal{H}}_{a,b}|_{\mathcal{L}(a,b)\times \Delta^1_{|a|,|b|}} = \mathcal{H}_{a,b}\colon \mathcal{L}(a,b)\times [0,1]\to G$ under the identification $\Delta^1_{|a|,|b|}\cong \Delta^1_{\mathrm{st}} = [0,1]$.
\begin{lemma}
The map $k^{\mathcal{H}}_{a,b}\colon W_{a,b}\to G$ is continuous.
\end{lemma}
\begin{proof}
Let $u_n\in \mathcal{L}(a,b)$ be a sequence with $u_n\to u_* \in \overline{\mathcal{L}}(a,b)$ a broken flow line with maximal decomposition $u_* = v_1\circ \ldots v_{\ell+1}$ with $v_1\in \mathcal{L}(a,c_1),v_2\in \mathcal{L}(c_1,c_2),\ldots, v_{\ell}\in \mathcal{L}(c_{\ell-1},c_{\ell}) , v_{\ell+1}\in \mathcal{L}(c_{\ell},b)$.
Furthermore, let $(\tau_n)\in \Delta^1_{|a|,|b|}$ be a converging sequence with $$\tau_n = (1,\ldots, 1, t_n,\ldots, t_n,0,\ldots, 0)\to \tau_* = (1,\ldots, 1 , t_*,\ldots ,t_*,0,\ldots ,0)\in  \Delta^1_{|a|,|b|}\subseteq \Delta^{\ell+1}_{|a|,|c_1|,\ldots, |b|} .  $$ 
Then 
$$      k^{\mathcal{H}}_{a,b}(u_n,\tau_n) = \mathcal{H}_{a,b}(u_n,t_n) \to \mathcal{H}_{a,b}(u_*,t_*) = \mathcal{H}_{a,c_1}(v_1,t_*)\cdot \ldots \cdot \mathcal{H}_{c_{\ell},b}(v_{\ell+1},t_*) = k^{\mathcal{H}}_{a,b}(u_*,\tau_*)         $$
since $\mathcal{H}$ is an isotopy of transport functions.
The other cases one has to consider are similar.
\end{proof}
Next, we want to show that the space $W_{a,b}$ is indeed a manifold.

\begin{lemma}\label{lemma_w_ab_is_mfld}
    The space $W_{a,b}$ is a topological manifold with boundary.
    There are embeddings $\iota_{a,c;b}\colon W_{a,c}\times \overline{\mathcal{L}}(c,b) \hookrightarrow W_{a,b}$ and $j_{a;c,b}\colon \overline{\mathcal{L}}(a,c)\times W_{c,b}\hookrightarrow W_{a,b}$ such that the boundary of $W_{a,b}$ satisfies
    $$    \partial W_{a,b} = \overline{\mathcal{L}}(a,b) \times \partial I \cup   \bigcup_{|b|<|z| < |a|} \iota_{a,z;b}(W_{a,z}\times \overline{\mathcal{L}}(z,b)) \cup  j_{a;z,b}(\overline{\mathcal{L}}(a,z)\times W_{z,b})   .  $$
\end{lemma}

\begin{proof}
We need to show that for each point $(u,\tau)\in W_{a,b}$ there is an open neighborhood with is either homeomorphic to $\mathbb{R}^{|a|-|b|}$ or to the half-space $\mathbb{R}^{|a|-|b|-1}\times [0,\infty)$.
First, for $(u_*,\tau)\in \mathcal{L}(a,b)\times (0,1)$ it is clear that there is an open neighborhood which is homeomorphic to $\mathbb{R}^{|a|-|b|}$.
Furthermore, the points $(u_*,\tau)\in \mathcal{L}(a,b)\times \{0,1\}$ are boundary points.

Next, we consider the case $u\in \partial\overline{\mathcal{L}}(a,b)$ and $\tau = (1,\ldots, 1 ,t_0,\ldots, t_0,0,\ldots, 0) \in \mathring{\Delta}^1_{|a|,|b|}$.
We want to show that there is an open neighborhood inside $W_{a,b}$ which is homeomorphic to $\mathbb{R}^n$.
Let $u\in \mathcal{L}(a,c_1)\times \ldots \times \mathcal{L}(c_{\ell},b)$ be the maximal decomposition.
The point $u_*\in \overline{\mathcal{L}}(a,b)$ has a neighborhood of the form $ \widetilde{V}^1_{a,b}\times [0,\delta)^{\ell} $ inside $\overline{\mathcal{L}}(a,b)$ with $\widetilde{V}^1_{a,b}\subseteq \mathbb{R}^{|a|-|b|-({\ell+1})}$ open since $\overline{\mathcal{L}}(a,b)$ is a manifold with corners.
There is hence an open neighborhood $V^1_{a,b}$ of $(u_*,\tau)\in \overline{\mathcal{L}}(a,b)\times \Delta^1_{|a|,|b|}$ inside $\overline{\mathcal{L}}(a,b)\times \Delta^1_{|a|,|b|}$ of the form
$$    V^1_{a,b} \cong  \widetilde{V}^1_{a,b}\times [0,\delta)^{\ell} \times (t_0-\delta,t_0 + \delta)        $$
Define a map 
\begin{eqnarray*}
    \psi^1\colon V^1_{a,b} & \to &   \widetilde{V}_{a,b}\times [0,\delta)^{\ell}\times (-\delta,\delta) \\
    (p,(s_1,\ldots, s_{\ell}) , t)  & \mapsto & (p, (s_1,\ldots, s_{\ell}) , t-t_0) .
\end{eqnarray*}
Note that the points
$$   (x, (s_1,\ldots, s_{\ell+1})) \in \widetilde{V}^1_{a,b}\times [0,\delta)^{\ell}      $$
lie in $\overline{\mathcal{L}}(a,c_{j_1})\times \ldots \times \overline{\mathcal{L}}(c_{j_k},b)$ for $1 \leq j_1 < j_2 < \ldots < j_k\leq \ell$ if and only if $s_{j_1} = \ldots = s_{j_k} =0$.
Hence, inside $\overline{\mathcal{L}}(a,c_{j_1})\times \ldots \times \overline{\mathcal{L}}(c_{j_k},b)$ there is an open neighborhood of $u_*$ of the form
$$   V^{k+1}_{a,c_{j_1},\ldots, c_{j_k},b} =  \widetilde{V}_{a,b} \times [0,\delta)^{j_1-1}\times \{0\} \times [0,\delta)^{j_2-j_1 -1}\times \{0\}\times \ldots \{0\}\times [0,\delta)^{\ell-j_k -1} \cong \widetilde{V}_{a,b}\times [0,\delta)^{\ell-k} .      $$
Now, consider $\tau \in \Delta^1_{|a|,|b|}\subseteq \Delta^{k+1}_{|a|,|c_{j_1}|,\ldots, |c_{j_k}|,|b|}$.
Since the $(k+1)$-simplex $\Delta^{k+1}_{|a|,|c_{j_1}|,\ldots, |c_{j_k}|,|b|}$ is a manifold with corners there is a neighborhood $U_{a,c_{j_1},\ldots, c_{j_k},b}$ around $\tau$ in $\Delta^{k+1}_{|a|,|c_{j_1}|,\ldots,|b|}$ of the form $(t_0-\delta,t_0+\delta)\times [0,\delta)^{k}$.
This neighborhood can be chosen such that the following holds.
If $m\in \{1,\ldots, \ell\}$, then $U_{a,c_{j_1},\ldots, c_{j_k},b}\cap \Delta^{k+1}_{|a|,|c_{j_1}|,\ldots, c_{j_{m-1}},c_{j_{m+1}},\ldots, |c_{\ell}|,b}$ corresponds to the set
$$    \{ (s, \rho_1,\ldots, \rho_k)\in (t-\delta,t+\delta)\times [0,\delta)^k \,|\,  \rho_m = 0\} .       $$
We introduce the following spaces.
Write $\mathbb{R}^N_+ = [0,\infty)^N$ for $N\in\mathbb{N}$.
Next, let $1 \leq i_1 < \ldots < i_k \leq N$ and define 
$$    \mathbb{R}^N_{i_1,\ldots ,i_k} = \{ (x_1,\ldots, x_N)\in \mathbb{R}^N \,|\,  x_j \geq 0 \text{ for }j\not\in \{i_1,\ldots, i_k\} , \,\, x_j \leq 0\text{ for }j\in \{i_1,\ldots, i_k\}\} .    $$
Define a map 
\begin{eqnarray*}
    \psi^{k+1}_{j_1,\ldots, j_k} \colon V^{k+1}_{a,c_{j_1},\ldots, c_{j_k},b}\times U_{a,c_{j_1},\ldots, c_{j_k},b} & \to & \widetilde{V}_{a,b} \times   \mathbb{R}^{\ell}_{j_1,\ldots, j_k}  \times (-\delta,\delta) \\ 
    (p,(s_1,\ldots, s_{\ell-k}), t, (u_1,\ldots, u_k)) & \mapsto&  (p, s_1,\ldots, s_{j_1-1}, -u_1, s_{j_1},\ldots, s_{j_2-2}, -u_2, \ldots, \\ &  & \hphantom{eoergaerggij } s_{j_k-k+1},\ldots, s_{\ell-k}, t-t_0)) . 
\end{eqnarray*}
The maps $\psi^i_{j_1,\ldots, j_i}$ glue together to define a local chart around $(u_*,\tau)$.
More precisely, let 
$$   V =  V^1_{a,b} \cup_{k=1}^{\ell} \bigcup_{1\leq j_1< \ldots < j_k\leq \ell} V^{k+1}_{a,c_{j_1},\ldots, c_{j_k},b} \times U_{a,c_{j_1},\ldots, c_{j_k},b} \subseteq W_{a,b}      $$
and define $\psi\colon V\to \mathbb{R}^{|a|-|b|}$ by defining $\psi|_{V^{k+1}_{a,c_{j_1},\ldots, c_{j_k},b}\times U_{a,c_{j_1},\ldots, c_{j_k},b}}$ to be $\psi_{j_1,\ldots, j_k}^{k+1}$.
By potentially shrinking $V$ we obtain a homeomorphism $\psi\colon V\xrightarrow[]{\cong} \widetilde{V}_{a,b}\times (-\delta,\delta)^{\ell+1}\subseteq \mathbb{R}^{|a|-|b|}$ with $V$ open.

More generally, assume that we are given a point $u_*\in \mathcal{L}(a,c_1)\times \ldots \times \mathcal{L}(c_{\ell},b)$ as well as $\tau\in \Delta^{\ell+1}_{|a|,|c_1|,\ldots, |c_{\ell}|,|b|}$.
If $\tau$ is in the interior of the simplex then we obtain a Euclidean neighborhood around $(u_*,\tau)$.
If $\tau\in \partial\Delta^{\ell+1}_{|a|,|c_1|,\ldots, |c_{\ell}|,|b|}$ we can proceed similarly to before as long as
$  \tau\in \Delta^{\ell}_{|a|,|c_1|,\ldots, |\check{c_m}|,\ldots, |b|}     $ for some $m\in \{1,\ldots, \ell\}$.
In this case there are other strata such that $(u_*,\tau)$ lies in the closure of these strata and we can proceed as above.

In case that $\tau\in \Delta^{\ell}_{|c_1|,|c_2|,\ldots, |c_{\ell}|,|b|} \subseteq \Delta^{\ell+1}_{|a|,|c_1|,\ldots ,|b|}$ we obtain a boundary point.
In fact, for $\tau\in \mathring{\Delta}^{\ell}_{|c_1|,|c_2|,\ldots, |c_{\ell}|,|b|} \subseteq \Delta^{\ell+1}_{|a|,|c_1|,\ldots ,|b|}$ there is an open neighborhood $U\subseteq \Delta^{\ell+1}_{|a|,|c_1|,\ldots, |b|}$ of the form $U\cong (-\delta,\delta)^{\ell}\times [0,\delta)$ such that $U\cap \Delta^{\ell}_{|c_1|,\ldots,|b|}$ corresponds to $ (-\delta,\delta)^{\ell}\times \{0\}$.
As $(u_*,\tau)$ does not lie in the closure of any other stratum, we obtain a boundary point.
An analogous argument holds in the case that $\tau\in \Delta^{\ell}_{|a|,|c_1|,\ldots,|c_{\ell}|}$.

Overall we see that $W_{a,b}$ is a manifold with boundary and the boundary is the subspace
\begin{eqnarray}\nonumber
     \partial W_{a,b} &=&  {\mathcal{L}}(a,b) \times \partial \Delta^1_{|a|,|b|} \cup    \bigcup_{\substack{c_1,\ldots, c_{\ell} \in \mathrm{Crit}(f)\\ |b| < |c_{\ell}|<\ldots < |c_1|  <|a|}}   \mathcal{L}(a,c_1)\times \mathcal{L}(c_1,c_2)\times \ldots \times   \\  &   &   \hphantom{aretfrggbgfehaehbdfbdfbhthh}   \mathcal{L}(c_{\ell},b)  \times(\Delta^{\ell}_{|c_1|,\ldots,|c_{\ell}|,|b|} \cup \Delta^{\ell}_{|a|,|c_1|,\ldots,|c_{\ell}|}) . \label{eq_blablup} 
\end{eqnarray} 

Let $c\in \mathrm{Crit}(f)$ such that $|a|> |c| > |b|$ and let $\rho\colon \overline{\mathcal{L}}(a,c)\times \overline{\mathcal{L}}(b,c)\hookrightarrow \overline{\mathcal{L}}(a,b)$ be the inclusion.
We claim that the map $ \iota_{a,c;b}\colon W_{a,c}\times \overline{\mathcal{L}}(c,b)\to W_{a,b}$ given by
$$     \iota_{a,c;b}( (u,\tau) , v) = (\rho(u,v),\tau) \quad \text{for}\,\,\, (u,\tau)\in W_{a,c},\,\, v\in \overline{\mathcal{L}}(c,b)      $$
has image in $\partial W_{a,b}$.
We can write 
\begin{eqnarray*}
    W_{a,c}\times \mathcal{L}(c,b ) &= &  \mathcal{L}(a,c)\times \Delta^1_{|a|,|c|} \times \mathcal{L}(c,b)  \cup \bigcup_{|c| < |d_k| < \ldots < |d_1|< |a|} \mathcal{L}(a,d_1)\times \ldots   \\  &  & \hphantom{eoarigargaergjaeaergaergearoigjaeorigjae}    \times \mathcal{L}(d_k,c)\times \Delta^{k+1}_{|a|,\ldots,|c|}\times \mathcal{L}(c,b) .
\end{eqnarray*}
By \eqref{eq_blablup} we see that the image of $W_{a,c}\times \mathcal{L}(c,b)$ under $\iota_{a,c;b}$ lands in $\partial W_{a,b}$.
Similarly, one can check for the strata of $\partial\mathcal{L}(c,b)$ that the image of $W_{a,c}\times \partial\overline{\mathcal{L}}(c,b)$ lies in $\partial W_{a,b}$.
Analogously one defines an embedding $j_{a;c,b}\colon \overline{\mathcal{L}}(a,c)\times W_{c,b}\to W_{a,b}$ with image in $\partial W_{a,b}$.
Then one checks that the boundary $\partial W_{a,b}$ is given by the union
$$    \overline{\mathcal{L}}(a,b) \times \partial I \cup   \bigcup_{|b|<|z| < |a|} \iota_{a,z;b}(W_{a,z}\times \overline{\mathcal{L}}(z,b)) \cup  j_{a;z,b}(\overline{\mathcal{L}}(a,z)\times W_{z,b})  $$
as claimed.
This completes the proof.
\end{proof}

We introduce the following notation.
For the $1$-simplex $\Delta^1_{j,k}$ define $1_{j,k}\in \Delta^1_{j,k}$ to be the element corresponding to $1\in [0,1]$ under identification \eqref{eq_identification_simplices}.
Further define $0_{j,k}$ to be the element corresponding to $0$.
Note that for $i\in \{k+1,\ldots, j-1\}$ we have $1_{i,k} = 0_{k,j}$ as points in the surrounding space $\Delta^n_{\mathrm{st}}$.
For the rest of this section we take chain complexes with $\mathbb{Z}_2$-coefficients.
\begin{lemma}
    There is a chain $w_{a,b}\in \mathrm{C}_{|a|-|b|}(W_{a,b})$ which is a cycle relative to the boundary such that
    \begin{equation}\label{eq_boundary_wab_chain}
             \partial w_{a,b} =  s_{a,b}\times ( 1_{|a|,|b|} + 0_{|a|,|b|})  + \sum_{|b|<|c|<|a|}  (\iota_{a,c;b})_* (w_{a,c}\times s_{c,b} )  + (j_{a;c,b})_*(s_{a,c} \times w_{c,b}) . 
    \end{equation}
\end{lemma}
\begin{proof}
    We show the claim by induction over the index difference.
    First consider critical points $a,b\in\mathrm{Crit}(f)$ with $|a|-|b| = 1$ and recall that $W_{a,b} = \mathcal{L}(a,b)\times \Delta^1_{|a|,|b|}$.
    Let $\delta^1_{|a|,|b|}\colon [0,1]\to \Delta^1_{|a|,|b|}$ be a linear parametrization of the $1$-simplex $\Delta^1_{|a|,|b|}$.
    We set $w_{a,b} = s_{a,b}\times \delta^1_{|a|,|b|}\in \mathrm{C}_1(W_{a,b})$ which satisfies $\partial w_{a,b} = s_{a,b}\times (1_{|a|,|b|} + 0_{|a|,|b|})$.
    
    Now, assume that the chains $w_{a,c}\in \mathrm{C}_{\bullet}(W_{a,c})$ have been defined for all critical point of index difference less than or equal to $ k$.
    Let $a,b\in\mathrm{Crit}(f)$ with $|a|-|b| = k+1$.
    We consider the chain $v_{a,b}\in \mathrm{C}_{k}(\partial W_{a,b})$ defined by
    $$
        v_{a,b} =  s_{a,b}\times ( 1_{|a|,|b|} + 0_{|a|,|b|})  + \sum_{|b|<|c|<|a|}  (\iota_{a,c;b})_* (w_{a,c}\times s_{c,b} )  + (j_{a;c,b})_*(s_{a,c} \times w_{c,b}) . $$
        We claim that $v_{a,b}$ is a cycle.
        We compute
        \footnotesize
        \begin{eqnarray*}
            \partial v_{a,b} &=& \sum_{|b|<|z|<|a|} s_{a,z}\times s_{z,b}\times (1_{|a|,|b|} +0_{|a|,|b|}) + \sum_{|b|<|c|<|a|} (\iota_{a,c;b})_* \Big[ s_{a,c}\times (1_{|a|,|c|} + 0_{|a|,|c|} )\times s_{c,b} + \\  &   & \sum_{|c|<|d|<|a|} \big( (\iota_{a,d;c})_*(w_{a,d}\times s_{d,c})\times s_{c,b} + (j_{a;c,d})_*(s_{a,d}\times w_{d,c})\times s_{c,b} \big) + \sum_{|b|<|d|<|c|} w_{a,c}\times s_{c,d}\times s_{d,b}   \Big] + \\
            &   & \sum_{|b|<|c|<|a|}(j_{a,c;b})_* \Big[  \sum_{|c|<|d|<|a|}s_{a,d}\times s_{d,c}\times w_{c,b} \, + \, s_{a,c}\times s_{c,b}\times (1_{|c|,|b|} + 0_{|c|,|b|}) +  \\ &   &   \sum_{|b|<|d|<|c|} \big(     s_{a,c}\times (\iota_{c,d;b})_*(w_{c,d}\times s_{d,b}) + s_{a,c}\times (j_{c;d,b})_*(s_{c,d}\times w_{d,b})     \big)    \Big] .
        \end{eqnarray*}
        \normalsize
        As remarked in the paragraph before the lemma we have $1_{|a|,|b|} = 1_{|c|,|b|}$ as well as $0_{|a|,|b|} = 0_{|a|,|c|}$ and $0_{|c|,|b|} = 1_{|a|,|c|}$.
        Furthermore, we have that the maps $$\iota_{a,c;b}\circ (j_{a;c,d}\times \mathrm{id})  \colon \overline{\mathcal{L}}(a,c)\times W_{c,d}\times \overline{\mathcal{L}}(d,b) \to W_{a,b}  $$
        and
        $$   j_{a,c;b}\circ (\mathrm{id}\times \iota_{c,d;b})      \colon \overline{\mathcal{L}}(a,c)\times W_{c,d}\times \overline{\mathcal{L}}(d,b) \to W_{a,b}    $$
        agree for critical points $a,c,d,b$ with $|a|> |c| > |d| > |b|$.
       Using these identities one verifies that indeed $\partial v_{a,b} = 0$.
        By Lemma \ref{lemma_w_ab_is_mfld} and the induction hypothesis we conclude that $v_{a,b}$ represents the fundamental class of the boundary $\partial W_{a,b}$.
        Let $w_{a,b}'\in \mathrm{C}_{|a|-|b|}(W_{a,b})$ be a cycle relative to the boundary representing the fundamental class of $W_{a,b}$.
        Consequently, $\partial w_{a,b}'\in \mathrm{C}_{|a|-|b|-1}(\partial W_{a,b})$ represents the fundamental class of $\partial W_{a,b}$.
        Therefore, $\partial w_{a,b}'$ is homologous to $v_{a,b}$ ,i.e. there is a $p_{a,b}\in \mathrm{C}_{|a|-|b|}(\partial W_{a,b})$ such that $\partial w_{a,b}' + v_{a,b} = \partial p_{a,b}$.
        We define $w_{a,b} = w_{a,b}' + p_{a,b}\in \mathrm{C}_{|a|-|b|}(W_{a,b})$.
        Since $p_{a,b}$ is supported in $\partial W_{a,b}$, the chain $w_{a,b}$ still represents the fundamental class of $W_{a,b}$ and it satisfies \eqref{eq_boundary_wab_chain} by construction.
\end{proof}

Recall from \eqref{eq_k_ab_h} the map $k_{a,b}^{\mathcal{H}}\colon W_{a,b}\to G$.
For $x,y\in \mathrm{Crit}(f)$ with $|x|> |y|$ we define $n_{x,y} = (k^{\mathcal{H}}_{x,y})_*(w_{x,y}) \in \mathrm{C}_{|x|-|y|}(G)$.
Furthermore, we define $n_{x,x} = e\in \mathrm{C}_0(G)$ to be the neutral element and $n_{x,y} = 0$ for $|x| = |y|$ with $x\neq y$.
\begin{lemma}
    The boundary of $n_{a,b}$ satisfies
    \begin{eqnarray*}
        \partial n_{a,b} &=& m^{\Phi}_{a,b}  + m^{\Psi}_{a,b} + \sum_{|b|<|c|<|a|} (n_{a,c}\cdot m_{a,b}^{\Phi} + m^{\Psi}_{a,c}\cdot n_{c,b} ) \\ 
        &=&  \sum_{|b|\leq |c|\leq |a|} (n_{a,c}\cdot m^{\Phi}_{c,b} + m^{\Psi}_{a,c}\cdot n_{c,b}) .
    \end{eqnarray*}
\end{lemma}
\begin{proof}
    This follows from equation \eqref{eq_boundary_wab_chain} once one has verified that for $((u,\tau),v)\in W_{a,c}\times \overline{\mathcal{L}}(c,b)$ we have 
    $$    k^{\mathcal{H}}_{a,b}\circ \iota_{a,c;b}((u,\tau),v) =  k^{\mathcal{H}}_{a,c}(u,\tau) \cdot \Phi(v)    $$
    and for $(u,(v,\tau))\in \overline{\mathcal{L}}(a,c)\times W_{c,b}$ it holds that
    $$   k^{\mathcal{H}}_{a,b}\circ j_{a;c,b}( u,(v,\tau)) = \Psi(u) \cdot \mathcal{H}_{c,b}(v,\tau) .     $$
\end{proof}

By \cite[Proposition 2.3]{barraud2025morse} it follows that the map $\Xi\colon \mathrm{C}_{\bullet}(M;(\Phi,\mathcal{F}_{\bullet})) \to \mathrm{C}_{\bullet}(M;(\Psi,\mathcal{F}_{\bullet}))$ defined by $$\Xi(\alpha\otimes x) = \sum_{|y|\leq |x|} \alpha\cdot n_{x,y}\otimes y   $$
is a chain map.
Using \cite[Proposition 4.10]{barraud2025morse} and the fact that $[n_{x,x}] = [e]\in \mathrm{H}_0(G)$ and $[n_{x,y}] = 0\in \mathrm{H}_0(G)  $ for $|x| = |y|$ with $x\neq y$, we conclude that $\Psi$ is a quasi-isomorphism.
This completes the proof of Theorem \ref{theorem_isotopic_tr_induces_quasi-iso}.

\section{Coherent chain homotopies} \label{sec_coherent_chain_homotopies_appenix}

In this section we define path modules and morphisms of path modules following the exposition in \cite[Section 5]{clivio2025goresky}.
This notion is used in the proof of Theorem \ref{theorem_gysin_map_chain_map} and we explain the specific construction used in that proof in Example \ref{example_excision_path_module-stuff}.

\begin{definition}[Definition 5.5 in \cite{clivio2025goresky}]
    Let $(\mathbf{R}_{\bullet},d_{\mathbf{R}},\mu)$ be a differential graded algebra and $(\mathbf{E}_{\bullet}, d_{\mathbf{E}},\mu)$ a left $\mathbf{R}_{\bullet}$-module such that $\mathbf{R}_{\bullet}\subseteq \mathbf{E}_{\bullet}$.
    A \emph{path module} over $(\mathbf{R}_{\bullet},\mathbf{E}_{\bullet})$ is given by a chain complex $(\mathbf{X}_{\bullet},d_{\mathbf{X}})$, a subcomplex $\mathcal{F}_{\bullet}\subseteq \mathbf{X}_{\bullet}$ and a collection of maps 
    $$    m_k \colon \mathcal{F}_{\bullet}\otimes \mathbf{R}_{\bullet}^{\otimes k-2} \otimes \mathbf{E}_{\bullet} \to \mathbf{X}_{\bullet +k-2}    $$
    such that $m_k$ restricts to a map
    $$  m^{\mathcal{F}}_k  \colon \mathcal{F}_{\bullet}\otimes \mathbf{R}^{\otimes k-1} \to \mathcal{F}_{\bullet + k-2}   $$
    Furthermore, we demand that $m_2$ is a chain map and that for $N\geq 3$ it holds that
    \begin{eqnarray*}
        &  & (-1)^{N+1} m_N \circ d + d \circ m_N = \sum_{r = 0}^{N-3} (-1)^r m_{N-1} \circ (\mathrm{id}^{\otimes r+1}\otimes \mu\otimes    \mathrm{id}^{\otimes N-r-3} )    \\ &  &  \hphantom{aerogijaaergaergaeeoirjg}  - \sum_{s = 2}^{N-1} (-1)^{s(N-s)} m_{N-s+1} \circ (m_s^{\mathcal{F}} \otimes \mathrm{id}^{\otimes N-s}) .
    \end{eqnarray*}
    A path-module is called \emph{strict} if $m_k = 0$ for $k\geq 3$.
\end{definition}


\begin{definition}[Definition 5.7 in \cite{clivio2025goresky}] \label{def_morphism_of_path_modules}
    Let $(\mathbf{X}_{\bullet}^i, \mathcal{F}_{\bullet}^i,m_k^i)$ be path-modules over $(\mathbf{R}_{\bullet},\mathbf{E}_{\bullet})$ for $i = 1,2$.
    A \emph{morphism of path-modules} $\eta\colon \mathbf{X}_{\bullet}^1\to \mathbf{X}_{\bullet+m}^2$ of degree $m\in\mathbb{Z}$ is given by a chain map $\eta_1\colon \mathbf{X}_{\bullet}^1\to \mathbf{X}_{\bullet+ m}^2$ and maps $\eta_k\colon \mathcal{F}_{\bullet}^1 \otimes \mathbf{R}_{\bullet}^{k-2}\otimes \mathbf{E}_{\bullet} \to \mathbf{X}_{\bullet  + k-1+m}^2$ such that for $N\geq 1$ the $A_{\infty}$-equation
    \begin{eqnarray*}
       &  &  \eta_{N+1}\circ d + (-1)^{N+1+m} d\circ \eta_{N+1} = \sum_{s=1}^N (-1)^{s(N-s)} \eta_{N-s+1}\circ (m_{s+1}^1 \otimes \mathrm{id}^{N-s}) + \\
       &  & (-1)^N \sum_{k=1}^N (-1)^{k(N-k)} m_{k+1}^2 \circ (\eta_{N-k}\otimes \mathrm{id}^{\otimes k}) + \sum_{r=1}^{N-1} (-1)^{N+1+r} \eta_N\circ (\mathrm{id}^{\otimes r}\otimes \mu \otimes  \mathrm{id}^{\otimes N-1-r})        
    \end{eqnarray*}
    holds and such that $\eta_k$ restricts to a map $\eta_k^{\mathcal{F}} \colon \mathcal{F}_{\bullet}^1\otimes \mathbf{R}_{\bullet}^{k-1}\to \mathcal{F}^2_{\bullet + k-1+m}$.
\end{definition}
\begin{remark}\label{remark_a_infinity}
    Let $(\mathbf{X}_{\bullet}^i, \mathcal{F}_{\bullet}^i,m_k^i)$ be path-modules over $(\mathbf{R}_{\bullet},\mathbf{E}_{\bullet})$ for $i = 1,2$.
    One checks that the restrictions $m_k^{\mathcal{F}}\colon \mathcal{F}_{\bullet}\otimes \mathbf{R}^{k-1}\to \mathcal{F}_{\bullet+ k-2}$ yield the data of an $A_{\infty}$-modules where we consider the DGA $\mathbf{R}_{\bullet}$ as an $A_{\infty}$-algebra with trivial higher multiplications.
    The restrictions $\eta_k^{\mathcal{F}}\colon \mathcal{F}_{\bullet}^1\otimes \mathbf{R}_{\bullet}^{k-1}\to \mathcal{F}^2_{\bullet + k-1+m}$ of Definition \ref{def_morphism_of_path_modules} can then be understood as a morphism of $A_{\infty}$-modules. 
\end{remark}

\begin{example}\label{example_associated_bundle_strict_path_module}
    Let $q\colon E\to M$ be a $G$-principal bundle and let $F$ be a $G$-space.
    Denote the associated bundle by $X = F\times_G E\to M$ and let $p\colon F\times E\to X$ be the canonical map.
    Choose a basepoint $e_0\in E$ and let $x_0 = q(e_0)\in M$.
    We set $\mathbf{R}_{\bullet} = \mathrm{C}_{\bullet}(G)$, $\mathbf{E}_{\bullet} = \mathrm{C}_{\bullet}(E)$, $\mathbf{X}_{\bullet} = \mathrm{C}_{\bullet}(X)$ and $\mathcal{F}_{\bullet} = \mathrm{C}_{\bullet}(F)$ where we consider $F\subseteq X$ via the identification $F\xrightarrow[]{\cong} X_{x_0}, f\mapsto [f,e_0]$ and $G\subseteq E$ via the identification $G\xrightarrow[]{\cong } E_{x_0}, g\mapsto g\cdot e_0$.
    Setting $m_2 = p_*\colon \mathcal{F}_{\bullet}\otimes \mathbf{E}_{\bullet}\to \mathbf{X}_{\bullet}$ yields a strict path-module structure.
    Analogously if $F'\subseteq F$ is a $G$-invariant subspace we obtain a strict path-module structure with $\mathbf{X}_{\bullet} = \mathrm{C}_{\bullet}(F\times_G E,F'\times_G E)$ and $\mathcal{F}_{\bullet} = \mathrm{C}_{\bullet}(F,F')$.

    Now, let $K_1$ and $K_2$ be $G$-spaces and assume that $\rho\colon K_1\to K_2$ is a $G$-equivariant map.
    If $q\colon E\to M$ is a $G$-principal bundle then the induced map $\overline{\rho}\colon K_1\times_G E\to K_2\times_G E$ induces a morphism of path-modules.
    In fact, the chain map $\eta_1 = \overline{\rho}_*\colon \mathrm{C}_{\bullet}(K_1\times_G E)\to \mathrm{C}_{\bullet}(K_2\times_G E)$ together with $\eta_k = 0$ for $k\geq 2$ defines a morphism of path-modules between the strict path modules $(\mathrm{C}_{\bullet}(K_1\times_G E), \mathrm{C}_{\bullet}(K_1),m_k^1)$ and $(\mathrm{C}_{\bullet}(K_2\times_G E),\mathrm{C}_{\bullet}(K_2),m_k^2)$.
\end{example}

The following proposition is an adaptation of \cite[Proposition 5.2]{clivio2025goresky}.
\begin{prop}\label{prop_bundle_theorem_and_ainfinity_morphisms}
    Let $q\colon E\to M$ be a $G$-principal bundle and let $F_1$ and $F_2$ be right $G$-spaces.
    Let $\Phi\colon \mathrm{mor}(\mathcal{M}_f)\to G$ be a transport function induced by $q$ and assume that all compatible sections for $\Phi$ exist.
    Consider the strict path modules $(\mathrm{C}_{\bullet}(F_i\times_G E),\mathrm{C}_{\bullet}(F_i))$ over $(\mathrm{C}_{\bullet}(G),\mathrm{C}_{\bullet}(E))$ for $i = 1,2$ as in Example \ref{example_associated_bundle_strict_path_module}.
    If $\eta\colon \mathrm{C}_{\bullet}(F_1\times_G E)\to \mathrm{C}_{\bullet+m}(F_2\times_G E)$ is a morphism of path-modules for a given chain map $\eta_1\colon \mathrm{C}_{\bullet}(F_1\times_G E)\to \mathrm{C}_{\bullet}(F_2\times_G E)$, then there is an induced chain map $\widetilde{\eta}\colon \mathrm{C}_{\bullet}(M;(\Phi,\mathrm{C}_{\bullet}(F_1))) \to \mathrm{C}_{\bullet + m}(M;\Phi,\mathrm{C}_{\bullet}(F_2)))$ and the diagram
    $$
        \begin{tikzcd}
            \mathrm{H}_{\bullet}(M;(\Phi,\mathrm{C}_{\bullet}(F_1))) \arrow[]{r}{t_1} \arrow[]{d}{\widetilde{\eta}} & \mathrm{H}_{\bullet}(F_1\times_G E) \arrow[]{d}{(\eta_1)_*} \\
            \mathrm{H}_{\bullet}(M;(\Phi,\mathrm{C}_{\bullet}(F_2))) \arrow[]{r}{t_2} & \mathrm{H}_{\bullet}(F_2\times_G E) 
        \end{tikzcd}
    $$
    commutes where $t_1,t_2$ are the isomorphisms as in Theorem \ref{theorem_assoc_bundle}.
\end{prop}
\begin{proof}
    Recall that for $i \in \{1,2\}$ the map $t_i\colon \mathrm{C}_{\bullet}(M;(\Phi,\mathrm{C}_{\bullet}(F_i)))\to \mathrm{C}_{\bullet}(F_i\times_G E)$ is given by
    $$    t_i(\alpha\otimes x) = p_* (\alpha\times m_x)    $$
    for $m_x\in \mathrm{C}_{|x|}(E)$ the chain defined by using the compatible section for $x\in\mathrm{Crit}(f)$.
    Suppose that $\eta\colon \mathrm{C}_{\bullet}(F_1\times_G E)\to \mathrm{C}_{\bullet+m}(F_2\times_G E)$ is a morphism of path-modules with a given chain map $\eta_1\colon \mathrm{C}_{\bullet}(F_1\times_G E)\to \mathrm{C}_{\bullet}(F_2\times_G E)$.
    We will define a chain map $\widetilde{\eta}\colon \mathrm{C}_{\bullet}(M;(\Phi,\mathrm{C}_{\bullet}(F_1))) \to \mathrm{C}_{\bullet + m}(M;\Phi,\mathrm{C}_{\bullet}(F_2)))$ using the fact that the restrictions $\eta_k^{\mathcal{F}}$ yield a morphsim of $A_{\infty}$-modules, see Remark \ref{remark_a_infinity}.
    Indeed, Riegel \cite{riegel2024chain} shows that in this case one can obtain a chain map for the DG Morse complexes as follows.
    We write $\mathbf{R}_{\bullet} = \mathrm{C}_{\bullet}(G)$ we denote by $T(\mathbf{R}_{\bullet}) = \oplus_{k\geq 0}\mathbf{R}_{\bullet}^{\otimes k}$ the tensor algebra of $\mathbf{R}_{\bullet}$.
    Let $${\mathbf{m}}'\in \mathrm{Hom}_{-1}(T(\mathbf{R}_{\bullet})\otimes  R\langle \mathrm{Crit}(f)\rangle  ,  T(\mathbf{R}_{\bullet})\otimes R\langle \mathrm{Crit}(f)\rangle) $$
    be the morphism which acts as
    $$   {\mathbf{m}}'( r_1\otimes \ldots \otimes r_k\otimes x) = \sum_{|y| < |x|} r_1\otimes \ldots \otimes r_k\otimes m_{x,y}\otimes y  $$
    where $m_{x,y}\in\mathbf{R}_{\bullet}$ are the chains which constitute the twisting cocycle.
    Riegel shows in \cite[Proposition 5.1]{riegel2024chain} that the map $\widetilde{\eta}\colon \mathrm{C}_{\bullet}(M;(\Phi,\mathrm{C}_{\bullet}(F_1))\to \mathrm{C}_{\bullet}(M;(\Phi,\mathrm{C}_{\bullet}(F_2))$ defined by
    $$    \widetilde{\eta} = \sum_{k\geq 0} (\eta_{k+1}^{\mathcal{F}}\otimes \mathrm{id})\circ (\mathrm{id}\otimes ({\mathbf{m}}')^{k})    $$
    is a chain map where $({\mathbf{m}}')^k$ means applying the morphism $\mathbf{m}'$ iteratively.
    Following \cite{clivio2025goresky} we extend the map $\mathbf{m}'$ to a map $$\widetilde{\mathbf{m}}\colon  \mathrm{C}_{\bullet}(F_1)\otimes T(\mathbf{R}_{\bullet})\otimes R\langle \mathrm{Crit}(f)\rangle \to \mathrm{C}_{\bullet}(F_1)\otimes T(\mathbf{R}_{\bullet})\otimes R\langle \mathrm{Crit}(f)\rangle  $$
    by requiring that $\widetilde{\mathbf{m}}$ act as $\mathrm{id}\otimes \mathrm{id}\otimes \mathbf{m}'$.
    Further define 
    \begin{eqnarray*}
        \widetilde{\mathbf{l}}  \colon \mathrm{C}_{\bullet}(F_1) \otimes T(\mathbf{R}_{\bullet})\otimes R\langle \mathrm{Crit}(f)\rangle  & \to &  \mathrm{C}_{\bullet}(F_1)\otimes T(\mathbf{R}_{\bullet})\otimes \mathrm{C}_{\bullet}(E)    \\
    \alpha \otimes r_1\otimes \ldots \otimes r_k\otimes x & \mapsto &  \alpha\otimes r_1\otimes \ldots \otimes r_k\otimes m_x .   
    \end{eqnarray*}
    We define the map
    $$  v= \sum_{N\geq 1} (-1)^{N+1} \eta_{N+1}\circ \widetilde{\mathbf{l}}\circ \widetilde{\mathbf{m}}^{N+1}  \colon \mathrm{C}_{\bullet}(M;(\Phi,\mathrm{C}_{\bullet}(F_1)))\to \mathrm{C}_{\bullet + 1+m}(F_2\times_G E) .    $$
    The computation of maps on chain level in \cite[Proof of Lemma 5.3]{clivio2025goresky} can be replicated in our situation and we obtain
    $$     v\circ d + (-1)^{m} d\circ v = \eta_1\circ t_1 - t_2 \circ \widetilde{\eta}   .  $$
    It follows that the diagram in the statement of the proposition commutes.
\end{proof}

Using a version of the homotopy transfer theorem, Clivio shows how one can invert morphism of path modules, see \cite[Proposition 5.14]{clivio2025goresky}.
An application to our situation yields the following. 
\begin{prop}\label{prop_invertin_quasi-iso}
    Let $q\colon E\to M$ be a $G$-principal bundle and let $(F_1,F_1'),(F_2,F_2')$ be pairs of $G$-spaces.
    Take homology with field coefficients.
    If there exists a morphism of path modules $\eta_{\bullet}\colon \mathrm{C}_{\bullet}(F_1\times_G E,F_1'\times_G E)\to \mathrm{C}_{\bullet}(F_2\times_G E,F_2'\times_G E)$ such that the chain map $\eta_1\colon \mathrm{C}_{\bullet}(F_1\times_G E,F_1'\times_G E)\to \mathrm{C}_{\bullet}(F_2\times_G E,F_2'\times_G E)$ is a quasi-isomorphism, then there is a morphism of path modules $\xi_{\bullet}\colon \mathrm{C}_{\bullet}(F_2\times_G E)\to \mathrm{C}_{\bullet}(F_1\times_G E)$ such that in homology $(\xi_1)_* = (\eta_1)_*^{-1}$.
\end{prop}

\begin{example}\label{example_excision_path_module-stuff}
    In the specific situation of Theorem \ref{theorem_gysin_map_chain_map} we consider a $G$-space $F_2$ with an invariant subspace $F_1$ that admits a $G$-invariant tubular neighborhood $U$ of rank $k$.
    The inclusion $ U\hookrightarrow F_2$ induces a chain map $\eta_1\colon \mathrm{C}_{\bullet}(U\times_G E, \sim F_1\times_G E)\to \mathrm{C}_{\bullet}(F_2\times_G E,\sim F_1\times_G E)$ which is a quasi-isomorphism by excision.
    Moreover, the map $\eta_1$ is part of a morphism of path modules as in Example \ref{example_associated_bundle_strict_path_module} where the higher maps are trivial.
    By Proposition \ref{prop_invertin_quasi-iso} we therefore obtain a morphism of path modules $\xi_{\bullet}\colon \mathrm{C}_{\bullet}(F_2\times_G E,\sim F_1\times_G E)\to \mathrm{C}_{\bullet}(U\times_G E,\sim F_1\times_G E)$ which in turn induces a chain map $ \widetilde{\xi}\colon \mathrm{C}_{\bullet}(M;(\Phi,\mathrm{C}_{\bullet}(F_2,\sim F_1)))\to \mathrm{C}_{\bullet}(M;(\Phi,\mathrm{C}_{\bullet}(U,\sim F_1)))$ such that in homology $(\xi_1)_* = (\eta_1)_*^{-1}$. 
    By Proposition \ref{prop_bundle_theorem_and_ainfinity_morphisms} it follows that the diagram
    $$
        \begin{tikzcd}
            \mathrm{H}_{\bullet}(M;(\Phi,\mathrm{C}_{\bullet}(F_2,\sim F_1))) \arrow[]{r}{t_2} \arrow[]{d}{\widetilde{\xi}} & \mathrm{H}_{\bullet}(F_2\times_G E,  \sim F_1\times_G E) \arrow[]{d}{\text{excision}} \\
            \mathrm{H}_{\bullet}(M;(\Phi,\mathrm{C}_{\bullet}(U,\sim F_1))) \arrow[]{r}{t_3} & \mathrm{H}_{\bullet}(U\times_G E,\sim F_1\times_G E) 
        \end{tikzcd}
    $$
    commutes.
    We use the map $\widetilde{\xi}$ and the commutativity of the above diagram in the proof of Theorem \ref{theorem_gysin_map_chain_map}.
\end{example}

\bibliography{lit}

\providecommand{\bysame}{\leavevmode\hbox to3em{\hrulefill}\thinspace}
\providecommand{\MR}{\relax\ifhmode\unskip\space\fi MR }
\providecommand{\MRhref}[2]{%
  \href{http://www.ams.org/mathscinet-getitem?mr=#1}{#2}
}
\providecommand{\href}[2]{#2}
\begin{thebibliography}{BDHO25}

\bibitem[AB21]{abouzaid2021arnold}
Mohammed Abouzaid and Andrew~J Blumberg, \emph{Arnold conjecture and {M}orava
  {K}-theory}, arXiv preprint arXiv:2103.01507 (2021).

\bibitem[AB24]{abouzaid2024foundation}
\bysame, \emph{Foundation of {F}loer homotopy theory {I}: Flow categories},
  arXiv preprint arXiv:2404.03193 (2024).

\bibitem[ADE14]{audin2014morse}
Michele Audin, Mihai Damian, and Reinie Ern{\'e}, \emph{Morse theory and
  {F}loer homology}, vol.~2, Springer, 2014.

\bibitem[ADP24]{asplund2024spectral}
Johan Asplund, Yash Deshmukh, and Alex Pieloch, \emph{Spectral equivalence of
  nearby {L}agrangians}, arXiv preprint arXiv:2411.08841 (2024).

\bibitem[BC07]{barraud2007lagrangian}
Jean-Fran{\c{c}}ois Barraud and Octav Cornea, \emph{Lagrangian intersections
  and the {S}erre spectral sequence}, Annals of Mathematics (2007), 657--722.

\bibitem[BDHO24]{barraud2024floer}
Jean-Fran{\c{c}}ois Barraud, Mihai Damian, Vincent Humiliere, and Alexandru
  Oancea, \emph{Floer homology with {D}{G} coefficients. {A}pplications to
  cotangent bundles}, arXiv preprint arXiv:2404.07953 (2024).

\bibitem[BDHO25]{barraud2025morse}
Jean-Fran{\c{c}}ois Barraud, Mihai Damian, Vincent Humili{\`e}re, and Alexandru
  Oancea, \emph{Morse homology with differential graded coefficients},
  Springer, 2025.

\bibitem[Bre72]{bredon1972introduction}
Glen~E Bredon, \emph{Introduction to compact transformation groups}, vol.~46,
  Academic press, 1972.

\bibitem[Bre13]{bredon:2013}
\bysame, \emph{Topology and geometry}, Graduate Texts in Mathematics, vol. 139,
  Springer, 2013.

\bibitem[Bro82]{brown1982cohomology}
Kenneth~S Brown, \emph{Cohomology of groups}, Springer, New York (1982).

\bibitem[CHV06]{cohen2006string}
Ralph~L Cohen, Kathryn Hess, and Alexander~A Voronov, \emph{String topology and
  cyclic homology}, Springer Science \& Business Media, 2006.

\bibitem[CJS95]{cohen1995morse}
Ralph~L Cohen, John~DS Jones, and Graeme~B Segal, \emph{Morse theory and
  classifying spaces}, {P}reprint (1995).

\bibitem[CL26]{calle2026classifying}
Maxine~E Calle and Fangji Liu, \emph{On the classifying space of a {M}orse flow
  category}, arXiv preprint arXiv:2603.23695 (2026).

\bibitem[Cli25]{clivio2025goresky}
Jonathan Clivio, \emph{The {G}oresky-{H}ingston coproduct in {M}orse homology
  with {D}{G} coefficients}, arXiv preprint arXiv:2510.23813 (2025).

\bibitem[CO15]{chataur2015basics}
David Chataur and Alexandru Oancea, \emph{Basics on free loop spaces}, Free
  loop spaces in geometry and topology, European Mathematical
  Society-EMS-Publishing House GmbH, 2015, pp.~21--65.

\bibitem[DK69]{dyer1969some}
Eldon Dyer and Daniel~S Kahn, \emph{Some spectral sequences associated with
  fibrations}, Transactions of the American Mathematical Society \textbf{145}
  (1969), 397--437.

\bibitem[Fou26]{fourel2026morse}
Colin Fourel, \emph{Morse flow categories as exit path categories}, arXiv
  preprint arXiv:2605.27112 (2026).

\bibitem[Hus66]{husemoller1966fibre}
Dale Husemöller, \emph{Fibre bundles}, Springer, New York, 1966.

\bibitem[KN63]{kobayashi:1963}
Shoshichi Kobayashi and Katsumi Nomizu, \emph{Foundations of differential
  geometry. {V}ol {I}}, Interscience Publishers, a division of John Wiley \&
  Sons, New York-London, 1963.

\bibitem[Lat94]{latour1994existence}
Fran{\c{c}}ois Latour, \emph{Existence de l-formes ferm{\'e}es non singulieres
  dans une classe de cohomologie de de {R}ham}, Publications Math{\'e}matiques
  de l'IH{\'E}S \textbf{80} (1994), 135--194.

\bibitem[Lee13]{lee:2013}
John~M. Lee, \emph{Introduction to smooth manifolds}, second ed., Graduate
  Texts in Mathematics, vol. 218, Springer, New York, 2013.

\bibitem[LS14]{lipshitz2014khovanov}
Robert Lipshitz and Sucharit Sarkar, \emph{A {K}hovanov stable homotopy type},
  Journal of the American Mathematical Society \textbf{27} (2014), no.~4,
  983--1042.

\bibitem[MW09]{muller2009equivalences}
Christoph M{\"u}ller and Christoph Wockel, \emph{Equivalences of smooth and
  continuous principal bundles with infinite-dimensional structure group.},
  Advances in Geometry \textbf{9} (2009), no.~4, 605--626.

\bibitem[Nic07]{nicolaescu2007invitation}
Liviu Nicolaescu, \emph{An invitation to {M}orse theory}, Springer, 2007.

\bibitem[Qin10]{qin2010moduli}
Lizhen Qin, \emph{On moduli spaces and {C}{W} structures arising from {M}orse
  theory on {H}ilbert manifolds}, Journal of Topology and Analysis \textbf{2}
  (2010), no.~04, 469--526.

\bibitem[Qin18]{qin2018associativity}
\bysame, \emph{On the associativity of gluing}, Journal of Topology and
  Analysis \textbf{10} (2018), no.~03, 585--604.

\bibitem[Rie24]{riegel2024chain}
Robin Riegel, \emph{A chain-level model for {C}has-{S}ullivan products in
  {M}orse homology with differential graded coefficients}, arXiv preprint
  arXiv:2412.18264 (2024).

\bibitem[Rie25]{riegel2025path}
\bysame, \emph{The path-product in {M}orse homology with differential graded
  coefficients}, arXiv preprint arXiv:2508.09583 (2025).

\bibitem[Ste99]{steenrod1999topology}
Norman~Earl Steenrod, \emph{The topology of fibre bundles}, vol.~14, Princeton
  university press, 1999.

\bibitem[Voi14]{voigt2014diss}
Roland Voigt, \emph{Transport functions and {M}orse {K}-theory}, Leipzig
  University, 2014, PhD Thesis.

\bibitem[Weh12]{wehrheim2012smooth}
Katrin Wehrheim, \emph{Smooth structures on {M}orse trajectory spaces,
  featuring finite ends and associative gluing}, Proceedings of the Freedman
  Fest \textbf{18} (2012), 369--450.

\end{thebibliography}
 \bibliographystyle{amsalpha}
 
\end{document}